\documentclass[12pt]{amsart}
\usepackage{amsmath,amssymb,latexsym,cancel,rotating}
\usepackage{graphicx,amssymb,mathrsfs,amsmath,color,fancyhdr,amsthm}
\usepackage[all]{xy}
\usepackage{verbatim} 
\usepackage[all]{xy}
\usepackage{graphicx}
\usepackage{rotating} 
\usepackage{tikz-cd}

\newtheorem{theorem}{Theorem}[section]
\newtheorem{lemma}[theorem]{Lemma}
\newtheorem{corollary}[theorem]{Corollary}
\newtheorem{definition}[theorem]{Definition}
\newtheorem{proposition}[theorem]{Proposition}

\newtheorem{remark}[theorem]{Remark}

  \def\leq{\leqslant}  \def\geq{\geqslant}

\begin{document}

\title[Notes on Chevalley Groups and Root Category]
{Notes on Chevalley Groups and Root Category \uppercase\expandafter{\romannumeral1}}

\author[Buyan Li]{Buyan Li}
\address{Department of Mathematical Sciences, Tsinghua University, Beijing 100084, P. R. China}
\email{liby21@mails.tsinghua.edu.cn}

\author[Jie Xiao]{Jie Xiao}
\address{School of Mathematical Sciences, Beijing Normal University, Beijing 100875, P. R. China}
\email{jxiao@bnu.edu.cn}

\subjclass[2000]{16G20, 20D06}

\date{\today}

\keywords{Chevalley group, root category}

\bibliographystyle{abbrv}

\maketitle

\begin{abstract}
    Based on the construction of simple Lie algebras via root category and following Chevalley's results, we construct Chevalley groups from the root category.
    Then we prove the Bruhat decomposition and the simplicity of the Chevalley groups, and calculate the orders of finite Chevalley groups.
\end{abstract}

\setcounter{tocdepth}{1}\tableofcontents

\section{Introduction}

In 1955, Chevalley published his famous paper\cite{Chevalley} in which he constructed a class of simple groups over arbitrary fields, especially finite fields.
These groups, now known as Chevalley groups, generalize the classical Lie groups to arbitrary fields and provide a unifying framework for studying finite and infinite groups. 

Chevalley groups are constructed from simple complex Lie algebras, which are characterized by their root systems. 
A Chevalley basis is a specific set of root vectors of the Lie algebra that has integer structure constants, allowing the algebra to be defined over $\mathbb{Z}$, and therefore arbitrary fields.
The Chevalley group is then obtained as the group of automorphisms of the Lie algebra constructed from the Chevalley basis over any fixed field, and is generated by specific unipotent elements corresponding to the roots.
In Carter's book\cite{carter} and Steinberg's lectures\cite{Steinberg}, the construction process of Chevalley groups is thoroughly explained, along with some of their important properties.

Ringel\cite{Ringel1990}\cite{RINGEL1990137} used the representation theory of finite dimensional associative hereditary representation-finite algebra $A$ to obtain a $\mathbb{Z}$-form of the positive part $\mathfrak{n}_+$, where $\mathfrak{g}=\mathfrak{n}_-\oplus \mathfrak{h} \oplus \mathfrak{n}_+$ is a triangular decomposition of the Lie algebra $\mathfrak{g}$, and used the Hall algebra to realize the positive part of the quantum group.
In particular, Ringel proved the existence of Hall polynomials, and defined the structure constants to be the evaluations of Hall polynomials.
In fact, he calculated Hall polynomials for any three indecomposable objects, and listed the results in his paper\cite{RINGEL1990137}.
On the other hand, Happel\cite{Happel}\cite{Happelbook} developed the theory of derived category of mod$A$, and extended the Gabriel theorem by pointing out that the isomorphism classes of indecomposable objects in the root category $\mathcal{R}$ correspond bijectively to the root system of the Lie algebra $\mathfrak{g}$ (not only the positive part).
Combining these, Peng and Xiao\cite{1997ROOT} used the root category to recover the whole Lie algebra $\mathfrak{g}$.
Specifically, the Grothendieck group provides the Cartan subalgebra, the isomorphism classes of indecomposable objects of $\mathcal{R}$ label a Chevalley basis.
They demonstrated the existence of the Hall polynomials in the triangulated category $\mathcal{R}$, and used the evaluations of the Hall polynomials within their framework as structure constants to get the Lie product. 
Using the octahedral axiom of $\mathcal{R}$, they obatined the Jacobi identity and realized a simple Lie algebra.

Chevalley\cite{C2} and Kostant\cite{Kostant} defined a $\mathbb{Z}$-form of the coordinate ring of a connected semisimple group over $\mathbb{C}$.
In \cite{zform}, Lusztig generalized the constructions of Kostant to quantum case. 
Based on the construction of $\dot{\mathbf{U} }$ and its canonical basis, Lusztig used the group scheme to construct a connected reductive group.
In \cite{Lusztig}, Lusztig used the canonical basis of quantum group to obtain a basis of the adjoint representation of a simple Lie algebra $\mathfrak{g}$, such that the action of Chevalley generators of $\mathfrak{g}$ is represented by matrices with natural number entries. 
As an application, he obtained a definition of the Chevalley group over a field $k$ associated to $\mathfrak{g}$, which is simpler compared to Chevalley's original definition.

Based on the results of Lusztig, Geck wrote a survey\cite{gecknotes} about the Chevalley groups and had some new developments, e.g. in \cite{Geck}. 
He noted in \cite{Geck}:

\begin{quote}
    \textit{Ringel \cite{RINGEL1990137} (see also Peng-Xiao \cite{1997ROOT}) found an entirely different method for fixing the signs in the structure constants, starting from any orientation of the Dynkin diagram of $\Phi$ and then using the representation theory of quivers and Hall polynomials.}
\end{quote}

The purpose of this article is just to continue the process starting from Ringel and present a construction of Chevalley groups.
Since Ringel, Peng and Xiao successfully used the representation theory of quivers and Hall polynomials to construct Lie algebras, they also provided a method to determine the signs of the structure constants for Chevalley groups.
Although it's a direct corollary of combining their constructions with Chevalley groups, we believe it is necessary to write down the whole story explicitly.
Our construction arising from root category is self-contained, while we also use methods and calculations of \cite{carter}.

In Section 2, we introduce some basic results about quiver representations and Hall polynomials.
In Section 3, we state the construction of the simple Lie algebra $\mathfrak{g}_{\mathcal{R}}$ arising from the root category $\mathcal{R}$ given by Peng and Xiao.
In Section 4, we construct the Chevalley group of the Lie algebra $\mathfrak{g}_{\mathcal{R}}$, using the properties of the root category.
The structure constants in the Chevalley group are given by evaluations of Hall polynomials.
In particular, we use the BGP reflection functors in root category\cite{BGP} to define some important elements in the Chevalley group, which will correspond to the elements in the Weyl group.
Finally, we prove the Bruhat decomposition and the simplicity of the Chevalley group.
When the field is finite, we calculate the order of the finite Chevalley group via properties of the AR-quiver.

\subsection*{Conventions}

We denote by $|S|$ the number of elements in the set $S$, and by $[M]$ the isomorphism class of object $M$.

\section{Background}

\subsection{Quivers and Representations}

In this subsection, we state some results in the representation theory of quivers.
We refer Chapter2,3 of \cite{Deng} for details.

\begin{definition}
    A quiver $Q=(Q_0,Q_1,t,h)$ consists of set $Q_0$ of vertices, set $Q_1$ of arrows, and two maps $t,h:Q_1\rightarrow Q_0$.
    Each arrow $\alpha$ in $Q_1$, $t\alpha, h\alpha$ are called the tail and head of $\alpha$, respectively.

    A valuation on a quiver $Q$ consists of two families $\{ d_i\}_{i\in Q_0}$ and $\{ m_{\alpha}\}_{\alpha \in Q_1}$ of positive integers, such that for each $\alpha \in Q_1$, $m_{\alpha}$ is a common multiple of $d_{t\alpha}$ and $d_{h\alpha}$.
    A quiver $Q$ together with a valuation $(\{ d_i\}_{i\in Q_0} , \{ m_{\alpha}\}_{\alpha \in Q_1})$ is called a valued quiver. 
\end{definition}

The valued quiver $\Gamma$ with valuation $(\{ d_i\}_{i\in \Gamma_0} , \{ m_{\alpha}\}_{\alpha \in \Gamma_1})$ defines a matrix $C_{\Gamma} = (c_{ij})_{i,j\in \Gamma_0}$, where 
\begin{equation*}
    c_{ij} = 
    \begin{cases}
        2  \qquad  & i=j,\\
        -\sum\limits_{\rho: \text{ arrows between } i \text{ and } j }   \frac{m_{\rho}}{d_i}            & i\neq j.
    \end{cases}
\end{equation*}

\begin{lemma} \cite[Lemma 3.5]{Deng} 
    Let $\Gamma$ be a valued quiver, $C_{\Gamma}$ be the matrix defined as above.
    Then $C_{\Gamma}$ is a symmetrizable Cartan matrix.
    Conversely, every symmetrizable Cartan matrix can be obtained in this way.
\end{lemma}

Now let $(Q,\sigma)$ be a quiver with an automorphism $\sigma$.
For the pair $(Q,\sigma)$, we can naturally define a valued quiver $\Gamma(Q,\sigma)$.

\begin{definition}
    The valued quiver $\Gamma=\Gamma(Q,\sigma)$ associated with a quiver $Q$ with an automorphism $\sigma$ is defined as follows.
   
    (1) The vertex set $\Gamma_0$ is the set of $\sigma$-orbits in $Q_0$, and the arrow set $\Gamma_1$ is the set of $\sigma$-orbits in $Q_1$.
     For each $\alpha \in \Gamma_1$, its tail (resp. head) is the $\sigma$-orbit of tails (resp. heads) of arrows in $\alpha$.
    
    (2) The valuation of $\Gamma$ consists of families of positive integers $(\{ d_i\}_{i\in \Gamma_0} , \{ m_{\alpha}\}_{\alpha \in \Gamma_1})$.
    For each $i\in \Gamma_0$, let $d_i = | \{ \text{vertices in the } \sigma \text{-orbit } i  \} |$.
    For each $\alpha \in \Gamma_1$, let $m_{\alpha} = | \{ \text{arrows in the } \sigma \text{-orbit } \alpha \} |$.
 \end{definition}

\begin{definition} 
    With the notations as above and let $k$ be a field. 

    A $k$-modulation $\mathbb{M}$ of $\Gamma$ consists of a family of division $k$-algebras $\mathcal{D}_i$ with $dim_k \mathcal{D}_i = d_i$, for each $i\in \Gamma_0$, and of a family of $\mathcal{D}_{ h \alpha}$-$\mathcal{D}_{ t \alpha}$-bimodules $M_{\alpha}$ with $dim_k M_{\alpha} = m_{\alpha}$, for each $\alpha\in \Gamma_1$, such that 
    the action of $k$ on $M_{\alpha}$ via the left action of $\mathcal{D}_{h\alpha}$ and via the right action of $\mathcal{D}_{t\alpha}$ coincide.

    A $k$-modulated quiver $\mathfrak{Q}$ is a pair $(\Gamma,\mathbb{M})$ consisting of a valued quiver $\Gamma$ and a $k$-modulation $\mathbb{M}$ of $\Gamma$.
\end{definition}

Now we introduce the definition of the Auslander-Reiten quiver (AR-quiver) of a finite dimensional algebra $A$ over a field $k$.

Let mod$A$ be the category of finite dimensional left $A$-modules.
For each module $M$ in mod$A$, let 
\begin{align*}
    \mathcal{D}_M = \operatorname{End}_A(M)/\operatorname{rad} (\operatorname{End}_A(M)), 
\end{align*}
which is a division algebra if $M$ is indecomposable.

For any two $A$-modules $M=\oplus_i M_i$ and $N=\oplus_j N_j$, where $M_i$ and $N_j$ are indecomposable, define the radical of $\operatorname{Hom}_A(M,N)$ by 
\begin{align*}
    \operatorname{rad}_A(M,N) = \{ f=(f_{ij}):M\rightarrow N | f_{ij}:M_i\rightarrow N_j  \text{ is not an isomorphism for all } i,j \}.
\end{align*}
This defines a subfunctor $\operatorname{rad}_A(-,-):(\operatorname{mod}A)^{op}\times \operatorname{mod}A \rightarrow \operatorname{mod}k$ of $\operatorname{Hom}_A(-,-)$.
If $M,N$ are indecomposable, the $\mathcal{D}_N$-$\mathcal{D}_M$-bimodule 
\begin{align*}
    \operatorname{Irr}_A(M,N) = \operatorname{rad}_A(M,N)/\operatorname{rad}_A^2(M,N)
\end{align*}
 is called the space of irreducible morphisms from $M$ to $N$.

\begin{definition}
    With the notations as above,
    the Auslander-Reiten quiver (AR-quiver) of a finite dimensional algebra $A$ over a field $k$, is a $k$-modulated quiver $\mathfrak{Q}_A=(\Gamma,\mathbb{M})$.
    The vertices of $\Gamma$ are isomorphism classes $[M]$ of indecomposable $A$-modules.
    There is a unique arrow from $[M]$ to $[N]$ for indecomposable modules $M,N$, if and only if $\operatorname{Irr}_A(M,N)\neq 0$.
    The valuation of $\Gamma$ assigns to each vertex $[M]$ the positive integers $d_M = dim_k \mathcal{D}_M$, and each arrow $[M] \xrightarrow{\alpha} [N]$ the positive integer $m_{\alpha} = dim_k \operatorname{Irr}_A(M,N)$.
    The $k$-modulation $\mathbb{M}$ is given by division algebras $\mathcal{D}_M$, for each vertex $[M]$, and $\mathcal{D}_N$-$\mathcal{D}_M$-bimodules $\operatorname{Irr}_A(M,N)$, for each arrow $[M]\rightarrow [N]$.
\end{definition}

Now we consider the case when $k=\mathbb{F}_q$ is a finite field with $q$ elements.
Let $K=\overline{\mathbb{F}}_q$ be the algebraic closure of $\mathbb{F}_q$ and $A$ be a finite dimensional $K$-algebra. 
A Frobenius map $F$ of a $K$-vector space $V$ is an $\mathbb{F}_q$-linear isomorphism of $V$, satisfying that for all $v\in V,\lambda \in K$, $F(\lambda v)=\lambda^q F(v)$, and $F^t(v)=v$ for some $t\geq 1$. 
A map $F:A\rightarrow A$ is a Frobenius morphism on $A$ if it's a Frobenius map on the underlying $K$-vector space and compatible with the multiplication of $A$. 

Suppose $\phi :K\rightarrow K$ is the field automorphism given by $\phi(\lambda) = \lambda^q$.
For any $K$-vector space $V$, let $V^{(1)} = V\otimes_{\phi} K$, which means for each $v\in V$ and $\lambda \in K$, we have $\lambda v \otimes 1 = v\otimes \lambda^q$.
Let $v^{(1)} = v\otimes 1$ and $\varpi_V:V\rightarrow V^{(1)}$ be the $\mathbb{F}_q$-linear isomorphism maps $v$ to $v^{(1)}$.

Let $M$ be an $A$-module and $\pi : A \rightarrow \operatorname{End}_K(M)$ be the $K$-algebra homomorphism describing this module structure.
Then we have a $K$-algebra homomorphism $\pi^{(1)}:A^{(1)}\rightarrow (\operatorname{End}_K(M))^{(1)}$.
The composition of maps
\begin{align*}
    A\xrightarrow{F^{-1}} A \xrightarrow{\varpi_A} A^{(1)} \xrightarrow{\pi^{(1)}} (\operatorname{End}_K(M))^{(1)} \xrightarrow{\cong} \operatorname{End}_K(M^{(1)}) 
\end{align*}
defines an $A$-module structure on $M^{(1)}$.
 Denote this module by $M^{[1]}$ and it's called the Frobenius twist of $M$.
If $f:M\rightarrow N$ is an $A$-module homomorphism, then $f^{[1]} = f \otimes 1 : M^{[1]} \rightarrow N^{[1]}$ is an $A$-module homomorphism.
We obtain a functor $(-)^{[1]}: \operatorname{mod}A \rightarrow \operatorname{mod}A$, which is called the Frobenius twist functor.


Since $A$ is finite dimensional, any finite dimensional $A$-module $M$ is F-periodic, and let $p_F(M)$ be the F-period of $M$.
The Frobenius twist functor induces an automorphism $\varsigma $ of the AR-quiver $ \mathfrak{Q}_A$ of $A$.
We describe $\varsigma $ as follows.
For each vertex $[M]$ in $\mathfrak{Q}_A$, define $\varsigma([M]) = [M^{[1]}]$.
Let $M,N$ be indecomposable $A$-modules.
For $0 \leq s < p_F(M)$ and $0\leq t < p_F(N)$, 
  there are $n_{s,t}$ arrows $\xi_{s,t}^{(m)} $ from $[M^{[s]}]$ to $[N^{[t]}]$ in $\mathfrak{Q}_A$, where $n_{s,t} = dim_K \operatorname{Irr}_A(M^{[s]},N^{[t]})$ and $1\leq m \leq n_{s,t}$.
Note that $n_{s,t} = n_{s+1,t+1}$ for all $s,t$ (modulo $p_F(M),p_F(N)$, respectively).
Define $\varsigma (\xi_{s,t}^{(m)}) = \xi_{s+1,t+1}^{(m)}$ for $0 \leq s < p_F(M)$ and $0\leq t < p_F(N)$.

For any quiver with automorphism $(Q,\sigma)$, we can attach an $\mathbb{F}_q$-modulated quiver $\mathfrak{Q}_{Q,\sigma,q}$ to it.
Let $KQ$ be the path algebra of $Q$ over $K$.
There is a Frobenius morphism $F=F_{Q,\sigma,q}$ of $KQ$, such that $F(\sum_{s} x_s p_s) = \sum_s x_s^q \sigma (p_s)$, where $x_s \in K,p_s$ are paths in $KQ$.
We call the $F$-fixed point algebra $(KQ)^F = \{  a\in KQ | F(a) = a\}$ the $\mathbb{F}_q$-algebra associated with $(Q,\sigma)$.

Let $\Gamma=\Gamma(Q,\sigma)$ be the valued quiver,
we define an $\mathbb{F}_q$-modulated quiver $(\Gamma,\mathbb{M})$ as follows.
For each vertex $i\in \Gamma_0$, let $N_i = \oplus_{a\in i} Ke_a$, where $e_a$ is the idempotent of $KQ$ corresponding to vertex $a\in Q_0$.
For each arrow $\alpha \in \Gamma_1$, let $N_{\alpha} = \oplus_{\tau\in \alpha} K\tau$, and the $N_{h\alpha}$-$N_{t\alpha}$-bimodule structure of $N_{\alpha}$ is given by the multiplication in $KQ$.
They are F-stable subspaces of $KQ$.
For each $i\in \Gamma_0$, let $N_i^F$ be the set of $F$-fixed points in $N_i$. 
Then the set $N_{\alpha}^F$ of $F$-fixed points in $N_{\alpha}$ is an $N_{h\alpha}^F$-$N_{t\alpha}^F$-bimodule.
The $\mathbb{F}_q$-modulation $\mathbb{M}$ of the valued quiver $\Gamma$ is defined to be $\mathbb{M} = \mathbb{M}(Q,\sigma,q) = ( \{ N_i^F \}_{i\in \Gamma_0}, \{  N_{\alpha}^F \}_{\alpha \in \Gamma_1} )$.
We denote this $\mathbb{F}_q$-modulated quiver $(\Gamma,\mathbb{M})$ by $\mathfrak{Q}_{Q,\sigma,q}$.

As a special case of the previous constructions, for $(\mathfrak{Q}_A,\varsigma)$ we obtain an $\mathbb{F}_q$-modulated quiver $\mathfrak{Q}_{\mathfrak{Q}_A,\varsigma,q}$. 
Now we can state the following theorem.

\begin{theorem} \label{deng} \cite[Thm 3.30]{Deng}
    Let $A$ be a finite dimensional algebra over $K=\overline{\mathbb{F}}_q$ with Frobenius morphism $F$. 
    Let $\mathcal{Q} = \mathfrak{Q}_A$ be the AR-quiver of $A$, and $\varsigma$ be the automorphism induced by $F$.
    Then the $\mathbb{F}_q$-modulated quiver $\mathfrak{Q}_{\mathcal{Q},\varsigma,q } $ associated with $(\mathcal{Q},\varsigma)$ is equivalent to the AR-quiver $\mathfrak{Q}_{A^F}$ of $A^F$.
\end{theorem}

\subsection{Hall Polynomials}

Let $k$ be a finite field with $q$ elements and $A$ be a finite dimensional $k$-algebra.
We always assume that $A$ is representation-finite and hereditary.
Ringel defined Hall polynomials in \cite{Ringel1990}\cite{RingelLie} for mod$A$, whose evaluations are structure constants of some Hall algebra.

Let $X,Y,M$ be finite dimensional $A$-modules and $F_{XY}^M$ be the number of filtrations 
   $ 0 \subseteq U \subseteq M$
of $M$ such that $M/U \cong X$ and $U\cong Y$.
For any field extension $E$ of $k$, we denote $(-)^E = - \otimes_k E$.
Then $M^E$ is well-defined for indecomposable object $M\in \operatorname{mod}A$, and for any $N\in \operatorname{mod}A$, we let $N^E=N_1^E\oplus \cdots \oplus N_t^E$, where $N$ decompose into direct sum of indecomposables $N_i$.
The extension $E$ of $k$ is conservative for $A$, if for any indecomposable object $M$ in mod$A$, the algebra $(\operatorname{End}(M)/\operatorname{rad}(\operatorname{End}(M)))^E$ is a field.
Ringel showed that for any finite dimensional $A$-modules $X,Y,M$, there exists a polynomial $\varphi_{XY}^M(T) \in \mathbb{Z}[T]$, such that for any conservative field extension $E$ of $k$, we have 
\begin{align*}
    \varphi_{XY}^M(|E|) = F_{X^E,Y^E}^{M^E},
\end{align*}
where $F_{X^E,Y^E}^{M^E}$ is similarly defined as the number of filtrations of $A^E$-modules.
The polynomials $\varphi_{XY}^M$ are called Hall polynomials.

Peng and Xiao\cite{1997ROOT} extended this definition to the framework of bounded derived category $D^b(A)$ of mod$A$.

Assume $\mathcal{C}$ is a triangulated category with translation functor $T$.
For any objects $X,Y,Z \in \mathcal{C}$, define a set
\begin{align*}
    W(X,Y;Z)=& \{(f,g,h)\in \operatorname{Hom}_{\mathcal{C}}(X,Z) \times \operatorname{Hom}_{\mathcal{C}}(Z,Y) \times \operatorname{Hom}_{\mathcal{C}}(Y,TX)| \\
    &  X \xrightarrow{f} Z \xrightarrow{g} Y \xrightarrow{h} TX \text{ is a triangle}\}, 
\end{align*}
and an action of $Aut(X)\times Aut(Y)$ on this set by 
\begin{align*}
    (\alpha,\beta)(f,g,h) = (\alpha f,g \beta, \beta^{-1}h(T\alpha)^{-1}),
\end{align*}
for any $(\alpha,\beta)\in Aut(X)\times Aut(Y),(f,g,h)\in W(X,Y;Z)$.
Let $[(f,g,h)]$ be the orbit of $(f,g,h)\in W(X,Y;Z)$ under this action, and put 
\begin{align*}
    V(X,Y;Z) = \{ [(f,g,h)] | (f,g,h)\in W(X,Y;Z)   \} = W(X,Y;Z)/(Aut(X)\times Aut(Y)).
\end{align*}

Notice that $D^b(A)$ is a triangulated category with translation functor $T$.
Then similar to the statement in the case of mod$A$, for objects $X,Y,M\in D^b(A)$, if there exists an polynomial $\varphi_{XY}^M(T)\in \mathbb{Z}[T]$ such that 
\begin{align*}
    \varphi_{XY}^M(|E|) = |V(X^E,Y^E;M^E)|
\end{align*}
holds for any conservative extension $E$ of $k$, we say that $D^b(A)$ has the Hall polynomial $\varphi_{XY}^M$.
Note that $V(X^E,Y^E;M^E)$ is considered in the full subcategory $D^b(A)^E$ of the category $D^b(A^E)$, whose objects are of the form $L^E$ with $L\in D^b(A)$.

Let $\mathcal{R}$ be the root category, i.e. the orbit category $D^b(A)/T^2$. 
By \cite[Prop 7.1]{1997ROOT}, $\mathcal{R}$ is a triangulated category with translation functor $T$.
Like $D^b(A)$, we can also define Hall polynomials for the root category $\mathcal{R}$.  
The following proposition demonstrates the existence of Hall polynomials.

\begin{proposition}\cite[Proposition 3.1, 3.2]{1997ROOT}
     $D^b(A)$(or $\mathcal{R}$) has the polynomials $\varphi_{XY}^M$ and $\varphi_{YX}^M$ for any objects $X,Y,M\in D^b(A)$(or $X,Y,M\in \mathcal{R}$) with $Y$ indecomposable.
\end{proposition}

The following lemma shows that the definition given by Peng and Xiao indeed extend the definition given by Ringel.
\begin{lemma}\cite[Lemma 3.1]{1997ROOT}
    For objects $X,Y,M\in \mathcal{R}$ with $X,Y$ indecomposable, if there exists a hereditary subcategory containing $X,Y$ and $M$, then the Hall polynomial $\varphi_{XY}^M$ defined for $\mathcal{R}$ coincide with that defined by Ringel. 
\end{lemma}

Note that if there doesn't exist any hereditary subcategory containing $X,Y$ and $M$, we have $\varphi^M_{XY}\equiv 0$.

\section{Root Category and Simple Lie Algebras}
Let $A$ be a finite dimensional associative representation-finite hereditary algebra over some base field $k$, and mod$A$ be the category of finite dimensional $A$ modules.
Let $D^b(A)$ be the bounded derived category of mod$A$, which is a triangulated category with translation functor $T$.
Let $\mathcal{R}$ be the root category, i.e. the orbit category $D^b(A)/T^2$. 
We have mentioned that $\mathcal{R}$ is also a triangulated category.
Let $\operatorname{ind} \mathcal{R}$ be the set of all the indecomposable objects in $\mathcal{R}$.
For any object $X$ of $\mathcal{R}$, we denote its isomorphism class by $[X]$. 

Given objects $X,Y,Z \in \mathcal{R}$, we have defined a set $W(X,Y;Z)$ in the previous section.
For any objects $X,Y,Z,L,M \in \mathcal{R}$, we define an action of $Aut L$ on $W(X,Y;L)\times W(Z,L;M)$ by 
\begin{align*}
\alpha ((f,g,h),(l,m,k))=((f\alpha,\alpha^{-1}g,h),(l,m\alpha,\alpha^{-1}k)),
\end{align*}
for any $\alpha \in AutL$,
and denote the orbit space by $W(X,Y;L)\times W(Z,L;M)/Aut L$, and the orbit of  $((f,g,h),(l,m,k))$ by $\overline{((f,g,h),(l,m,k))}$.

The orbit space $W(Z,X;L)\times W(L,Y;M)/Aut L$ is similarly defined.

We have the following proposition, which can be deduced from the octahedral axiom and diagram chasing.

\begin{proposition}\cite[Prop 2.1]{1997ROOT}\label{bmt}
    Given $X,Y,Z,M \in \operatorname{ind} \mathcal{R}$ and assume that $Y \ncong TX, Y \ncong TZ, X \ncong TZ$. 
    Then we have a bijection
    \begin{align*}
        \bigcup\limits_{[L],L \in \mathcal{R} } W(X,Y;L)\times W(Z,L;M)/Aut L &\rightarrow \bigcup\limits_{[L'],L' \in \mathcal{R}} W(Z,X;L')\times W(L',Y;M)/Aut L'. \\
            \overline{((f,g,h),(l,m,k))} &\mapsto \overline{((l',m',k'),(f',g',h'))}
    \end{align*}
     The morphisms are given by the following commutative diagram, whth all rows and columns being triangles.
\[
\xymatrix{
                                          & Z \ar[d]_{l'} \ar@2{-}[r]  & Z \ar[d]_{l}             &            \\
   T^{-1}Y \ar[r]^{-T^{-1}h'} \ar@2{-}[d] & L' \ar[r]^{f'} \ar[d]_{m'} & M \ar[r]^{g'} \ar[d]_{m} & Y \ar@2{-}[d] \\
   T^{-1}Y \ar[r]^{-T^{-1}h}              & X \ar[r]^{f} \ar[d]_{k'}   & L \ar[r]^{g} \ar[d]_{k}  & Y          \\
                                          & TZ \ar@2{-}[r]             & TZ                       &        
}
\]

\end{proposition}

Now we introduce the Lie algebra arising from the root category $\mathcal{R}$ following \cite{1997ROOT}.

   Let $K$ be the Grothendieck group of the root category $\mathcal{R}$, i.e. the quotient of a free abelian group generated by $\{ H_{[M]}| M \in \mathcal{R}\}$, and subject to the relations $H_{[X]}+H_{[Z]}=H_{[Y]}$, if there exists a triangle $X \rightarrow Y \rightarrow Z \rightarrow TX$.
Denote the image of $H_{[M]}$ in $K$ still by $H_{[M]}$.
Let $\mathcal{N}$ be a free abelian group with a basis $\{ u_{[X]} \}_{X \in \operatorname{ind}\mathcal{R}}$. 
By abuse of notations, we denote $H_M$ instead of $H_{[M]}$ and $u_X$ instead of $u_{[X]}$.

Define the symmetric Euler form of $\mathcal{R}$ by $( - | - ):K\times K \rightarrow \mathbb{Z}$.
For $X,Y \in \mathcal{R}$, 
\begin{align*}
(H_X|H_Y)=&dim_k \operatorname{Hom}_{\mathcal{R}}(X,Y)-dim_k \operatorname{Hom}_{\mathcal{R}}(X,TY)\\
        &+dim_k \operatorname{Hom}_{\mathcal{R}}(Y,X)-dim_k \operatorname{Hom}_{\mathcal{R}}(Y,TX).
\end{align*}
Note that $(H_X|H_Y)=(H_Y|H_X)$ and $(H_X|H_Y)=-(H_X|H_{TY})$. 

Define $d(X)=dim_k \operatorname{End}_{\mathcal{R}}X$ for any $X\in \operatorname{ind}\mathcal{R}$. 
Then $(H_X|H_X)=2d(X)$.

\begin{definition}\label{PXg}
    Let $\mathfrak{g}=(K\oplus \mathcal{N}) \otimes_{\mathbb{Z}} \mathbb{Q}$. Define a bilinear operation $[-,-]$ on $\mathfrak{g}$ as follows:

(1) For any $X,Y \in \operatorname{ind} \mathcal{R}$,
\begin{equation*}
  [u_X,u_Y]=
    \begin{cases}
        \sum\limits_{[L],L \in \operatorname{ind} \mathcal{R}} \gamma^L_{XY} u_L  \qquad &\text{if} \quad Y\ncong TX ,\\
        \frac{H_X}{d(X)}  &\text{otherwise},
    \end{cases}
\end{equation*}
where $\gamma^L_{XY}=\varphi^L_{XY}(1)-\varphi^L_{YX}(1)$ is given by evaluations of the Hall polynomials.

(2)For any $X \in \mathcal{R}$ and $Y \in \operatorname{ind}\mathcal{R}$,
\begin{align*}
    [H_X,u_Y]=-(H_X|H_Y)u_Y , \\
    [u_Y,H_X]=-[H_X,u_Y].
\end{align*}

(3)For any $X,Y \in \mathcal{R}$,
\begin{align*}
    [H_X,H_Y]=0.
\end{align*}

Extend $[-,-]$ bilinearly to the whole $\mathfrak{g}$.
\end{definition}

\begin{theorem}\cite[Thm 4.1]{1997ROOT}
The vector space $\mathfrak{g}$ together with $[-,-]$ is a Lie algebra.
\end{theorem}

A complete section of $\mathcal{R}$ is a connected full subquiver of the AR-quiver of $\mathcal{R}$, which contains one vertex for each $\tau$-orbit.
Choosing an arbitrary complete section of $\mathcal{R}$, we get a hereditary subcategory $\mathcal{B}$ of $\mathcal{R}$. 
Denote the corresponding root system by $\Phi (\mathcal{R})$, whose Dynkin graph is the graph of this complete section. (Since the root system is independent of the complete section chosen, the notation is reasonable.)
As pointed out in \cite{Happel}, the map $\underline{dim}_{\mathcal{B}}$ induces a bijection between the isomorphism class of indecomposables of $\mathcal{R}$ and the root system $\Phi(\mathcal{R})$.
According to $\mathcal{B}$, $\Phi (\mathcal{R})=\Phi=\Phi^+(\mathcal{B}) \cup \Phi^-(\mathcal{B})$ splits into disjoint union of positive roots and negative roots.
Precisely, for $M\in \operatorname{ind}\mathcal{R}$, $\underline{dim}_{\mathcal{B}} M \in \Phi^+(\mathcal{B}) $ if $M \in \operatorname{ind} \mathcal{B}$, and $\underline{dim}_{\mathcal{B}} M \in \Phi^-(\mathcal{B}) $ if $M \in \operatorname{ind} T\mathcal{B}$.   

Let $\mathfrak{g}(\Phi(\mathcal{R}))$ be the simple Lie algebra corresponding to the root system $\Phi(\mathcal{R})$.

\begin{theorem}\cite[Thm 4.2]{1997ROOT} \label{PXgiso}
 We have an isomorphism $\mathfrak{g} \otimes _{\mathbb{Q}} \mathbb{C} \cong \mathfrak{g}(\Phi(\mathcal{R}))$.
\end{theorem}

Notice that $\mathfrak{g}$ is not a $\mathbb{Z}$-form of $\mathfrak{g}(\Phi(\mathcal{R}))$, since the coefficient $\frac{1}{d(X)}$ appears in Def \ref{PXg}(1) is not necessarily in $\mathbb{Z}$.

To fix this, we define $H_{[X]}'=\frac{H_{[X]}}{d(X)}$. 
Let $K'$ be the quotient of a free abelian group generated by $\{ H_{[M]}'|M\in \mathcal{R}\}$, and subject to the relations  $d(X)H_{[X]}'+d(Z)H_{[Z]}'=d(Y)H_{[Y]}'$, if there exists a triangle $X \rightarrow Y \rightarrow Z \rightarrow TX$.
Again, by abuse of notations, we denote $H_{[M]}'$ by $H_M'$.

The operation $[-,-]$ on $\{ H_X' \}_{X\in \mathcal{R}} \cup \{ u_Y \}_{Y\in \operatorname{ind}\mathcal{R}}$ is summarize as follows:

(1) For any $X,Y \in \operatorname{ind} \mathcal{R}$,
\begin{equation*}
  [u_X,u_Y]=
    \begin{cases}
        \sum\limits_{[L],L \in \operatorname{ind} \mathcal{R}} \gamma^L_{XY} u_L  \qquad &\text{if} \quad Y\ncong TX, \\
        H_X'  &\text{otherwise}.
    \end{cases}
\end{equation*}

(2)For any $X \in \mathcal{R}$ and $Y \in \operatorname{ind} \mathcal{R}$,
\begin{align*}
    [H_X',u_Y]=-\frac{(H_X|H_Y)}{d(X)}u_Y ,\\
    [u_Y,H_X']=-[H_X',u_Y].
\end{align*}

(3)For any $X,Y \in \mathcal{R}$,
\begin{align*}
    [H_X',H_Y']=0.
\end{align*}

It's well-known that the Killing form $(-,-)$ of the Lie algebra $\mathfrak{g}$ coincide with the Euler form $(-|-)$ on the Grothendieck group $K$.
Hence
\begin{align*}
    \frac{(H_X|H_Y)}{d(X)} = \frac{2(H_X|H_Y)}{(H_X|H_X)} = \frac{2(H_X,H_Y)}{(H_X,H_X)} \in \mathbb{Z}.
\end{align*}
We denote this integer by $A_{XY}$. 

Let $\mathfrak{g}_{\mathbb{Z}}=K' \oplus \mathcal{N}$.
Then $\mathfrak{g}_{\mathbb{Z}}$ is the $\mathbb{Z}$-span of the set $\mathbb{S}=\{ H_X' \}_{X\in \mathcal{R}} \cup \{ u_Y \}_{Y\in \operatorname{ind}\mathcal{R}}$,
and $\mathfrak{g}_{\mathbb{Z}}$ is a $\mathbb{Z}$-form of $\mathfrak{g}(\Phi(\mathcal{R}))$.

\section{Construction of Chevalley Groups}

In this section, we construct the Chevalley group arising from the root category.
We follow the construction of Chevalley, and refer the reader to Carter's book\cite{carter} for details.

\subsection{Definitions and Basic Properties}
\subsubsection{Generators and the Commutator Formula}

With notations as the previous section, 
we first define $E_{[X]}(t)=exp(t \mathbf{ad} u_X)$ for $X \in \operatorname{ind}\mathcal{R}$ and $t \in \mathbb{C}$. 
We can calculate the action of $E_{[X]}(t)$ on $\mathbb{S}$ explicitly:
\begin{align}
    &E_{[X]}(t)u_X=u_X .\\
    &E_{[X]}(t)u_{TX}=u_{TX}+tH_X'+t^2u_X .\\
    &E_{[X]}(t)H_Y'=H_Y'+A_{YX}tu_X   \qquad \text{for any }  Y \in \mathcal{R} .
\end{align}
For $Y\in \operatorname{ind}\mathcal{R}$ such that $Y \ncong X$ and $Y \ncong TX$, we have 
\begin{align*}
    E_{[X]}(t)u_Y=u_Y+t\sum\limits_{[L_1]} \gamma^{L_1}_{XY} u_{L_1} + \frac{t^2}{2!}\sum\limits_{[L_1],[L_2]} \gamma^{L_1}_{XY} \gamma^{L_2}_{X,L_1} u_{L_2} + \cdots.
\end{align*}
The sum on the right hand side is finite, since repeatedly doing extension by $X$ for enough times, the extensions will finally be decomposable.
Given $X,Y \in \operatorname{ind}\mathcal{R}$, we may use the AR-quiver to see that there is at most one isomorphism class $[L]$ of indecomposable objects in $\mathcal{R}$ such that $\gamma^L_{XY} \neq 0$, taken any representative $L$ of $[L]$.
Thus we can get rid of $\sum$ in the rest of the artice.
Also notice that $\gamma^L_{XY}$ depends on the isomorphism classes $[X],[Y]$ and $[L]$, instead of the representatives $X,Y$ and $L$.
For convenience, we arbitrarily fix a representative for each isomorphism class.

For $X,Y\in \operatorname{ind}\mathcal{R}$ such that $X\ncong Y$ and $X \ncong TY$, integers $i,j\geq 0$, we inductively define isomorphism classes $[L_{X,Y,i,j}]$ (and their representatives $L_{X,Y,i,j}$) as follows.
Let $[L_{X,Y,1,0}]=[X]$ and $[L_{X,Y,0,1}]=[Y]$. 
If $[L_{X,Y,i,j}]$ is defined, and $L_{X,Y,i,j}$ and $X$ have indecomposable extension, then let $[L_{X,Y,i+1,j}]$ be the isomorphism class of this extension. 
Otherwise, we say $[L_{X,Y,i+1,j}]$ doesn't exist.
Similarly, if  $L_{X,Y,i,j}$ and $Y$ have indecomposable extension, then let $[L_{X,Y,i,j+1}]$ be the isomorphism class of this extension.
Otherwise, we say $[L_{X,Y,i,j+1}]$ doesn't exist.

For convenience, we stipulate that
if $[L_{X,Y,i,j}]$ doesn't exist but appears in some $\gamma_{MN}^L$, then $\gamma_{MN}^L = 0$.

We consider the possible extensions (or equivalently, triangles) in $\mathcal{R}$.
For $\mathcal{R}$ of type ADE, and for any $X,Y\in \mathcal{R}$, at most $[L_{X,Y,1,1}]$ exists.
For $\mathcal{R}$ of type BC or $F_4,G_2$, as an application of Thm \ref{deng}, we can obtain their AR-quiver from that of type ADE.
Moreover, since the valuation on the arrows are less than 4, we deduce that the $i,j$ such that $[L_{X,Y,i,j}]$ exists are both less than 4. 
Clearly $[L_{X,Y,2,2}]$ and $[L_{X,Y,3,3}]$ don't exist.

For example, we draw the AR-quiver for $\mathcal{R}$ of type $E_6$.
The left hand side coincide with the right hand side in the following quiver, i.e. we glue $i$ with $i'$ via shift functor $T^2$. 

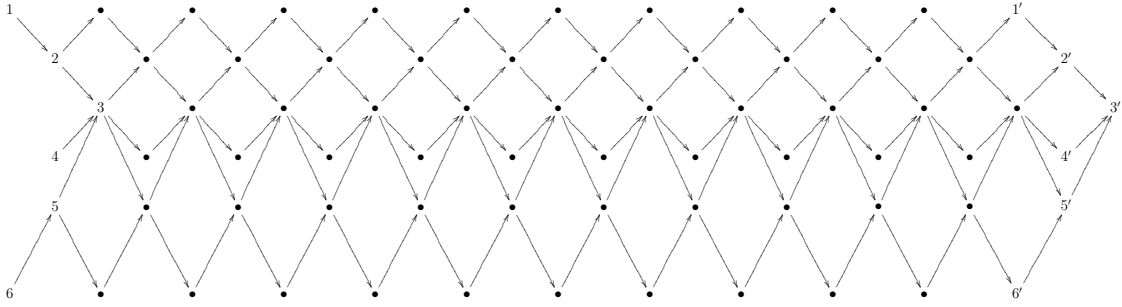
\[
    \resizebox{15cm}{!}{
\xymatrix{
 1 \ar[rd]  &                           & \bullet \ar[rd]                     &                           & \bullet \ar[rd]                     &                           & \bullet \ar[rd]                     &                           & \bullet \ar[rd]                     &                           & \bullet \ar[rd]                     &                           & \bullet \ar[rd]                     &                           & \bullet \ar[rd]                     &                           & \bullet \ar[rd]                     &                           & \bullet \ar[rd]                      &                           & \bullet \ar[rd]                      &                           & 1' \ar[rd]                     &                  &   \\
                  & 2 \ar[rd] \ar[ru]   &                                     & \bullet \ar[ru] \ar[rd]   &                                     & \bullet \ar[ru] \ar[rd]   &                                     & \bullet \ar[ru] \ar[rd]   &                                     & \bullet \ar[ru] \ar[rd]   &                                     & \bullet \ar[ru] \ar[rd]   &                                     & \bullet \ar[ru] \ar[rd]   &                                     & \bullet \ar[ru] \ar[rd]   &                                     & \bullet \ar[ru] \ar[rd]   &                                      & \bullet \ar[ru] \ar[rd]   &                                      & \bullet \ar[ru] \ar[rd]   &                                     & 2' \ar[rd]  &   \\
                  &                           & 3 \ar[ru] \ar[rd] \ar[rdd]    &                          & \bullet \ar[ru] \ar[rd] \ar[rdd]    &                           & \bullet \ar[ru] \ar[rd] \ar[rdd]    &                           & \bullet \ar[ru] \ar[rd] \ar[rdd]    &                           & \bullet \ar[ru] \ar[rd] \ar[rdd]    &                           & \bullet \ar[ru] \ar[rd] \ar[rdd]    &                           & \bullet \ar[ru] \ar[rd] \ar[rdd]    &                            & \bullet \ar[ru] \ar[rd] \ar[rdd]    &                           & \bullet \ar[ru] \ar[rd] \ar[rdd]     &                           & \bullet \ar[ru] \ar[rd] \ar[rdd]     &                           & \bullet \ar[ru] \ar[rd] \ar[rdd]    &                  & 3' \\
                  & 4 \ar[ru]           &                                     & \bullet \ar[ru]           &                                     & \bullet \ar[ru]           &                                     & \bullet \ar[ru]           &                                     & \bullet \ar[ru]           &                                     & \bullet \ar[ru]           &                                     & \bullet \ar[ru]           &                                     & \bullet \ar[ru]           &                                     & \bullet \ar[ru]           &                                      & \bullet \ar[ru]           &                                      & \bullet \ar[ru]           &                                     & 4' \ar[ru]  &   \\
                  & 5 \ar[ruu] \ar[rdd] &                                     & \bullet \ar[ruu] \ar[rdd] &                                     & \bullet \ar[ruu] \ar[rdd] &                                     & \bullet \ar[ruu] \ar[rdd] &                                     & \bullet \ar[ruu] \ar[rdd] &                                     & \bullet \ar[ruu] \ar[rdd] &                                     & \bullet \ar[ruu] \ar[rdd] &                                     & \bullet \ar[ruu] \ar[rdd] &                                     & \bullet \ar[ruu] \ar[rdd] &                                      & \bullet \ar[ruu] \ar[rdd] &                                      & \bullet \ar[ruu] \ar[rdd] &                                     & 5' \ar[ruu] &   \\
                  &                           &                                     &                           &                                     &                           &                                     &                           &                                     &                           &                                     &                           &                                     &                           &                                     &                           &                                     &                           &                                      &                           &                                      &                           &                                     &                  &   \\
 6 \ar[ruu] &                           & \bullet \ar[ruu]                    &                           & \bullet \ar[ruu]                    &                           & \bullet \ar[ruu]                    &                           & \bullet \ar[ruu]                    &                           & \bullet \ar[ruu]                    &                           & \bullet \ar[ruu]                    &                           & \bullet \ar[ruu]                    &                           & \bullet \ar[ruu]                    &                           & \bullet \ar[ruu]                     &                           & \bullet \ar[ruu]                     &                           & 6' \ar[ruu]                    &                  &  
}
    }
\]

Use these notations, we have:
%
%
%
\begin{equation}
    \begin{aligned}
    E_{[X]}(t)u_Y = & u_Y+t \gamma^{L_{X,Y,1,1}}_{XY} u_{L_{X,Y,1,1}} + \frac{t^2}{2!} \gamma^{L_{X,Y,1,1}}_{XY} \gamma^{L_{X,Y,2,1}}_{X,L_{X,Y,1,1}} u_{L_{X,Y,2,1}} \\
    &+ \frac{t^3}{3!} \gamma^{L_{X,Y,1,1}}_{XY} \gamma^{L_{X,Y,2,1}}_{X,L_{X,Y,1,1}} \gamma^{L_{X,Y,3,1}}_{X,L_{X,Y,2,1}} u_{L_{X,Y,3,1}}\\
    =& \sum\limits_{i=0}^3 C_{X,Y,i,1} t^i u_{L_{X,Y,i,1}},
    \end{aligned}
\end{equation}
where we denote the coefficients by $C_{X,Y,i,1}$ for $i=0,1,2,3$.
Note that if $[L_{X,Y,i,1}]$ doesn't exist for some $X,Y$ and $i\leq 3$, then as stipulated, we have $\gamma_{X,L_{X,Y,i-1,1}}^{L_{X,Y,i,1}} = 0$, which means the term $u_{L_{X,Y,i,1}}$ actually doesn't appear in the above formula.  

Ringel calculated all the possible values of $\gamma_{XY}^{L}$ in \cite{RINGEL1990137}.
Precisely, for $X,Y,Z\in \operatorname{ind}\mathcal{R}$ with $\gamma_{XY}^Z\neq 0$, if $d(X)=d(Y)=1$ and $d(Z)=2$, then $\gamma_{XY}^Z=\pm 2$.
If $d(X)=d(Y)=1$ and $d(Z)=3$, then $\gamma_{XY}^Z=\pm 3$.
If $d(X)=d(Y)=d(Z)=1$ and there exists a indecomposable object $M$ in $\mathcal{R}$ such that $d(M)=3$, then $\gamma_{XY}^Z=\pm 2$.
Otherwise, $\gamma_{XY}^Z =\pm 1$.

Using his result, for $i=0,1,2,3$, we have $C_{X,Y,i,1}\in \{ 0,1,-1\}$, thus are all integers.

Now let $\mathbb{K}$ be any field, and let $\mathfrak{g}_{\mathbb{K}} = \mathbb{K}\otimes \mathfrak{g}_{\mathbb{Z}} $.
Then $\mathfrak{g}_{\mathbb{K}}$ is a $\mathbb{K}$-vector space spanned by $1\otimes \mathbb{S}$.
For any $x,y\in  \mathbb{S}$, define $[-,-]$ on $1\otimes \mathbb{S}$ by 
\begin{align*}
    [1\otimes x,1\otimes y] = 1\otimes [x,y]
\end{align*}
and linearly extend it to the whole $\mathfrak{g}_{\mathbb{K}}$. 
Then the vector space $\mathfrak{g}_{\mathbb{K}}$ together with $[-,-]$ is a Lie algebra over $\mathbb{K}$.

We denote the image of any $a\in \mathbb{Z}$ in $\mathbb{K}$ still by $a$.
Since the coefficients in (1)$\sim $(4) are of the form $at^i,a\in \mathbb{Z},i\geq 0$, for each $X\in \operatorname{ind}\mathcal{R}$ and $t\in \mathbb{K}$, we define linear maps $\overline{E}_{[X]}(t)$ of $\mathfrak{g}_{\mathbb{K}}$ into itself by
\begin{align*}
    &\overline{E}_{[X]}(t) (1\otimes u_X)=1\otimes u_X .\\
    &\overline{E}_{[X]}(t) (1\otimes u_{TX})=1\otimes u_{TX}+t1\otimes H_X'+t^21\otimes u_X .\\
    &\overline{E}_{[X]}(t)(1\otimes H_Y')=1\otimes H_Y'+A_{YX}t1\otimes u_X,   \qquad \text{for any }  Y \in \mathcal{R} .
\end{align*}
and for any $ Y\in \operatorname{ind}\mathcal{R}$ such that $ Y\ncong X,Y\ncong TX$,
\begin{align*}
    \overline{E}_{[X]}(t)(1\otimes u_Y) = \sum\limits_{i: [L_{X,Y,i,1}] \text{ exists}} C_{X,Y,i,1} t^i 1\otimes u_{L_{X,Y,i,1}}.
\end{align*}

For each $X\in \operatorname{ind}\mathcal{R}$ and $t\in \mathbb{K}$, to show $ \overline{E}_{[X]}(t)$ is a Lie algebra endomorphism of $\mathfrak{g}_{\mathbb{K}}$, we consider the action of it on $1\otimes \mathbb{S}$, and calculate some polynomials in $\mathbb{Z}[t]$.
Since $E_{[X]}(s)$ is an automorphism of $\mathfrak{g}_{\mathbb{Z}}\otimes_{\mathbb{Z}}\mathbb{C}$ for all $s\in \mathbb{C}$, the same polynomials in $\mathbb{Z}[s]$ vanish for all $s\in \mathbb{C}$. 
Hence the polynomials in $\mathbb{K}[t]$ are identically zero and $ \overline{E}_{[X]}(t)$ is a Lie algebra endomorphism.
Since $ \overline{E}_{[X]}(t)$ is nonsingular and $E_{[X]}(s)E_{[X]}(-s) = 1$ holds for all $s\in \mathbb{C}$, we have $ \overline{E}_{[X]}(t) \overline{E}_{[X]}(-t) = 1$, and $ \overline{E}_{[X]}(-t)$ is the inverse of $\overline{E}_{[X]}(t)$.
Thus $\overline{E}_{[X]}(t)$ are automorphisms of $\mathfrak{g}_{\mathbb{K}}$.

From now on, we simply denote $1\otimes \mathbb{S}$ by $\mathbb{S}$, the elements $1\otimes H_X'$ by $H_X'$, $1\otimes u_Y$ by $u_Y$, and $\overline{E}_{[X]}(t)$ by $E_{[X]}(t)$.

Let $G=G(\mathcal{R})$ be the group generated by $E_{[X]}(t), X \in \operatorname{ind}\mathcal{R}, t\in \mathbb{K}$,and let $E_{[X]}$ be the subgroup of $G$ generated by $E_{[X]}(t), t \in \mathbb{K}$, for each $X\in \operatorname{ind}\mathcal{R}$.

\begin{lemma}\label{ABG}
    For any $X,Y\in \operatorname{ind}\mathcal{R}$ such that $X\ncong Y$ and $X\ncong TY$, the subcategory of $\mathcal{R}$ generated by $X,Y$ is of type $A_1\times A_1,A_2,B_2$ or $G_2$.
\end{lemma}

\begin{proof}
    Since $X\ncong TY$, there exists a hereditary category $\mathcal{B}$ containing $X$ and $Y$.
    Let $\mathcal{B}'$ be the intersection of $\mathcal{B}$ and the subcategory of $\mathcal{R}$ generated by $X,Y$.
    Then $\mathcal{B}'$ is closed under subobjects, quotients and extensions.
    Hence $\mathcal{B}'$ is a hereditary abelian category, and the subcategory generated by $X,Y$ is of type $A_1\times A_1,A_2,B_2$ or $G_2$.
\end{proof}

Then we develop the commutator formula for $E_{[X]}(t),E_{[Y]}(s)$, for those $X,Y\in \operatorname{ind}\mathcal{R}$ such that $X\ncong Y$ and $X \ncong TY$, $ t,s \in \mathbb{K}$. 
Denote the commutator of $E_{[X]}(t)$ and $E_{[Y]}(s)$ by
\begin{align*}
    (E_{[X]}(t),E_{[Y]}(s)) = E_{[X]}(t) E_{[Y]}(s) E_{[X]}(t)^{-1} E_{[Y]}(s)^{-1} .
\end{align*}

\begin{proposition}\label{commutator}
    For $X,Y\in \operatorname{ind}\mathcal{R}$ such that $X\ncong Y$ and $X \ncong TY$, $ t,s\in \mathbb{K}$, we have
    \begin{equation*}
        (E_{[X]}(t),E_{[Y]}(s))= \prod\limits_{i,j>0} E_{[L_{X,Y,i,j}]}(C_{X,Y,i,j}t^i s^j),
    \end{equation*}
where the product is taken over all pairs of $i,j>0$ such that $[L_{X,Y,i,j}]$ exists, and in order of increasing $i+j$. 
(Actually there's one case in which there're two terms in the product with the same value $i+j$, but these two terms are commutative.)

The constants $C_{X,Y,i,j}$ are given as follows (denote $L_{X,Y,i,j}$ simply by $L_{i,j}$)
\begin{align*}
    &C_{X,Y,1,1} = \gamma_{XY}^{L_{1,1}}, \qquad C_{X,Y,2,1} = \frac{1}{2!}\gamma_{XY}^{L_{1,1}}\gamma_{XL_{1,1}}^{L_{2,1}}, \qquad C_{X,Y,1,2} = -\frac{1}{2!}\gamma_{YX}^{L_{1,1}}\gamma_{YL_{1,1}}^{L_{1,2}}, \\
    &C_{X,Y,3,1} = \frac{1}{3!}\gamma_{XY}^{L_{1,1}}\gamma_{XL_{1,1}}^{L_{2,1}}\gamma_{XL_{2,1}}^{L_{3,1}},\qquad C_{X,Y,3,2} = \frac{1}{3} \gamma_{XY}^{L_{1,1}} \gamma_{XL_{1,1}}^{L_{2,1}} \gamma_{XL_{2,1}}^{L_{3,1}} \gamma_{YL_{3,1}}^{L_{3,2}},\\
    &C_{X,Y,1,3} = - \frac{1}{3!}\gamma_{YX}^{L_{1,1}}\gamma_{YL_{1,1}}^{L_{1,2}}\gamma_{YL_{1,2}}^{L_{1,3}} ,\qquad C_{X,Y,2,3} = \frac{1}{6} \gamma_{YX}^{L_{1,1}} \gamma_{YL_{1,1}}^{L_{1,2}} \gamma_{YL_{1,2}}^{L_{1,3}} \gamma_{XL_{1,3}}^{L_{2,3}}  . 
\end{align*}

\end{proposition}

\begin{proof}
    For $X,Y\in \operatorname{ind}\mathcal{R}$ such that $X\ncong Y$ and $X \ncong TY$, we denote $L_{X,Y,i,j}$ simply by $L_{i,j}$. Then
    \begin{align*}
        &E_{[X]}(t) E_{[Y]}(s) E_{[X]}(t)^{-1} = E_{[X]}(t) exp (s \mathbf{ad}u_Y)  E_{[X]}(t)^{-1}  = exp(s \mathbf{ad} E_{[X]}(t)u_Y)\\
         &= exp(s\mathbf{ad}u_Y + st\gamma_{XY}^{L_{1,1}}\mathbf{ad}u_{L_{1,1}} + \frac{st^2}{2!}\gamma_{XY}^{L_{1,1}}\gamma_{XL_{1,1}}^{L_{2,1}}\mathbf{ad}u_{L_{2,1}} + \frac{st^3}{3!}\gamma_{XY}^{L_{1,1}}\gamma_{XL_{1,1}}^{L_{2,1}}\gamma_{XL_{2,1}}^{L_{3,1}}\mathbf{ad}u_{L_{3,1}} ) .
    \end{align*}

Firstly, we assume that $[L_{1,2}]$ doesn't exist. Then 
\begin{align*}
       & E_{[X]}(t) E_{[Y]}(s) E_{[X]}(t)^{-1} \\
       &= exp(  s\mathbf{ad}u_Y + \frac{st^3}{3!}\gamma_{XY}^{L_{1,1}}\gamma_{XL_{1,1}}^{L_{2,1}}\gamma_{XL_{2,1}}^{L_{3,1}}\mathbf{ad}u_{L_{3,1}}    ) exp(st\gamma_{XY}^{L_{1,1}}\mathbf{ad}u_{L_{1,1}} +  \frac{st^2}{2!}\gamma_{XY}^{L_{1,1}}\gamma_{XL_{1,1}}^{L_{2,1}}\mathbf{ad}u_{L_{2,1}})  \\
       &= exp(st\gamma_{XY}^{L_{1,1}}\mathbf{ad}u_{L_{1,1}} +  \frac{st^2}{2!}\gamma_{XY}^{L_{1,1}}\gamma_{XL_{1,1}}^{L_{2,1}}\mathbf{ad}u_{L_{2,1}})  exp(  s\mathbf{ad}u_Y + \frac{st^3}{3!}\gamma_{XY}^{L_{1,1}}\gamma_{XL_{1,1}}^{L_{2,1}}\gamma_{XL_{2,1}}^{L_{3,1}}\mathbf{ad}u_{L_{3,1}}    ) .
    \end{align*}
    To calculate this, we use the following trick(see proof in \cite[Lemma 5.1.2]{carter}). 
    If two nilpotent linear andomorphisms $\varphi,\psi $ of $\mathfrak{g}\otimes_{\mathbf{Q}} \mathbf{C}$ satisfy $[\varphi,\psi]=\varphi\psi-\psi\varphi$ is nilpotent and $[\varphi,\psi]$ commutes with $\varphi$ and $\psi$, then 
    \begin{align*}
        exp(\varphi + \psi) = exp(\varphi) exp(\psi) exp(-\frac{1}{2}[\varphi,\psi]).
    \end{align*}
    
    For $\mathbb{K} = \mathbb{C}$, by applying this lemma, we have 
    \begin{align*}
        &exp(st\gamma_{XY}^{L_{1,1}}\mathbf{ad}u_{L_{1,1}} +  \frac{st^2}{2!}\gamma_{XY}^{L_{1,1}}\gamma_{XL_{1,1}}^{L_{2,1}}\mathbf{ad}u_{L_{2,1}})  \\
    =   &exp(st\gamma_{XY}^{L_{1,1}}\mathbf{ad}u_{L_{1,1}})  exp( \frac{st^2}{2!}\gamma_{XY}^{L_{1,1}}\gamma_{XL_{1,1}}^{L_{2,1}}\mathbf{ad}u_{L_{2,1}}) exp(-\frac{s^2 t^3}{4} (\gamma_{XY}^{L_{1,1}})^2 \gamma_{XL_{1,1}}^{L_{2,1}} \gamma_{L_{1,1}L_{2,1}}^{L_{3,2}} \mathbf{ad} u_{L_{3,2}}) \\
    =   &E_{[L_{1,1}]}(st\gamma_{XY}^{L_{1,1}}) E_{[L_{2,1}]}(\frac{st^2}{2!}\gamma_{XY}^{L_{1,1}}\gamma_{XL_{1,1}}^{L_{2,1}}) E_{[L_{3,2}]}(-\frac{s^2 t^3}{4} (\gamma_{XY}^{L_{1,1}})^2 \gamma_{XL_{1,1}}^{L_{2,1}} \gamma_{L_{1,1}L_{2,1}}^{L_{3,2}})
    \end{align*}
   and
   \begin{align*}
     &  exp(  s\mathbf{ad}u_Y + \frac{st^3}{3!}\gamma_{XY}^{L_{1,1}}\gamma_{XL_{1,1}}^{L_{2,1}}\gamma_{XL_{2,1}}^{L_{3,1}}\mathbf{ad}u_{L_{3,1}}  )\\
    =&  E_{[L_{3,1}]}( \frac{st^3}{3!}\gamma_{XY}^{L_{1,1}}\gamma_{XL_{1,1}}^{L_{2,1}}\gamma_{XL_{2,1}}^{L_{3,1}})   E_{[Y]}(s)  E_{[L_{3,2}]}(\frac{s^2 t^3}{12} \gamma_{XY}^{L_{1,1}} \gamma_{XL_{1,1}}^{L_{2,1}} \gamma_{XL_{2,1}}^{L_{3,1}} \gamma_{YL_{3,1}}^{L_{3,2}}    ).
   \end{align*}

   Putting them together, we have
   \begin{align*}
      &(E_{[X]}(t),E_{[Y]}(s)) \\
    = &E_{[L_{1,1}]}(st\gamma_{XY}^{L_{1,1}}) E_{[L_{2,1}]}(\frac{st^2}{2!}\gamma_{XY}^{L_{1,1}}\gamma_{XL_{1,1}}^{L_{2,1}}) E_{[L_{3,2}]}(-\frac{s^2 t^3}{4} (\gamma_{XY}^{L_{1,1}})^2 \gamma_{XL_{1,1}}^{L_{2,1}} \gamma_{L_{1,1}L_{2,1}}^{L_{3,2}}) \\
      & \times  E_{[L_{3,1}]}( \frac{st^3}{3!}\gamma_{XY}^{L_{1,1}}\gamma_{XL_{1,1}}^{L_{2,1}}\gamma_{XL_{2,1}}^{L_{3,1}})  E_{[L_{3,2}]}(\frac{s^2 t^3}{12} \gamma_{XY}^{L_{1,1}} \gamma_{XL_{1,1}}^{L_{2,1}} \gamma_{XL_{2,1}}^{L_{3,1}} \gamma_{YL_{3,1}}^{L_{3,2}}  ) \\
    = &   E_{[L_{1,1}]}(st\gamma_{XY}^{L_{1,1}}) E_{[L_{2,1}]}(\frac{st^2}{2!}\gamma_{XY}^{L_{1,1}}\gamma_{XL_{1,1}}^{L_{2,1}}) E_{[L_{3,1}]}( \frac{st^3}{3!}\gamma_{XY}^{L_{1,1}}\gamma_{XL_{1,1}}^{L_{2,1}}\gamma_{XL_{2,1}}^{L_{3,1}}) \\
      & \times  E_{[L_{3,2}]}(\frac{s^2 t^3}{12} \gamma_{XY}^{L_{1,1}} \gamma_{XL_{1,1}}^{L_{2,1}} (\gamma_{XL_{2,1}}^{L_{3,1}} \gamma_{YL_{3,1}}^{L_{3,2}} - 3 \gamma_{XY}^{L_{1,1}} \gamma_{L_{1,1}L_{2,1}}^{L_{3,2}} ) ).
   \end{align*}
  
   As a consequence of Prop \ref{bmt} (or Jacobi identity), we have 
   \begin{align*}
    \gamma_{XL_{2,1}}^{L_{3,1}} \gamma_{YL_{3,1}}^{L_{3,2}} = \gamma_{YX}^{L_{1,1}}\gamma_{L_{1,1}L_{2,1}}^{L_{3,2}}.
   \end{align*}

   Since $\gamma_{YX}^{L_{1,1}} = - \gamma_{XY}^{L_{1,1}}$, 
   \begin{align*}
    \frac{s^2 t^3}{12} \gamma_{XY}^{L_{1,1}} \gamma_{XL_{1,1}}^{L_{2,1}} (\gamma_{XL_{2,1}}^{L_{3,1}} \gamma_{YL_{3,1}}^{L_{3,2}} - 3 \gamma_{XY}^{L_{1,1}} \gamma_{L_{1,1}L_{2,1}}^{L_{3,2}} ) = \frac{s^2 t^3}{3} \gamma_{XY}^{L_{1,1}} \gamma_{XL_{1,1}}^{L_{2,1}} \gamma_{XL_{2,1}}^{L_{3,1}} \gamma_{YL_{3,1}}^{L_{3,2}},
   \end{align*}
   and thus 
   \begin{align*}
     &(E_{[X]}(t),E_{[Y]}(s))\\
    =&  E_{[L_{1,1}]}(st\gamma_{XY}^{L_{1,1}}) E_{[L_{2,1}]}(\frac{st^2}{2!}\gamma_{XY}^{L_{1,1}}\gamma_{XL_{1,1}}^{L_{2,1}}) E_{[L_{3,1}]}( \frac{st^3}{3!}\gamma_{XY}^{L_{1,1}}\gamma_{XL_{1,1}}^{L_{2,1}}\gamma_{XL_{2,1}}^{L_{3,1}}) \\
     & \times  E_{[L_{3,2}]}( \frac{s^2 t^3}{3} \gamma_{XY}^{L_{1,1}} \gamma_{XL_{1,1}}^{L_{2,1}} \gamma_{XL_{2,1}}^{L_{3,1}} \gamma_{YL_{3,1}}^{L_{3,2}}).
   \end{align*}

   Notice that the coefficients of $s^it^j$ in $E_{[L_{j,i}]}(\cdot)$, for various $i,j$ appearing in the above formula, are integers. 
   Thus the above formula holds for arbitrary field $\mathbb{K}$.



Secondly, if $[L_{2,1}]$ doesn't exist, then by the above result, we have
    \begin{align*}
        &(E_{[Y]}(s),E_{[X]}(t))\\
        =&  E_{[L_{1,1}]}(ts\gamma_{YX}^{L_{1,1}}) E_{[L_{1,2}]}(\frac{ts^2}{2!}\gamma_{YX}^{L_{1,1}}\gamma_{YL_{1,1}}^{L_{1,2}}) E_{[L_{1,3}]}( \frac{ts^3}{3!}\gamma_{YX}^{L_{1,1}}\gamma_{YL_{1,1}}^{L_{1,2}}\gamma_{YL_{1,2}}^{L_{1,3}}) \\
         & \times  E_{[L_{2,3}]}( \frac{t^2 s^3}{3} \gamma_{YX}^{L_{1,1}} \gamma_{YL_{1,1}}^{L_{1,2}} \gamma_{YL_{1,2}}^{L_{1,3}} \gamma_{XL_{1,3}}^{L_{2,3}})     .
    \end{align*}

    Thus
    \begin{align*}
        &(E_{[X]}(t),E_{[Y]}(s)) = (E_{[Y]}(s),E_{[X]}(t))^{-1} \\
       =&  E_{[L_{2,3}]}( -\frac{t^2 s^3}{3} \gamma_{YX}^{L_{1,1}} \gamma_{YL_{1,1}}^{L_{1,2}} \gamma_{YL_{1,2}}^{L_{1,3}} \gamma_{XL_{1,3}}^{L_{2,3}})  E_{[L_{1,3}]}( - \frac{ts^3}{3!}\gamma_{YX}^{L_{1,1}}\gamma_{YL_{1,1}}^{L_{1,2}}\gamma_{YL_{1,2}}^{L_{1,3}}) \\
        & \times E_{[L_{1,2}]}(-\frac{ts^2}{2!}\gamma_{YX}^{L_{1,1}}\gamma_{YL_{1,1}}^{L_{1,2}})  E_{[L_{1,1}]}(-ts\gamma_{YX}^{L_{1,1}}) \\
       =&  E_{[L_{1,1}]}(-ts\gamma_{YX}^{L_{1,1}}) E_{[L_{1,2}]}(-\frac{ts^2}{2!}\gamma_{YX}^{L_{1,1}}\gamma_{YL_{1,1}}^{L_{1,2}}) E_{[L_{1,3}]}( - \frac{ts^3}{3!}\gamma_{YX}^{L_{1,1}}\gamma_{YL_{1,1}}^{L_{1,2}}\gamma_{YL_{1,2}}^{L_{1,3}}) \\
        & \times  E_{[L_{2,3}]}(\frac{t^2 s^3}{6} \gamma_{YX}^{L_{1,1}} \gamma_{YL_{1,1}}^{L_{1,2}}(3\gamma_{YX}^{L_{1,1}} \gamma_{L_{1,2}L_{1,1}}^{L_{2,3}}-2\gamma_{YL_{1,2}}^{L_{1,3}} \gamma_{XL_{1,3}}^{L_{2,3}}  ) ),
    \end{align*}
    where the last equality follows from
    \begin{align*}
        (E_{[L_{1,2}]}(-\frac{ts^2}{2!}\gamma_{YX}^{L_{1,1}}\gamma_{YL_{1,1}}^{L_{1,2}}), E_{[L_{1,1}]}(-ts\gamma_{YX}^{L_{1,1}})) = E_{[L_{2,3}]}(\frac{t^2 s^3}{2}(\gamma_{YX}^{L_{1,1}})^2 \gamma_{YL_{1,1}}^{L_{1,2}} \gamma_{L_{1,2}L_{1,1}}^{L_{2,3}}).
    \end{align*}

    Again, use Prop \ref{bmt} (or Jacobi identity) we have 
    \begin{align*}
        \gamma_{YL_{1,2}}^{L_{1,3}} \gamma_{XL_{1,3}}^{L_{2,3}} = \gamma_{XY}^{L_{1,1}} \gamma_{L_{1,1}L_{1,2}}^{L_{2,3}} = \gamma_{YX}^{L_{1,1}} \gamma_{L_{1,2}L_{1,1}}^{L_{2,3}} .
    \end{align*}
    
    Thus
    \begin{align*}
        & (E_{[X]}(t),E_{[Y]}(s)) \\
       =&  E_{[L_{1,1}]}(-ts\gamma_{YX}^{L_{1,1}}) E_{[L_{1,2}]}(-\frac{ts^2}{2!}\gamma_{YX}^{L_{1,1}}\gamma_{YL_{1,1}}^{L_{1,2}}) E_{[L_{1,3}]}( - \frac{ts^3}{3!}\gamma_{YX}^{L_{1,1}}\gamma_{YL_{1,1}}^{L_{1,2}}\gamma_{YL_{1,2}}^{L_{1,3}}) \\
        & \times  E_{[L_{2,3}]}(\frac{t^2 s^3}{6} \gamma_{YX}^{L_{1,1}} \gamma_{YL_{1,1}}^{L_{1,2}} \gamma_{YL_{1,2}}^{L_{1,3}} \gamma_{XL_{1,3}}^{L_{2,3}}   ).
    \end{align*}


Finally, we assume that both $[L_{1,2}]$ and $[L_{2,1}]$ exist.
By Lemma \ref{ABG}, in this case $[L_{1,3}]$ and $[L_{3,1}]$ don't exist.
Then we have
    \begin{align*}
      &  E_{[X]}(t) E_{[Y]}(s) E_{[X]}(t)^{-1} \\
    = & exp(s\mathbf{ad}u_Y + st\gamma_{XY}^{L_{1,1}}\mathbf{ad}u_{L_{1,1}} + \frac{st^2}{2!}\gamma_{XY}^{L_{1,1}}\gamma_{XL_{1,1}}^{L_{2,1}}\mathbf{ad}u_{L_{2,1}} ) \\
    = & exp(s\mathbf{ad}u_Y + st\gamma_{XY}^{L_{1,1}}\mathbf{ad}u_{L_{1,1}}) exp( \frac{st^2}{2!}\gamma_{XY}^{L_{1,1}}\gamma_{XL_{1,1}}^{L_{2,1}}\mathbf{ad}u_{L_{2,1}} ) \\
    = & exp(st\gamma_{XY}^{L_{1,1}}\mathbf{ad}u_{L_{1,1}}) exp(s\mathbf{ad}u_Y) exp(-\frac{s^2 t}{2!} \gamma_{XY}^{L_{1,1}} \gamma_{L_{1,1}Y}^{L_{1,2}} \mathbf{ad} u_{L_{1,2}} )exp( \frac{st^2}{2!}\gamma_{XY}^{L_{1,1}}\gamma_{XL_{1,1}}^{L_{2,1}}\mathbf{ad}u_{L_{2,1}} ) .
    \end{align*}

    Thus
    \begin{align*}
        (E_{[X]}(t),E_{[Y]}(s)) = E_{[L_{1,1}]}(st\gamma_{XY}^{L_{1,1}}) E_{[L_{1,2}]}(-\frac{s^2 t}{2!} \gamma_{YX}^{L_{1,1}} \gamma_{YL_{1,1}}^{L_{1,2}}) E_{[L_{2,1}]}(\frac{st^2}{2!}\gamma_{XY}^{L_{1,1}}\gamma_{XL_{1,1}}^{L_{2,1}}),
    \end{align*}
    where the last two terms in the product are commutative.
   The proof is finished.
\end{proof}

\begin{remark}
    For a pair of objects $X,Y$ in $\mathcal{R}$ and $X\ncong TY$, we can define a relative position on them.
    Since $X\ncong TY$, there exists an abelian subcategory $\mathcal{A}$ of $\mathcal{R}$ containing $X,Y$.
    We denote $X>Y$ if $\operatorname{Ext}_{\mathcal{A}}^1(Y,X)\neq 0$ and $\operatorname{Ext}_{\mathcal{A}}^1(X,Y)=0$.
    Since $\operatorname{Ext}_{\mathcal{A}}^1(Y,X)$ is nonzero if and only if there exists a triangle $X\rightarrow L \rightarrow Y \rightarrow TX$ in $\mathcal{R}$ for some $L$, this relative position doesn't depend on which abelian subcategory (containing $X,Y$) is chosen.
    Note that this definition doesn't give a partial order, since in general $X>Y,Y>Z$ do not imply $X>Z$.

    For $X,Y,L\in \mathcal{R}$, if $\gamma_{XY}^L\neq 0$, we have $X>Y$ implies $\gamma_{XY}^L>0$ and $X<Y$ implies $\gamma_{XY}^L<0$.
    Thus the relative position provides a way to determine the signs of the structure constants.
    For example, we assume $X>Y$ and denote $L_{X,Y,i,j}$ (resp. $C_{X,Y,i,j}$) by $L_{i,j}$ (resp. $C_{i,j}$).
    The structure constants are as follows.

    (1) If $[L_{1,2}]$ doesn't exist, then 
    \begin{align*}
        C_{1,1}=1, \quad C_{2,1}=1, \quad C_{3,1}=1, \quad C_{3,2}=-2,
    \end{align*}
    if nonzero.

    (2) If $[L_{2,1}]$ doesn't exist, then 
    \begin{align*}
        C_{1,1}=1,\quad C_{1,2}=-1,\quad C_{1,3}=1, \quad C_{2,3}=-1,
    \end{align*}
    if nonzero.

    (3) If $[L_{1,2}],[L_{2,1}]$ exist, then 
    \begin{align*}
        C_{1,1}=2, \quad C_{1,2}=-3, \quad C_{2,1}=3,
    \end{align*}
    if nonzero.
\end{remark}

\subsubsection{BGP Reflection Functors and Automorphisms}

For $X,Y \in \operatorname{ind}\mathcal{R}$ such that $Y \ncong X, Y \ncong TX$, 
let $p_{XY}$ be the largest integer $r$ such that $[L_{TX,Y,r,1}]$ exists, 
and let $q_{XY}$ be the largest integer $s$ such that $[L_{X,Y,s,1}]$ exists.

Define automorphism $h_{[X]}(t) $ of $\mathfrak{g}_{\mathbb{K}}$ for each $X\in \operatorname{ind} \mathcal{R}$ and $t\in \mathbb{K}^{\times}$, whose action on $\mathbb{S}$ is given by 
\begin{align*}
   &h_{[X]}(t)u_{Y} = t^{A_{XY}} u_{Y}, \\
   &h_{[X]}(t)H_{S}' = H_{S}'.
\end{align*}

Following Xiao-Zhang-Zhu\cite{BGP}, we use BGP reflection functors to define automorphisms $n_{[X]}$ of $\mathfrak{g}_{\mathbb{K}}$, for $X \in \operatorname{ind} \mathcal{R}$.
For a given indecomposable object $X$ in $\mathcal{R}$, there exists a hereditary subcategory $\mathcal{A}$ of $\mathcal{R}$ such that $X$ is a projective simple object in $\mathcal{A}$.
Suppose $X$ corresponds to some sink $i$ and denote the corresponding BGP reflection functor by $R(S_i^+)$.
Let $\mathcal{A}'=R(S_i^+) \mathcal{A}$. 
Suppose $\Pi(\mathcal{A})=\{S_1,\cdots ,S_m\}$ is a complete set of pairwise nonisomorphic simple objects in $\mathcal{A}$.
We can take $\mathbb{B} = \{ u_M \}_{M\in \operatorname{ind}\mathcal{R}} \cup \{ H_{S_j}' \}_{j=1,\cdots,m} \subseteq \mathbb{S}$,which is a basis of $\mathfrak{g}_{\mathbb{K}}$.
In the proof of Thm \ref{PXgiso} (\cite[Thm 4.2]{1997ROOT}), Peng and Xiao have showed that there exists a unique Lie algebra isomorphism $\phi_{\mathcal{A}}:\mathfrak{g}\otimes_{\mathbb{Q}}\mathbb{C} \rightarrow \mathfrak{g}(\Phi(\mathcal{R}))$ such that
\begin{align*}
    u_{S_j} \mapsto x_{\underline{dim}_{\mathcal{A}}S_j}, \qquad u_{TS_j} \mapsto -x_{\underline{dim}_{\mathcal{A}}TS_j}, \qquad H_{S_j}' \mapsto -h_{S_j}.
\end{align*}

\begin{remark}
    Notice that here we only give the explicit correspondence for simple objects (which is enough because $\phi_{\mathcal{A}}$ is a Lie algebra isomorphism), since the signs $\pm$ of coefficients for other $\phi_{\mathcal{A}}(u_Y)$ depend on $\mathcal{A}$ and structure constants, and we can't write them down in a uniform way.
\end{remark}

We generalize these to arbitrary field $\mathbb{K}$.
Note that $\{ x_{\underline{dim}_{\mathcal{A}}M} \}_{M\in \operatorname{ind}\mathcal{R}} \cup \{ h_{S_j} \}_{j=1,\cdots,m}$ form a Chevalley basis of $\mathfrak{g}(\Phi(\mathcal{R}))$. 
Denote the $\mathbb{Z}$-span of this Chevalley basis by $\mathfrak{g}(\Phi(\mathcal{R}))_{\mathbb{Z}}$, and let $\mathfrak{g}(\Phi(\mathcal{R}))_{\mathbb{K}} = \mathbb{K} \otimes \mathfrak{g}(\Phi(\mathcal{R}))_{\mathbb{Z}}$.
Then we have a Lie algebra isomorphism $\mathfrak{g}_{\mathbb{K}}\rightarrow \mathfrak{g}(\Phi(\mathcal{R}))_{\mathbb{K}}$, which is still denoted by $\phi_{\mathcal{A}}$.
Similar for $\phi_{\mathcal{A}'}$. 

Let $n_{[X]} = \phi_{\mathcal{A}'}^{-1} \phi_{\mathcal{A}}:\mathfrak{g}_{\mathbb{K}} \rightarrow \mathfrak{g}_{\mathbb{K}}$.
We can calculate its action on $\mathbb{B}$ by definition:
\begin{align*}
    &n_{[X]}u_{X} = u_{TX}. \\
    &n_{[X]}u_{TX} = u_X .
\end{align*}
For $Y\in \Pi(\mathcal{A}), Y \ncong X$,
\begin{equation*}
    n_{[X]}u_Y =  u_{L_{X,Y,q_{XY},1}} = C_{X,Y,q_{XY},1}   u_{L_{X,Y,q_{XY},1}} .
\end{equation*}
The last equality holds since $X$ is projective simple in $\mathcal{A}$, which means
 \begin{align*}
     \gamma_{XY}^{L_{X,Y,1,1}} = 1, \quad \gamma_{XL_{X,Y,1,1}}^{L_{X,Y,2,1}} = 2, \quad \gamma_{XL_{X,Y,2,1}}^{L_{X,Y,3,1}} = 3,
 \end{align*}
if they're nonzero.

For simplicity, we introduce a symbol $\omega_M(N)$ for each pair $M,N \in \operatorname{ind}\mathcal{R}$.
If $M \cong N$ or $M \cong TN$, let $\omega_M(N) = TN$.
Otherwise, define
\begin{equation*}
    \omega_M(N) = 
    \begin{cases}
        L_{TM,N,p_{MN}-q_{MN},1}  &  \text{if } p_{MN}-q_{MN} >0 \\
        L_{M,N,q_{MN}-p_{MN},1}  &  \text{if } p_{MN}-q_{MN} \leqslant 0
    \end{cases}
    .
\end{equation*}
Notice that (considering the correspondence between isomorphism class of indecomposables of $\mathcal{R}$ and the root system $\Phi = \Phi(\mathcal{R})$)
\begin{align*}
  \underline{dim} \omega_M(N) &= \underline{dim}N - (p_{MN}-q_{MN}) \underline{dim}M \\
  &= \underline{dim}N - (p_{\underline{dim}M,\underline{dim}N}-q_{\underline{dim}M,\underline{dim}N}) \underline{dim}M \\
  &=\underline{dim}N - A_{\underline{dim}M,\underline{dim}N}\underline{dim}M\\
  &= \omega_{\underline{dim}M}(\underline{dim}N) ,   
\end{align*}
where for any linearly independent $\alpha,\beta \in \Phi$, we let $p_{\alpha,\beta}$ be the largest integer $r$ such that $\beta - r\alpha\in \Phi$, 
and let $q_{\alpha,\beta}$ be the largest integer $s$ such that $\beta + s\alpha\in \Phi$.
The second equality follows from $p_{MN} = p_{\underline{dim}M,\underline{dim}N},q_{MN} = q_{\underline{dim}M,\underline{dim}N}$, which can be deduced by definition.
For $\alpha,\beta \in \Phi$, it's well-known that $A_{\alpha,\beta} = p_{\alpha,\beta} - q_{\alpha,\beta}$, 
and the last $\omega_r$ is the reflection corresponds to $r\in \Phi$ in the Weyl group.
Moreover, we have 
\begin{align*}
    A_{MN} = \frac{2(H_M,H_N)}{(H_M,H_M)} = \frac{2(\underline{dim}M,\underline{dim}N)}{(\underline{dim}M,\underline{dim}M)} = A_{\underline{dim}M,\underline{dim}N}.
\end{align*}
Hence $A_{MN} = p_{MN} - q_{MN}$.

Use this symbol, we have 
\begin{align*}
    n_{[X]}u_Y = u_{\omega_X(Y)}.
\end{align*}
for $Y \in \Pi(\mathcal{A}), Y \ncong X$.

Similarly, we have
\begin{align*}
    &n_{[X]}u_{TY} = u_{\omega_X(TY)}. \\
    &n_{[X]} H_{S_j}' = H_{\omega_X(S_j)}'.
\end{align*}
where the second equality holds since $\underline{dim} \omega_M(N) = \omega_{\underline{dim}M}(\underline{dim}N)$ and $(-,-)$ is $W$-invariant.

Now we calculate the explicit action of $E_{[X]}(t) E_{[TX]}(t^{-1}) E_{[X]}(t)$ on $\mathbb{S}$, 
from which we want to show that the action of $n_{[X]}$ coincide with the action of $E_{[X]}(1) E_{[TX]}(1) E_{[X]}(1)$ on $\{u_Y| Y \in \Pi(\mathcal{A}) \text{ or } TY \in \Pi(\mathcal{A}) \} \bigcup \{H_{S_j}'|j=1,\cdots,m \}$. 
\begin{align*}
    &E_{[X]}(t) E_{[TX]}(t^{-1}) E_{[X]}(t) u_X = t^{-2} u_{TX} .\\
    &E_{[X]}(t) E_{[TX]}(t^{-1}) E_{[X]}(t) u_{TX} = t^2 u_X .\\
    &E_{[X]}(t) E_{[TX]}(t^{-1}) E_{[X]}(t) H_S' = H_S' -A_{SX}H_X'  = H_{\omega_X(S)}'.
\end{align*}

For $Y \ncong X,Y \ncong TX$ and $i,j \geq 0$, denote $p=p_{XY}$, $q=q_{XY}$, $L_{i,j}=L_{X,Y,i,j}$, $C_{i,j} = C_{X,Y,i,j}$, $L_{-i,j}=L_{TX,Y,i,j}$. 
Then
\begin{align*}
    E_{[X]}(t) u_Y = \sum_{i=0}^{q} C_{i,1} &t^i u_{L_{i,1}},\\
    E_{[TX]}(t^{-1}) E_{[X]}(t) u_Y &= \sum_{i=0}^{q} \sum_{j=0}^{i+p} C_{i,1} C_{TX,L_{i,1},j,1} t^{i-j} u_{L_{i-j,1}}\\
    &= \sum_{k=-p}^{q} ( \sum_{i=0}^{q} C_{i,1} C_{TX,L_{i,1},i-k,1}  ) t^k u_{L_{k,1}}.
\end{align*}

Since
\begin{align*}
    C_{i,1} C_{TX,L_{i,1},j,1} = \frac{1}{i!j!}\gamma_{XY}^{L_{1,1}} \cdots \gamma_{XL_{i-1,1}}^{L_{i,1}}  \gamma_{TX,L_{i,1}}^{L_{i-1,1}} \cdots \gamma_{TX,L_{i-j+1,1}}^{L_{i-j,1}},
\end{align*}
we can calculate the coefficients case by case and simplify the formula.
For example, if $p=0,q=2$, then 
\begin{align*}
    E_{[TX]}(t^{-1}) E_{[X]}(t)u_Y &=(1+\gamma_{XY}^{L_{1,1}}\gamma_{TX,L_{1,1}}^Y+\frac{1}{4}\gamma_{XY}^{L_{1,1}}\gamma_{XL_{1,1}}^{L_{2,1}}\gamma_{TX,L_{2,1}}^{L_{1,1}}\gamma_{TX,L_{1,1}}^Y) u_Y \\
    &+ (t\gamma_{XY}^{L_{1,1}} + \frac{t}{2}\gamma_{XY}^{L_{1,1}}\gamma_{XL_{1,1}}^{L_{2,1}}\gamma_{TX,L_{2,1}}^{L_{1,1}} ) u_{L_{1,1}} + \frac{t^2}{2!} \gamma_{XY}^{L_{1,1}}\gamma_{XL_{1,1}}^{L_{2,1}} u_{L_{2,1}}.
\end{align*}
In this case, 
\begin{align*}
    \gamma_{XY}^{L_{1,1}} = \pm 1, \quad \gamma_{XL_{1,1}}^{L_{2,1}} = \pm 2, \quad \gamma_{TX,L_{1,1}}^Y = -2\gamma_{XY}^{L_{1,1}} , \quad \gamma_{TX,L_{2,1}}^{L_{1,1}} = -\frac{1}{2} \gamma_{XL_{1,1}}^{L_{2,1}}.
\end{align*}
Thus 
\begin{align*}
    E_{[TX]}(t^{-1}) E_{[X]}(t)u_Y = \frac{t^2}{2!} \gamma_{XY}^{L_{1,1}}\gamma_{XL_{1,1}}^{L_{2,1}} u_{L_{2,1}} .
\end{align*}

Finally apply $E_{[X]}(t)$ to the formula, and obtain that 
\begin{align*}
    E_{[X]}(t) E_{[TX]}(t^{-1}) E_{[X]}(t) u_Y = \eta_{XY} t^{q-p} u_{\omega_X(Y)},
\end{align*}
where $\eta_{XY}$ is a product of some $\gamma$'s and some rational number, and $\eta_{XY} = \pm 1$ can be checked case by case.
Let $\eta_{XX} =  \eta_{X,TX} = 1$.

We can easily see that the action of $n_{[X]}$ coincide with the action of $E_{[X]}(1) E_{[TX]}(1) E_{[X]}(1)$ on $\{u_Y| Y \in \Pi(\mathcal{A}) \text{ or } TY \in \Pi(\mathcal{A}) \} \bigcup \{H_{S_j}'|j=1,\cdots,m \}$. 
Since they are both Lie algebra isomorphisms, we have 
\begin{align*}
    n_{[X]} = E_{[X]}(1) E_{[TX]}(1) E_{[X]}(1).
\end{align*}
As a corollary, we can write down the action of $n_{[X]}$ on all the element of $\mathbb{S}$ explicitly.

\begin{remark}
    When $\mathbb{K}=\mathbb{F}_2$, we can define uniformly $\phi_{\mathcal{A}}: u_M \mapsto x_{\underline{dim}_{\mathcal{A}}M}, M\in \operatorname{ind}\mathcal{R}$, saving the trouble of defining the sign $\pm$.
    Moreover, we have 
    \begin{align*}
        n_{[X]}u_Y &= u_{\omega_X(Y)}. \\
        n_{[X]}H_S' &= H_{\omega_X(S)}'.
    \end{align*}
    and since $ \underline{dim} \omega_M(N) =  \omega_{\underline{dim}M}(\underline{dim}N)  $, our $\{n_{[X]}\}_{X \in \operatorname{ind}\mathcal{R}}$ generate the Weyl group.
\end{remark}

Define $n_{[X]}(t) = h_{[X]}(t) n_{[X]}$ for all $X \in \operatorname{ind}\mathcal{R}$ and $t \in \mathbb{K}^{\times}$.
\begin{lemma}\label{nEEE}
    We have
    \begin{align*}
        &n_{[X]}(t) = E_{[X]}(t) E_{[TX]}(t^{-1}) E_{[X]}(t). \\
        &n_{[X]}(1) = n_{[X]} .\\
        &n_{[X]}(-1) = n_{[X]}^{-1}.
    \end{align*}
\end{lemma}

\begin{proof}
    Check the action of them on $\mathbb{S}$ .  
\end{proof}

By this Lemma, we have $n_{[X]}(t),h_{[X]}(t) \in G$, for any $ X\in \operatorname{ind}\mathcal{R}$ and $ t\in \mathbb{K}^{\times}$.

Now we calculate some conjugate relations which we will use later.

\begin{lemma}\label{nEn}
    For any $ X,Y \in \operatorname{ind}\mathcal{R},  t \in \mathbb{K}^{\times}, s \in \mathbb{K}$,
    \begin{align*}
        n_{[X]}(t) E_{[Y]}(s) n_{[X]}(t)^{-1} = E_{[\omega_X(Y)]}(\eta_{XY}t^{-A_{XY}}s).
    \end{align*}
\end{lemma}

\begin{proof}
    By definition, 
    \begin{align*}
        n_{[X]}(t) E_{[Y]}(s) n_{[X]}(t)^{-1} &= h_{[X]}(t) n_{[X]} E_{[Y]}(s) n_{[X]}^{-1} h_{[X]}(t)^{-1} \\
            &= h_{[X]}(t)  exp(s \eta_{XY} \mathbf{ad} u_{\omega_X(Y)})  h_{[X]}(t)^{-1} \\
            &= exp(s \eta_{XY}t^{A_{X,\omega_X(Y)}} \mathbf{ad}  u_{\omega_X(Y)})  \\
            &= E_{[\omega_X(Y)]}(\eta_{XY}t^{-A_{XY}}s),
    \end{align*}
    where the last equality follows from $A_{X,\omega_X(Y)} = -A_{XY}$.
\end{proof}

\begin{lemma} \label{nnn}
    For any $X,Y \in \operatorname{ind}\mathcal{R},  t,s \in \mathbb{K}^{\times}$,
    \begin{align*}
        n_{[X]}(t) n_{[Y]}(s) n_{[X]}(t)^{-1} = n_{[\omega_X(Y)]}(\eta_{XY}t^{-A_{XY}}s) .
    \end{align*}
\end{lemma}

\begin{proof}
   By Lemma \ref{nEEE} and Lemma \ref{nEn}, we have 
    \begin{align*}
        &n_{[X]}(t) n_{[Y]}(s) n_{[X]}(t)^{-1} \\
        &= n_{[X]}(t) E_{[Y]}(s) E_{[TY]}(s^{-1}) E_{[Y]}(s) n_{[X]}(t)^{-1} \\
              &=  E_{[\omega_X(Y)]}(\eta_{XY}t^{-A_{XY}}s) E_{[\omega_X(TY)]}(\eta_{X,TY}t^{-A_{X,TY}}s^{-1}) E_{[\omega_X(Y)]}(\eta_{XY}t^{-A_{XY}}s) \\
              &= E_{[\omega_X(Y)]}(\eta_{XY}t^{-A_{XY}}s) E_{[T\omega_X(Y)]}(\eta_{XY}^{-1}t^{A_{XY}}s^{-1}) E_{[\omega_X(Y)]}(\eta_{XY}t^{-A_{XY}}s) \\
              &= n_{[\omega_X(Y)]}(\eta_{XY}t^{-A_{XY}}s),
    \end{align*}
    where $\omega_X(TY) = T\omega_X(Y), \eta_{X,TY} = \eta_{XY} = \eta_{XY}^{-1}$ can be checked case by case.
\end{proof}

\begin{lemma} \label{nhn}
    For any $X,Y \in \operatorname{ind}\mathcal{R},  t,s \in \mathbb{K}^{\times}$,
    \begin{align*}
        n_{[X]}(t) h_{[Y]}(s) n_{[X]}(t)^{-1} = h_{[\omega_X(Y)]}(s).  
    \end{align*}
    
\end{lemma}

\begin{proof}
    By definition and Lemma \ref{nnn}, we have 
    \begin{align*}
       & n_{[X]}(t) h_{[Y]}(s) n_{[X]}(t)^{-1} \\
       &= n_{[X]}(t) n_{[Y]}(s) n_{[Y]}(-1) n_{[X]}(t)^{-1} \\
       &= n_{[\omega_X(Y)]}(\eta_{XY}t^{-A_{XY}}s) n_{[\omega_X(Y)]}(-\eta_{XY}t^{-A_{XY}}) \\
       &= h_{[\omega_X(Y)]}(\eta_{XY}t^{-A_{XY}}s) n_{[\omega_X(Y)]} n_{[\omega_X(Y)]}^{-1}  h_{[\omega_X(Y)]}(\eta_{XY}t^{-A_{XY}})^{-1} \\
       &= h_{[\omega_X(Y)]}(s).
    \end{align*}
\end{proof}

\begin{lemma} \label{hEh}
    For any  $ X,Y \in \operatorname{ind}\mathcal{R},  t \in \mathbb{K}^{\times}, s \in \mathbb{K}$,
    \begin{align*}
        h_{[X]}(t) E_{[Y]}(s) h_{[X]}(t)^{-1} = E_{[Y]}(t^{A_{XY}}s).
    \end{align*} 
\end{lemma}

\begin{proof}
By definition and Lemma \ref{nEn}, we have 
    \begin{align*}
        h_{[X]}(t) E_{[Y]}(s) h_{[X]}(t)^{-1} &= n_{[X]}(t) n_{[X]}(-1) E_{[Y]}(s) n_{[X]}(-1)^{-1} n_{[X]}(t)^{-1} \\
            &= n_{[X]}(t)  E_{[\omega_X(Y)]}(\eta_{XY}(-1)^{A_{XY}}s) n_{[X]}(t)^{-1} \\
            &= E_{[Y]}(\eta_{X,\omega_X(Y)}t^{-A_{X,\omega_X(Y)}} \eta_{XY} (-1)^{A_{XY}} s).
    \end{align*}
    Check case by case, we notice that 
    \begin{align*}
        \eta_{XY} \eta_{X,\omega_X(Y)} = (-1)^{A_{XY}},
    \end{align*}
    and the Lemma follows.
\end{proof}

\begin{lemma} \label{hhh}
    For any $X,Y \in \operatorname{ind}\mathcal{R},  t,s \in \mathbb{K}^{\times}$,
    \begin{align*}
        h_{[X]}(t) h_{[Y]}(s) h_{[X]}(t)^{-1} = h_{[Y]}(s).
    \end{align*}

\end{lemma}

\begin{proof}
    By definition and Lemma \ref{nhn}, we have 
    \begin{align*}
        &h_{[X]}(t) h_{[Y]}(s) h_{[X]}(t)^{-1} \\
        &= n_{[X]}(t) n_{[X]}(-1) h_{[Y]}(s) n_{[X]}(-1)^{-1} n_{[X]}(t)^{-1} \\
        &=  n_{[X]}(t) h_{[\omega_X(Y)]}(s) n_{[X]}(t)^{-1}\\
        &= h_{[Y]}(s).
    \end{align*}
\end{proof}

\begin{lemma}\label{hnh}
    For any  $X,Y \in \operatorname{ind}\mathcal{R},  t,s \in \mathbb{K}^{\times}$,
    \begin{align*}
        h_{[X]}(t) n_{[Y]}(s) h_{[X]}(t)^{-1} = n_{[Y]}(t^{A_{XY}}s).
    \end{align*}
    
\end{lemma}

\begin{proof}
    By Lemma \ref{nEEE} and Lemma \ref{hEh}, we have 
    \begin{align*}
        &h_{[X]}(t) n_{[Y]}(s) h_{[X]}(t)^{-1} \\
        &= h_{[X]}(t)  E_{[Y]}(s) E_{[TY]}(s^{-1}) E_{[Y]}(s)  h_{[X]}(t)^{-1} \\
        &= E_{[Y]}(t^{A_{XY}}s) E_{[TY]}(t^{A_{X,TY}}s^{-1}) E_{[Y]}(t^{A_{XY}}s) \\
        &= E_{[Y]}(t^{A_{XY}}s) E_{[TY]}(t^{-A_{XY}}s^{-1}) E_{[Y]}(t^{A_{XY}}s) \\
        &= n_{[Y]}(t^{A_{XY}}s).
    \end{align*}
\end{proof}

\subsubsection{Properties of Some Subgroups}

From now on we arbitrarily fix a complete section of the root category $\mathcal{R}$, and hence we get a hereditary subcategory $\mathcal{B}$ of $\mathcal{R}$.
This complete section gives a basis of $K'$, and let $\underline{dim}_{\mathcal{B}}M$ be the image of $M$ in $K'$ with respect to this basis.  
We denote $\underline{dim}_{\mathcal{B}}$ simply by $\underline{dim}$.
The length of an object $M$ in $\mathcal{B}$ is denoted by $l(M)$, and for $M\in T\mathcal{B}$, we define the length of $M$ by $l(M) = -l(TM)$.
Note that if we fix a set of representatives $\{ S_1,S_2,\cdots,S_m \}$ of isomorphism classes of simple objects in $\mathcal{B}$, and write 
$\underline{dim}M = \sum\limits_{i=1}^{m} a_i \underline{dim}S_i$, then $l(M)=  \sum\limits_{i=1}^{m} a_i $.
We can fix an order on $\operatorname{ind}\mathcal{R}$ compatible with the length, i.e. $X \succ  Y$ implies $l(X) \geq l(Y)$. 

Let $U = U(\mathcal{B})$ be the subgroup of $G$ generated by $E_{[M]}(t)$, for any $M \in \operatorname{ind}\mathcal{B}$, $ t \in \mathbb{K}$.
Let $V = V(\mathcal{B})$ be the subgroup of $G$ generated by $E_{[M]}(t)$, for any $M \in \operatorname{ind}T\mathcal{B}$, $ t \in \mathbb{K}$.

\begin{proposition}
    We have $U \cap V = \{1\}.$
\end{proposition}

\begin{proof}
    Let $\theta \in U \cap V$. Consider the action of $\theta$ on $\mathbb{S}$.
    We write the weight space decomposition of $\mathfrak{g}_{\mathbb{K}}$ as 
    \begin{align*}
        \mathfrak{g}_{\mathbb{K}}=K' \oplus  (\bigoplus\limits_{M \in \operatorname{ind}\mathcal{R}} \mathfrak{g}_{M} ) ,
    \end{align*}
    where $\mathfrak{g}_M$ is the $\mathbb{K}$-subspace spanned by $u_M$.
    
    Since $\theta \in U$, we have 
    \begin{align*}
        \theta H_S' = H_S' + x ,\quad x\in \bigoplus\limits_{M\in \operatorname{ind}\mathcal{B}} \mathfrak{g}_{M}.
    \end{align*}
    However, since $\theta \in V$, we also have $x\in \bigoplus\limits_{M\in \operatorname{ind}T\mathcal{B}} \mathfrak{g}_{M}$. 
    Thus $x=0$ and $\theta H_S' = H_S'$.

    Since $\theta \in U$, we have 
    \begin{align*}
        \theta u_M = u_M + x', \quad x' \in \bigoplus\limits_{L\succ M} \mathfrak{g}_{L},
    \end{align*}
   while $\theta \in V$ implies that $x' \in \bigoplus\limits_{L \prec M} \mathfrak{g}_{L}$. 
   Thus $x' = 0$ and $\theta u_M =u_M$.

   Since $\theta$ acts trivially on $\mathbb{S}$, we have $\theta = 1$ and the Proposition follows.
\end{proof}

\begin{proposition}\label{Uunique}
    Every element of $U$ can be uniquely expressed of the form 
    \begin{align*}
        \prod\limits_{M\in \operatorname{ind}\mathcal{B}} E_{[M]}(t_M),
    \end{align*}
    where the product is taken over all indecomposable objects in $\mathcal{B}$ in increasing order (with respect to the chosen order $\succ$).
\end{proposition}

\begin{proof}
    By Prop \ref{commutator}, every element in $U$ can be written of the form $ \prod\limits_{M\in \operatorname{ind}\mathcal{B}} E_{[M]}(t_M)$ with the product being taken over all objects in $\operatorname{ind}\mathcal{B}$ in increasing order.
    Now we prove the uniqueness.

    For any integer $t\geq 1$, define $U_t = \prod\limits_{l(M)\geq t} E_{[M]}$, the product being taken in any order.
    It can be easily shown that $U_t$ is a subgroup of $U$, $U_{t+1}$ is a normal subgroup of $U_t$ and $U_t/U_{t+1}$ is abelian.
    We prove the uniqueness by descending induction on $t$.

    For $t$ sufficiently large, we have $U_t = 1$ and the uniqueness is obvious.

    Assume that the uniqueness holds for $U_{t+1}$, and we prove for $U_t$. Suppose
    \begin{align*}
        \prod\limits_{l(M)=t} E_{[M]}(t_M)   \prod\limits_{l(M)\geq t+1} E_{[M]}(t_M) =  \prod\limits_{l(M)=t} E_{[M]}(t_M')    \prod\limits_{l(M)\geq t+1} E_{[M]}(t_M').
    \end{align*}

    Denote the natural map $\bar{ } : U_t \rightarrow U_t/U_{t+1}$. Thus we have
    \begin{align*}
        \prod\limits_{l(M)=t} \overline{E_{[M]}(t_M)} =  \prod\limits_{l(M)=t} \overline{E_{[M]}(t_M')}.
    \end{align*} 

    Since $U_t/U_{t+1}$ is abelian, 
    \begin{align*}
        \prod\limits_{l(M)=t} \overline{E_{[M]}(t_M - t_M')} = 1.
    \end{align*}

    Now we choose an arbitrary $N$ with $l(N) = t$. Then  
    \begin{align*}
        \prod_{\substack{l(M) = t \\ M\ncong N}} \overline{E_{[M]}(t_M - t_M')} =   \overline{E_{[N]}(t_N' - t_N) }.
    \end{align*}
    
    Thus
    \begin{align*}
        E_{[N]}(t_N' - t_N) =  \prod_{\substack{l(M) = t \\ M\ncong N}} E_{[M]}(t_M - t_M') \prod\limits_{l(M)\geq t+1} E_{[M]}(a_M),
    \end{align*}
    for some $a_M$ for each $M$ satisfying $l(M)\geq t+1$.

    On the one hand, 
    \begin{align*}
        E_{[N]}(t_N' - t_N)u_{TN} = u_{TN} + (t_N'-t_N) H_N' + (t_N'-t_N)^2u_N.
    \end{align*}

    On the other hand, for each $M$ such that $l(M)\geq t+1$, we have
    \begin{align*}
        E_{[M]}(a_M) u_{TN} = u_{TN} + \sum\limits_{l(L)\geq 1} b_{ML} u_L,
    \end{align*}
    and for each $M$ such that $l(M) = t$, 
    \begin{align*}
        E_{[M]}(t_M-t_M') u_{TN} = u_{TN}.
    \end{align*}
    This holds because there's no extension of $M$ and $TN$, otherwise its length would be zero, which is impossible.
    Put them together, we have
    \begin{align*}
        \prod_{\substack{l(M) = t \\ M\ncong N}} E_{[M]}(t_M - t_M') \prod\limits_{l(M)\geq t+1} E_{[M]}(a_M) u_{TN} = u_{TN} + \sum\limits_{l(L)\geq 1}b_L u_L.
    \end{align*}
    Since $H_N'$ and $u_M, M \in \operatorname{ind}\mathcal{R}$ are linearly independent, we have
    \begin{align*}
        t_N'-t_N = 0.
    \end{align*}

    Because $N$ is arbitrarily chosen, $t_M=t_M'$ holds for all $M$ such that $l(M)=t$. Thus
    \begin{align*}
        \prod\limits_{l(M)\geq t+1} E_{[M]}(t_M) =  \prod\limits_{l(M)\geq t+1} E_{[M]}(t_M')  .
    \end{align*}
    By induction hypothesis, we have $t_M=t_M'$ holds for all $M$ such that $l(M)\geq t$.
    The uniqueness follows.
\end{proof}

Let $H$ be the subgroup of $G$ generated by $h_{[M]}(t), M \in \operatorname{ind}\mathcal{R}, t \in \mathbb{K}^{\times}$.
By Lemma \ref{hhh}, $H$ is abelian.
Let $B$ be the subgroup of $G$ generated by $U$ and $H$.
By Lemma \ref{hEh}, $U$ is a normal subgroup of $B$ and thus $B=UH$.

\begin{lemma}
    We have $U\cap H = \{1\}$.
\end{lemma}

\begin{proof}
    For any $ h \in U \cap H$, we can write $h=E_{[M_1]}(t_1)\cdots E_{[M_l]}(t_l)$ for some $M_i\in \operatorname{ind}\mathcal{B},t_i \in \mathbb{K}, i=1,\cdots ,l$.
    Since $H$ is abelian, we have $h_{[N]}(s) h h_{[N]}(s)^{-1} = h$ for any $ N\in \operatorname{ind}\mathcal{R}, s \in \mathbb{K}^{\times}$. Since
    \begin{align*}
        h_{[N]}(s) h h_{[N]}(s)^{-1} = E_{[M_1]}(s^{A_{NM_1}}t_1) \cdots E_{[M_l]}(s^{A_{NM_l}}t_l),
    \end{align*} 
    by Prop \ref{Uunique}, 
    \begin{align*}
        s^{A_{NM_i}}t_i = t_i \qquad  i=1,\cdots ,l.
    \end{align*}
    Since $N,s$ are arbitrarily chosen, 
    if $\mathbb{K} \neq \mathbb{F}_2 $, then $t_i = 0 ,  i=1,\cdots , l$. 
    Then it follows that $h=1$ and   $U\cap H = \{1\}$.

    In the case $\mathbb{K}=\mathbb{F}_2$, by definition we have $h_{[X]}(1)=1$ (the identity map) for any $ X \in \operatorname{ind}\mathcal{R}$. 
    Thus $H=1$ and obviously   $U\cap H = \{1\}$.
\end{proof}

Let $N$ be the subgroup of $G$ generated by $H$ and $n_{[M]}, M \in \operatorname{ind}\mathcal{R}$.
By Lemma \ref{nhn}, $H$ is a normal subgroup of $N$.
Let $W$ be the Weyl group associated to the root system $\Phi$. 

\begin{proposition}
    We have an isomorphism $N/H \cong W$.
\end{proposition}

\begin{proof}
    Define a group homomorphism $\varphi:N\rightarrow W$ such that $\varphi(n_{[X]}) = \omega_{\underline{dim}X}$ for each $X \in \operatorname{ind}\mathcal{R}$.
    We have shown that $n_{[X]}H_S' = H_{\omega_X(S)}', n_{[X]}u_Y=\eta_{XY}u_{\omega_X(Y)}$ and $\underline{dim}\omega_X(Y) = \omega_{\underline{dim}X}(\underline{dim}Y)$, 
    thus $n_{[X]}$ permute the root spaces in the same way that the element $\omega_{\underline{dim}X}$ of $W$ permutes the roots, while $h \in H$ stabilize each root space.
    So $\varphi $ is a well-defined surjective group homomorphism and $H \subseteq Ker\varphi$.
    $\varphi$ induces a surjective group homomorphism from $N/H$ to $W$ as the composition $N/H \rightarrow N/Ker\varphi \rightarrow W$.

    Denote by $\overline{n}$ the image of $n\in N$ in $N/H$.
    Since $n_{[X]}^2 = h_{[X]}(-1)^{-1} \in H$ and by Lemma \ref{nnn} $n_{[X]}n_{[Y]}n_{[X]}^{-1} = n_{[\omega_X(Y)]}(\eta_{XY})$, we have
    \begin{align*}
        \overline{n_{[X]}}^2 = 1, \qquad (\overline{n_{[X]}})(\overline{n_{[Y]}})(\overline{n_{[X]}}^{-1}) = \overline{n_{[\omega_X(Y)]}}.
    \end{align*}

    Notice that $W$ can be identified with the abstract group generated by symbols $\omega_{\alpha}, \alpha \in \Phi$ subject to relations 
    \begin{align*}
        \omega_{\alpha}^2 = 1, \qquad \omega_{\alpha} \omega_{\beta} \omega_{\alpha}^{-1} = \omega_{\omega_{\alpha}(\beta)},
    \end{align*}
    where $\omega_{\alpha}(\beta) = \beta - A_{\alpha\beta}\alpha = \beta - <\beta,\alpha>\alpha$ with $<\beta,\alpha>=A_{\alpha\beta} = \frac{2(\alpha,\beta)}{(\alpha,\alpha)}$.

    Hence there exists a surjective group homomorphism $\psi:W \rightarrow N/H$, $\psi(\omega_{\alpha}) = \overline{n_{[M]}}$, where $M$ is such that $\underline{dim}M = \alpha \in \Phi$.
    
    Obviously, $\varphi,\psi$ are inverse to each other and thus isomorphisms.
\end{proof}

\subsection{The Bruhat Decomposition}

We keep the notations in the previous subsection.

\begin{lemma}\label{nBn}
    We have the following formula.

    (1) For any $M\in \operatorname{ind}\mathcal{R}$, $Bn_{[M]}B = Bn_{[M]}^{-1}B$.

    (2) For any $i=1,\cdots,m$, 
    \begin{align*}
        n_{[S_i]}Bn_{[S_i]} = n_{[S_i]}Bn_{[S_i]}^{-1} \subseteq  E_{[TS_i]}B \subseteq  B\cup Bn_{[S_i]}B,
    \end{align*}
    where $\{ [S_1],\cdots,[S_m]\}$ corresponds to the basis of $K'$ given by the complete section.
    
\end{lemma}

\begin{proof}
    (1) follows from $n_{[M]}^2 \in H \subseteq  B$.

    (2) We only need to show that $E_{[TS_i]}\subseteq  B\cup Bn_{[S_i]}B$.

    If $t=0$, $E_{[TS_i]}(0) = 1 \in B$.

    If $t \neq 0$, 
    \begin{align*}
        E_{[TS_i]}(t) &= E_{[S_i]}(-t^{-1}) n_{[S_i]}(t^{-1}) E_{[S_i]}(-t^{-1}) \\ 
              &= E_{[S_i]}(-t^{-1}) h_{[S_i]}(t^{-1}) n_{[S_i]}  E_{[S_i]}(-t^{-1}) \in Bn_{[S_i]}B.
    \end{align*}
\end{proof}

For any $n\in N$, consider $\overline{n}\in N/H$, we define its length
\begin{align*}
    l(\overline{n}) = min\{ t \geq 0 | \overline{n} = \overline{n_{[S_{i_1}]}} \cdots \overline{n_{[S_{i_t}]}}, S_{i_1},\cdots,S_{i_t} \text{ are simple objects in } \mathcal{B} \}.
\end{align*}
A reduced expression of $\overline{n}$ is a sequence $(i_1,\cdots,i_t)$ such that $\overline{n} = \overline{n_{[S_{i_1}]}} \cdots \overline{n_{[S_{i_t}]}}$ and $t=l(\overline{n})$.

Let  $(i_1,\cdots,i_t)$ be a reduced expression of $\overline{n}$, define the set
\begin{align*}
    R(\overline{n}) = \{  [M] | M\in \operatorname{ind}\mathcal{B}, \omega_{S_{i_1}}\cdots \omega_{S_{i_t}}(M) \in \operatorname{ind}T\mathcal{B}    \}.
\end{align*}

Denote $\alpha_i = \underline{dim} S_i$ for each simple object $S_i$ in $\mathcal{B}$.
Since $N/H\cong W$, let $w\in W$ corresponds to $\overline{n}$, then a reduced expression $\overline{n} = \overline{n_{[S_{i_1}]}} \cdots \overline{n_{[S_{i_t}]}}$ is equivalent to a reduced expression $w=\omega_{\alpha_{i_1}}\cdots \omega_{\alpha_{i_t}}$.
We have 
\begin{align*}
    \underline{dim} \omega_{S_{i_1}} \cdots \omega_{S_{i_t}}(M) = \omega_{\alpha_{i_1}}\cdots \omega_{\alpha_{i_t}}(\underline{dim}M) = w(\underline{dim}M).
\end{align*}
Thus $[\omega_{S_{i_1}} \cdots \omega_{S_{i_t}}(M)]$ is determined by $\overline{n}$ and $M$, regardless of the reduced expression we have chosen. 
So the set $R(\overline{n})$ is well-defined.

\begin{lemma}\label{Rn}\cite[Lemma 8.3.2(i)]{Springer}
    For any $n\in N$ and simple object $S_i$ in $\mathcal{B}$, we have the following properties.

    (1) $R(\overline{n_{[S_i]}}) = \{ [S_i] \}$.

    (2)Let  $(i_1,\cdots,i_t)$ be a reduced expression of $\overline{n}$.
     \begin{equation*}
        R(\overline{n}\overline{n_{[S_i]}}) = 
        \begin{cases}
            \{ [S_i] \} \cup \omega_{S_i}(R(\overline{n}))   \qquad &\text{ if } \omega_{S_{i_1}} \cdots \omega_{S_{i_t}}(S_i) \in \operatorname{ind}\mathcal{B} \\
            \omega_{S_i}(R(\overline{n})\setminus \{[S_i]\})  \qquad \quad &\text{ if } \omega_{S_{i_1}} \cdots \omega_{S_{i_t}}(S_i) \in \operatorname{ind}T\mathcal{B}
        \end{cases}
        .
    \end{equation*}

    (3)$R(\overline{n}) = \{ [S_{i_t}], [\omega_{S_{i_t}}(S_{i_{t-1}})], \cdots,[\omega_{S_{i_t}} \cdots \omega_{S_{i_2}}(S_{i_1})]   \}$.
\end{lemma}

\begin{proof}
    (1),(2) follows from definition.

    We prove (3) by induction on $l=l(\overline{n})$.

    If $l=1$, it follows from (1).

    If $l>1$, let $\overline{n'} = \overline{n_{[S_{i_1}]}} \cdots \overline{n_{[S_{i_{t-1}}]}}$. 
    By induction hypothesis, we have 
    \begin{align*}
        R(\overline{n'}) = \{ [S_{i_{t-1}}], [\omega_{S_{i_{t-1}}}(S_{i_{t-2}})], \cdots,[\omega_{S_{i_{t-1}}} \cdots \omega_{S_{i_2}}(S_{i_1})]   \}.
    \end{align*}
    
    If $\omega_{S_{i_1}} \cdots \omega_{S_{i_{t-1}}}(S_{i_t}) \in \operatorname{ind}\mathcal{B}$, by (2), we're done.
    
    Otherwise, $\omega_{S_{i_1}} \cdots \omega_{S_{i_{t-1}}}(S_{i_t}) \in \operatorname{ind}T\mathcal{B} $.
    By definition, we have $[S_{i_t}]\in R(\overline{n'})$.
    Thus $[S_{i_t}] = [\omega_{S_{i_{t-1}}} \cdots \omega_{S_{i_{j+1}}}(S_{i_j})] $ for some $j$.
    Then 
    \begin{align*}
         \overline{n_{[S_{i_{t-1}}]}} \cdots \overline{n_{[S_{i_{j+1}}]}} \overline{n_{[S_{i_{j}}]}} \overline{n_{[S_{i_{j+1}}]}} \cdots \overline{n_{[S_{i_{t-1}}]}} = \overline{n_{[ \omega_{S_{i_{t-1}}} \cdots \omega_{S_{i_{j+1}}}(S_{i_j}) ]}} = \overline{n_{[S_{i_t}]}}.
    \end{align*}
    Hence 
    \begin{align*}
        \overline{n}& = \overline{n_{[S_{i_1}]}} \cdots \overline{n_{[S_{i_t}]}} \\
             &= \overline{n_{[S_{i_1}]}} \cdots \overline{n_{[S_{i_{t-1}}]}} (  \overline{n_{[S_{i_{t-1}}]}} \cdots \overline{n_{[S_{i_{j}}]}} \cdots \overline{n_{[S_{i_{t-1}}]}})\\
             &=  \overline{n_{[S_{i_1}]}} \cdots \overline{n_{[S_{i_{j-1}}]}} \overline{n_{[S_{i_{j+1}}]}}  \cdots \overline{n_{[S_{i_{t-1}}]}}.
    \end{align*}
    This contradicts to $l(\overline{n}) = t$, thus this case doesn't exist.
    The Lemma is proved.
\end{proof}

\begin{lemma}\label{lastj}\cite[Lemma 8.3.2(iv)]{Springer}
    For any simple object $S_j$ in $\mathcal{B}$, any $\overline{n} \in N/H$ and let $(i_1,\cdots,i_t)$ be a reduced expression of $\overline{n}$.
    If $\omega_{S_{i_1}} \cdots \omega_{S_{i_{t}}}(S_j) \in \operatorname{ind}T\mathcal{B} $, then there exists a reduced expression of  $\overline{n}$ with last element $j$.
\end{lemma}

\begin{proof}
    This proof is similar to the proof of Lemma \ref{Rn}(3).
    
    Since $\omega_{S_{i_1}} \cdots \omega_{S_{i_{t}}}(S_j) \in \operatorname{ind}T\mathcal{B} $, we have by definition $[S_j] \in R(\overline{n})$.
    By Lemma \ref{Rn}(3), $[S_j] = [\omega_{S_{i_{t}}} \cdots \omega_{S_{i_{r+1}}}(S_{i_r})] $ for some $r$.
    Similarly, we have 
    \begin{align*}
        &\overline{n_{[S_{i_{t}}]}} \cdots  \overline{n_{[S_{i_{r}}]}} \cdots \overline{n_{[S_{i_{t}}]}}  = \overline{n_{[S_{j}]}} \\
        &\overline{n}\overline{n_{[S_j]}} = \overline{n_{[S_{i_1}]}} \cdots \overline{n_{[S_{i_{r-1}}]}} \overline{n_{[S_{i_{r+1}}]}}  \cdots \overline{n_{[S_{i_{t}}]}}\\
        &\overline{n} = \overline{n_{[S_{i_1}]}} \cdots \overline{n_{[S_{i_{r-1}}]}} \overline{n_{[S_{i_{r+1}}]}}  \cdots \overline{n_{[S_{i_{t}}]}} \overline{n_{[S_j]}}.
    \end{align*}
    Thus $(i_1,\cdots,i_{r-1},i_{r+1},\cdots,i_t,j)$ is a reduced expression of $\overline{n}$ as required.
\end{proof}

\begin{proposition}
    For all $n\in N$ and simple object $S_i$ in $\mathcal{B}$, we have 
    \begin{align*}
       & (BnB)(Bn_{[S_i]}B) \subseteq  Bnn_{[S_i]}B \cup BnB, \\
       & (Bn_{[S_i]}B)(BnB) \subseteq  Bn_{[S_i]}nB \cup BnB.
    \end{align*}
\end{proposition}

\begin{proof}
    We prove the first statement.
    For $n\in N$, there exists $M_1,\cdots, M_l \in \operatorname{ind}\mathcal{R}$ and $h\in H$ such that $n=n_{[M_1]}\cdots n_{[M_l]}h$.
    Then 
    \begin{align*}
        (BnB)(Bn_{[S_i]}B) &= B n_{[M_1]}\cdots n_{[M_l]} Bn_{[S_i]}B \\
               &= Bn_{[M_1]}\cdots n_{[M_l]} U n_{[S_i]}B.
    \end{align*}
    Let 
    \begin{align*}
        U_i=\prod_{\substack{M \in \operatorname{ind}\mathcal{B} \\ M\ncong S_i}} E_{[M]} ,
    \end{align*}
  the product being taken in any order.
    Then $U=E_{[S_i]}U_i$.
    \begin{align*}
        (BnB)(Bn_{[S_i]}B) &= Bn_{[M_1]}\cdots n_{[M_l]} E_{[S_i]}U_i n_{[S_i]}B\\
         &= B(n_{[M_1]}\cdots n_{[M_l]}E_{[S_i]}n_{[M_l]}^{-1} \cdots n_{[M_1]}^{-1} )(n_{[M_1]}\cdots n_{[M_l]}n_{[S_i]})(n_{[S_i]}^{-1}U_i n_{[S_i]})B.
    \end{align*}
    Apply the commutator formula, we have $E_{[S_i]}(t),E_{[TS_i]}(t)$ normalize $U_i$.
    
    Thus $n_{[S_i]}^{-1} U_i n_{[S_i]}\subseteq  B$ and
    \begin{align*}
        (BnB)(Bn_{[S_i]}B) = B(n_{[M_1]}\cdots n_{[M_l]}E_{[S_i]}n_{[M_l]}^{-1} \cdots n_{[M_1]}^{-1} ) (nn_{[S_i]})B.
    \end{align*}
    
    If $\omega_{M_1}\cdots\omega_{M_l}(S_i)\in \operatorname{ind}\mathcal{B}$, then
    \begin{align*}
        n_{[M_1]}\cdots n_{[M_l]}E_{[S_i]}n_{[M_l]}^{-1} \cdots n_{[M_1]}^{-1}  \subseteq  B.
    \end{align*}
    Thus $(BnB)(Bn_{[S_i]}B) = Bnn_{[S_i]}B$.

    If $\omega_{M_1}\cdots\omega_{M_l}(S_i)\in \operatorname{ind}T\mathcal{B}$. 
    By the same argument after the definition of $R(\overline{n})$, $[\omega_{M_1}\cdots\omega_{M_l}(S_i)]$ depends on $\overline{n}$ and $S_i$.
    Thus we may choose $M_1,\cdots, M_l $ corresponds to a reduced expression.
    Moreover, by Lemma \ref{lastj}, we may assume that $n=n_{[S_{i_1}]}\cdots n_{[S_{i_{l-1}}]} n_{[S_i]}h$ with $(i_1,\cdots,i_{l-1},i)$ being a reduced expression of $\overline{n}$.

    Let $n'=n_{[S_{i_1}]}\cdots n_{[S_{i_{l-1}}]} $, then $n=n'n_{[S_i]}h,n'=nn_{[S_i]}h'$ for some $h'\in H$.
    Since $\omega_{S_{i_1}}\cdots\omega_{S_{i_{l-1}}}(S_i)\in \operatorname{ind}\mathcal{B}$, we have proved that 
    \begin{align*}
        Bn'BBn_{[S_i]}B = Bn'n_{[S_i]}B = BnB.
    \end{align*}
    Thus
    \begin{align*}
        (BnB)(Bn_{[S_i]}B) &= (Bn'B)(Bn_{[S_i]}B)(Bn_{[S_i]}B) \\
        & \subseteq  Bn'B(B\cup Bn_{[S_i]}B)\\
        &= Bn'B \cup Bn'BBn_{[S_i]}B \\
        & = Bnn_{[S_i]}B \cup BnB,
    \end{align*}
    where the first inclusion follows from Lemma \ref{nBn}. 
    This finishes the proof of the first statement.

    The second statement follows easily from the first: 

    For any $ x \in (Bn_{[S_i]}B)(BnB)$, then 
    \begin{align*}
        x^{-1} \in Bn^{-1}BBn_{[S_i]}B \subseteq  Bn^{-1}n_{[S_i]}B \cup Bn^{-1}B.
    \end{align*}
    Thus $x\in Bn_{[S_i]}nB \cup BnB$.
\end{proof}

\begin{proposition}\label{GBNB}
    $G=BNB$.
\end{proposition}

\begin{proof}
    $BNB \subseteq  G$, obviously closed under inversion. 
    To show $BNB$ is a subgroup of $G$, we need to show that it's closed under multiplication.
    For any $n \in N$, like the proof of the previous proposition, we can write $n=n_{[S_{i_1}]}\cdots n_{[S_{i_l}]}h$ with $(i_1,\cdots,i_l)$ being a reduced expression of $\overline{n}$ and $h\in H$. 
    Then
    \begin{align*}
        (BNB)(BnB) &= BNBBn_{[S_{i_1}]}\cdots n_{[S_{i_l}]}B \\
            &\subseteq  BNBn_{[S_{i_2}]}\cdots n_{[S_{i_l}]}B \subseteq  \cdots \subseteq  BNB.
    \end{align*}
    Thus $(BNB)(BNB)\subseteq  BNB$ and $BNB$ is a subgroup of $G$.

    For $M \in \operatorname{ind}\mathcal{R}$ such that $M \in \operatorname{ind}\mathcal{B}$, we have $E_{[M]}(t) \in B$ and 
    \begin{align*}
        E_{[TM]}(t) = n_{[M]}E_{[M]}(t)n_{[M]}^{-1} \in BNBBNB \subseteq  BNB.
    \end{align*}
    Thus $BNB$ contains all the generators of $G$, hence $G=BNB$.
\end{proof}

\begin{lemma}
    We have 
    $U\cap N = \{1\}$ and $B\cap N = H$.
\end{lemma}

\begin{proof}
    Let $n\in U\cap N$. Then $n=n_{[S_{i_1}]}\cdots n_{[S_{i_l}]}h$ with $(i_1,\cdots,i_l)$ being a reduced expression of $\overline{n}$ and $h\in H$.
    Then for any $ M \in \operatorname{ind}\mathcal{R}$, we have
    \begin{align*}
        nE_{[M]}n^{-1} = E_{[\omega_{S_{i_1}}\cdots \omega_{S_{i_l} } (M) ]}.
    \end{align*}
    Hence for any $M \in \operatorname{ind}\mathcal{B}$, we have $\omega_{S_{i_1}}\cdots \omega_{S_{i_l} } (M)\in \operatorname{ind}\mathcal{B}$.
    By definition, we have $R(\overline{n}) =  \emptyset $.
    By Lemma \ref{Rn}, we have $l(\overline{n}) = |R(\overline{n}) |= 0$.
    Thus $n=h\in H$. 
    Since $U\cap H = \{1\}$, we have $n=1$ and $U\cap N=\{ 1\}$. Then
    \begin{align*}
        B\cap N = UH\cap N = (U\cap N)H=H.
    \end{align*}
\end{proof}

\begin{proposition}
    For $n,n'\in N$, $BnB=Bn'B$ if and only if $\overline{n}=\overline{n'} $ in $N/H$.
\end{proposition}

\begin{proof}
    If $\overline{n}=\overline{n'} $ in $N/H$, obviously $BnB=Bn'B$.
    
    Conversely, if $BnB=Bn'B$, we may assume that $l(\overline{n})\leq l(\overline{n'})$ and prove by induction on $l(\overline{n})$.

    If $l(\overline{n})=0$, then $\overline{n} = 1$ and $n\in H$. 
    Thus $Bn'B=B$ and $n'\in B\cap N = H$. 
    Hence we have $\overline{n'} = 1 = \overline{n}$.

    If $l(\overline{n})>0$, then we can write $\overline{n}=\overline{n_{[S_i]}}\overline{n''}$ for some $i$ and $n''=n_{[S_i]}^{-1}n$. 
    We have $l(\overline{n''}) = l(\overline{n}) - 1$, $Bn_{[S_i]}n''B = BnB = Bn'B$.
    Thus
    \begin{align*}
        Bn''B\subseteq  Bn_{[S_i]}Bn'B \subseteq  Bn_{[S_i]}n'B \cup Bn'B.
    \end{align*}
    Hence $n''\in Bn_{[S_i]}n'B$ or $n''\in Bn'B$, which means 
    \begin{align*}
        Bn''B\subseteq  Bn_{[S_i]}n'B \quad \text{  or  } \quad Bn''B\subseteq  Bn'B.
    \end{align*}
    Since they are all double cosets, we have $ Bn''B= Bn_{[S_i]}n'B$ or $Bn''B= Bn'B$.
    By induction hypothesis, $\overline{n''}=\overline{n_{[S_i]}}\overline{n'}$ or $\overline{n''}=\overline{n'}$.
    The later case is impossible because $l(\overline{n''})=l(\overline{n})-1 < l(\overline{n'})$.
    Thus $\overline{n'} = \overline{n_{[S_i]}}\overline{n''} = \overline{n}$.
\end{proof}

As a corollary, we obtain the Bruhat decomposition.

\begin{corollary}
    (Bruhat decomposition)
    \begin{align*}
         G = \bigsqcup\limits_{\overline{n}\in N/H} BnB .
    \end{align*}
   
\end{corollary}

For any $ n\in N$, let $(i_1,\cdots,i_t)$ be a reduced expression of $\overline{n}$. 
Define 
\begin{align*}
    R'(\overline{n}) = \{  [M] | M\in \operatorname{ind}\mathcal{B}, \omega_{S_{i_1}}\cdots \omega_{S_{i_t}}(M) \in \operatorname{ind}\mathcal{B}   \}
\end{align*}
and
\begin{align*}
    U_{\overline{n}}^+ = \prod\limits_{M\in  R'(\overline{n})} E_{[M]}, \qquad  U_{\overline{n}}^- = \prod\limits_{M\in  R(\overline{n})} E_{[M]},
\end{align*}
where the products are taken in any order.

\begin{lemma}
    For any $n\in N$, we have $U=U_{\overline{n}}^+ U_{\overline{n}}^- $.
\end{lemma}

\begin{proof}
    Clearly $U_{\overline{n}}^+ \cap U_{\overline{n}}^- = 1$ and $U_{\overline{n}}^+ ,U_{\overline{n}}^-$ are both subgroups of $U$.
    Recall that for integers $ t\geq 1$, we have defined $U_t = \prod\limits_{l(M)\geq t} E_{[M]}$. 
    Let $U_t^{\pm} = U_t \cap U_{\overline{n}}^{\pm}$.

    We claim that $U_t=U_t^+ U_t^-$ holds for all integers $t\geq 1$, and prove by descending induction on $t$.
    For $t$ sufficiently large, it's clear.
    Assume the claim holds for $t+1$, i.e. $U_{t+1} = U_{t+1}^+ U_{t+1}^-$.
    Since $U_{t+1}$ is a normal subgroup of $U_t$ and $U_t/U_{t+1}$ is abelian, we have
    \begin{align*}
        U_t = U_t^+ U_t^- U_{t+1} = U_t^+ U_{t+1}  U_t^- = U_t^+ U_{t+1}^+ U_{t+1}^- U_t^- = U_t^+ U_t^-.
    \end{align*}
    Take $t=1$, we have $U=U_{\overline{n}}^+ U_{\overline{n}}^- $.
\end{proof}

\begin{lemma}
    For any $M \in \operatorname{ind}T\mathcal{B}$, we have $E_{[M]}\cap B = \{1\}$.
\end{lemma}

\begin{proof}
    Since $U$ is a normal subgroup of $B$, $B= UH$ and $U\cap H = \{1\}$, we have isomorphism
    \begin{align*}
        UH/U \cong H/(U\cap H) \cong H.
    \end{align*}
    
    If there exists $t\in \mathbb{K}$ such that $E_{[M]}(t) \in B$, then the image of $E_{[M]}(t)$ under the canonical homomorphism $B\twoheadrightarrow B/U$ should be commutative with any $h\in H$.
    Thus for any $ h_{[X]}(s) \in H$, we have
    \begin{align*}
        h_{[X]}(s) E_{[M]}(t) h_{[X]}(s)^{-1} E_{[M]}(t)^{-1}  \in U,
    \end{align*}
    i.e.
    \begin{align*}
        E_{[M]}(s^{A_{XM}}t-t) \in U, \qquad  X\in \operatorname{ind}\mathcal{R},  s\in \mathbb{K}^{\times}.
    \end{align*}
    
    If $\mathbb{K}\neq \mathbb{F}_2$ and $t\neq 0$, we can take $X,s$ such that $s^{A_{XM}} \neq 1$.
    Then $M \in \operatorname{ind}T\mathcal{B}$ leads to a contradiction.
    Thus $t=0$ and $E_{[M]}\cap B = \{1\}$.

    For $\mathbb{K} = \mathbb{F}_2$, we have $H=1$ and $B=U$. 
    Hence $E_{[M]}\cap B \subseteq  V\cap U =\{1\}$.
\end{proof}

\begin{proposition}
    Arbitrarily fix an $n\in N$, any $x\in BnB=BnU_{\overline{n}}^-$ can be uniquely expressed as $bnu$ for some $b\in B,u\in U_{\overline{n}}^-$.
\end{proposition}

\begin{proof}
    Since $nU_{\overline{n}}^+ n^{-1} \subseteq  U$, we have
    \begin{align*}
        BnB = BnHU_{\overline{n}}^+U_{\overline{n}}^- = BnU_{\overline{n}}^+U_{\overline{n}}^- \subseteq  BnU_{\overline{n}}^-.
    \end{align*}
    On the other hand, $BnU_{\overline{n}}^- \subseteq  BnB$ is obvious, thus 
    \begin{align*}
        BnB = BnU_{\overline{n}}^-,
    \end{align*}
    and each $x\in BnB$ can be expressed as $bnu$ for some $b\in B,u\in U_{\overline{n}}^-$. 

    Now we prove the uniqueness by induction on $l(\overline{n})$.
    
    \textbf{Case} $l(\overline{n})=0$.  Then $R(\overline{n}) = \emptyset$ and $BnB = B$. There's nothing to prove. 

    \textbf{Case} $l(\overline{n})=1$.  Then $n=hn_{[S_i]}$ for some simple object $S_i$ and $h\in H$. 
    $R(\overline{n}) = \{ [S_i] \}$. 
    If $b_1 n_{[S_i]} u_1 = b_2 n_{[S_i]} u_2$, then 
    \begin{align*}
        b_2^{-1}b_1 = n_{[S_i]} u_2u_1^{-1} n_{[S_i]}^{-1}.
    \end{align*}
    Since $U_{\overline{n}}^- = E_{[S_i]}$, we have 
    \begin{align*}
        b_2^{-1}b_1 = n_{[S_i]} u_2u_1^{-1} n_{[S_i]}^{-1} \in E_{[TS_i]} \cap B = \{1\}.
    \end{align*}
    Thus $b_1=b_2,u_1=u_2$.

    \textbf{Case} $l(\overline{n})>1$.  Then write $n=n'n_{[S_i]}$ with $l(\overline{n'}) = l(\overline{n}) - 1$.
    Suppose $( i_1,\cdots,i_{l-1} )$ is a reduced expression of $\overline{n'}$.
    Since 
    \begin{align*}
      R(\overline{n}) = \{ [S_{i}], [\omega_{S_{i}}(S_{i_{l-1}})], \cdots,[\omega_{S_{i}} \cdots \omega_{S_{i_2}}(S_{i_1})]   \},
    \end{align*}
    we have 
    \begin{align*}
        U_{\overline{n}}^- = n_{[S_i]}^{-1} U_{\overline{n'}}^- n_{[S_i]} E_{[S_i]} 
    \end{align*}
    and then 
    \begin{align*}
        BnU_{\overline{n}}^- = Bn'n_{[S_i]} n_{[S_i]}^{-1} U_{\overline{n'}}^- n_{[S_i]} E_{[S_i]} = Bn' U_{\overline{n'}}^- n_{[S_i]} E_{[S_i]}.
    \end{align*}
    If $b_1nu_1 = b_2nu_2$, then $b_1nu_1n_{[S_i]}^{-1} = b_2nu_2n_{[S_i]}^{-1}$. 
    Hence
    \begin{align*}
        b_1n'u_1'E_{[TS_i]}(t_1) = b_2n'u_2' E_{[TS_i]}(t_2) ,
    \end{align*}
    for some $u_1',u_2' \in U_{\overline{n'}}^- $, i.e. $b_1n'u_1'E_{[TS_i]}(t_1-t_2) = b_2n'u_2' $.

    Recall the proof of Lemma \ref{nBn}.
    
    If $t_1-t_2 = 0$, then $E_{[TS_i]}(t_1-t_2) = 1\in B$.

    If $t_1-t_2 \neq 0$, then $E_{[TS_i]}(t_1-t_2) \in Bn_{[S_i]}B$ and 
    \begin{align*}
        b_1n'u_1'E_{[TS_i]}(t_1-t_2) \in Bn'BBn_{[S_i]}B.
    \end{align*} 
    Since $\omega_{S_{i_1}} \cdots \omega_{S_{i_{l-1}}}(S_{i}) \in \operatorname{ind}\mathcal{B} $, 
    \begin{align*}
        Bn'BBn_{[S_i]}B = Bn'n_{[S_i]}B = BnB.
    \end{align*}

    Hence $  b_1n'u_1'E_{[TS_i]}(t_1-t_2) \in BnB$, while $b_2n'u_2' \in Bn'B$.
    Since $\overline{n}\neq \overline{n'}$, $BnB\cap Bn'B = \emptyset$, contradiction.
    Hence $t_1 = t_2$ and $b_1n'u_1' = b_2n'u_2'$. 
    By induction hypothesis, $b_1=b_2,u_1'=u_2'$. Hence $u_1=u_2$.
\end{proof}

\begin{proposition}
    Any $x\in G$ has a unique expression of the form $x=u'hnu$ with $u'\in U,h\in H,n\in N,u\in U_{\overline{n}}^- $.
\end{proposition}

\begin{proof}
    For any $ x\in G$, by Bruhat decomposition, there exists a unique $\overline{n} \in N/H$ such that $x\in BnB=UHnU_{\overline{n}}^-$.
    By the previous proposition and $B=UH,U\cap H=1$, we get the desired result.
\end{proof}

\begin{corollary}
    $B\cap V = \{1\}$.
\end{corollary}

\begin{proof}
    There exists $n_0\in N$ such that 
    \begin{align*}
        R(\overline{n_0}) = \{ [M]| M \in \operatorname{ind}\mathcal{B} \}.
    \end{align*} 
    Then we have $n_0U_{\overline{n_0}}^- n_0^{-1} = V$.
    If $x\in B\cap V = B\cap n_0U_{\overline{n_0}}^- n_0^{-1}$, then $xn_0\in Bn_0\cap n_0U_{\overline{n_0}}^- $.
    Since $xn_01=1n_0u$ for some $u\in U_{\overline{n_0}}^- $ are two expressions of an element in $Bn_0U_{\overline{n_0}}^- $, we have $x=1$ and $u=1$.
    Thus  $B\cap V = \{1\}$.
\end{proof}

\subsection{Parabolic Subgroups and Levi Subgroups}

In this subsection, we define some subgroups of $G$ arising from the category, and prove that they coincide with the parabolic subgroups or Levi subgroups of the Chevalley group $G$.
Moreover, we show that some of the properties of parabolic subgroups or Levi subgroups can be deduced from properties of the category.

Again, fix a complete section of $\mathcal{R}$, and hence also a hereditary subcategory $\mathcal{B}$ of $\mathcal{R}$.
Let $\Pi = \{  \alpha_1,\cdots,\alpha_m \}$ be the set of simple roots in $\Phi$, 
and fix representatives $S_1,\cdots,S_m$ of isomorphism classes of simple objects in $\mathcal{B}$, such that $\underline{dim}S_i = \alpha_i$.

For any subset $J\subseteq  \Pi$, write $J = \{ \alpha_{i_1}, \cdots, \alpha_{i_l} \}$.
Let $W_J$ be the subgroup of $W$ generated by $\omega_i,i\in J$.
Let $N_J$ be the subgroup of $N$ mapping to $W_J$ under the homomorphism $N \rightarrow N/H \cong W$. 
Then similar to the proof of Prop \ref{GBNB}, we have $BN_J B$ is a subgroup of $G$.
These $BN_J B$ with various subsets $J\subseteq  \Pi$ are so called parabolic subgroups of $G$.

Let $\mathcal{F}_J$ be the smallest extension closed full subcategory of $\mathcal{B}$, containing $[S_{i_1}],\cdots, [S_{i_l}]$.
Let $P_J$ be the subgroup of $G$ generated by $E_{[M]}(t)$, $ t\in \mathbb{K},  M\in \operatorname{ind}T\mathcal{F}_J$ and $B$.

We will show that this subgroup is a parabolic subgroup.

\begin{proposition}\label{PJBNB}
    For any subset $J\subseteq \Pi$, we have $P_J = BN_J B$.
\end{proposition}

To prove this proposition, we need the following lemma.

\begin{lemma}\label{lemmaPJBNB}
    For $M\in \operatorname{ind}\mathcal{R}$, $\underline{dim}M $ is supported on $J= \{ \alpha_{i_1}, \cdots, \alpha_{i_l} \}\subseteq \Pi$ if and only if $\omega_{\underline{dim}M}$ is generated by $\omega_{\alpha},\alpha \in J$.
\end{lemma}

\begin{proof}
    If $\underline{dim}M $ is supported on $J= \{ \alpha_{i_1}, \cdots, \alpha_{i_l} \}\subseteq \Pi$,
    we write  $r = \underline{dim}M  = \sum\limits_{j=1}^l a_j\alpha_{i_j}$. 
    If $r=0$, there's nothing to prove. For $r \neq 0$,
    we may assume that  $M \in \operatorname{ind}\mathcal{B}$, i.e. $a_j \geq 0$ for all $ j$.
    The proof for $M \in \operatorname{ind}T\mathcal{B}$ is similar.

    We claim that there exists $t\in \{ i_1,\cdots, i_l \}$ such that 
    \begin{align*}
        \omega_{\alpha_t}(r)-r \in \mathbb{Z}_{\leq 0}\Pi,
    \end{align*}
    which means $ \omega_{\alpha_t}(r)-r $ can be written in the form of $\mathbb{Z}_{\leq 0}$-linear combination of elements of $\Pi$.
    If the claim fails, then 
    \begin{align*}
        \omega_{\alpha_t}(r) = r-<r,\alpha_i>\alpha_i \geq r, \qquad  i\in \{ i_1,\cdots,i_l \}.
    \end{align*}
    Hence $<r,\alpha_i> \leq 0$, for all $i\in \{ i_1,\cdots,i_l \}$.
    Then 
    \begin{align*}
        <r,r> = \sum\limits_{j=1}^l <r,\alpha_{i_j}>a_j \leq 0,
    \end{align*}
    contradiction. 
    Hence the claim holds, i.e. there exists $t\in \{ i_1,\cdots, i_l \}$ such that $\omega_{\alpha_t}(r)\leq r$.
    
    Since $\omega_{\alpha_t} \omega_r \omega_{\alpha_t}^{-1} = \omega_{\omega_{\alpha_t}(r)}$, we have 
    \begin{align*}
        \omega_r = \omega_{\alpha_t}^{-1} \omega_{\omega_{\alpha_t}(r)} \omega_{\alpha_t} .
    \end{align*}
    By induction on the order of $\Phi$, we have $ \omega_{\omega_{\alpha_t}(r)} $ is generated by  $\omega_{\alpha},\alpha \in J$.
    Hence $\omega_r$ is generated by  $\omega_{\alpha},\alpha \in J$.

    Conversely, suppose $\omega_{\underline{dim}M}=\omega_{\alpha_1}\cdots \omega_{\alpha_t}$ for some $\alpha_1,\cdots,\alpha_t \in J\subseteq \Pi$.
    There exists $X\in \mathcal{R}$ such that $(H_X|H_M)\neq 0$.
    For such $X$, on the one hand, we have 
    \begin{align*}
        \omega_{\underline{dim}M}(\underline{dim}X) - \underline{dim}X = -A_{MX} \underline{dim}M 
    \end{align*}
    with $A_{MX}\neq 0$.
    On the other hand,
        $\omega_{\alpha_1}\cdots \omega_{\alpha_t}(\underline{dim}X)-\underline{dim}X$
    is a linear combination of $\alpha_1,\cdots,\alpha_t$.
    Thus $\underline{dim}M$ can be expressed as a linear combination of some elements in $J$, i.e. $\underline{dim}M$ is supported on $J$. 
\end{proof}

Now we prove Prop \ref{PJBNB}.

\begin{proof}
    On the one hand, since $E_{[M]}(t)\in P_J$ for each $M\in \operatorname{ind}T\mathcal{F}_J$ and $n_{[X]}(t) = E_{[X]}(t) E_{[TX]}(t^{-1}) E_{[X]}(t)$, we have 
    \begin{align*}
        n_{[M]}(t) \in P_J, \qquad \text{for any }  M\in \operatorname{ind}\mathcal{F}_J.
    \end{align*}
    In particular, $n_{[S_{i_1}]},\cdots,n_{[S_{i_l}]} \in P_J$.
    Hence 
    \begin{align*}
        BN_JB = <B,n_{[S_{i_1}]},\cdots,n_{[S_{i_l}]} > \subseteq  P_J,
    \end{align*}
    where we use the notation $<a,b,\cdots>$ to denote the subgroup of $G$ generated by $a,b,\cdots$.   

    On the other hand, for any $M\in \operatorname{ind}\mathcal{F}_J$, $\underline{dim}M$ is supported on $J$.
    By Lemma \ref{lemmaPJBNB}, $\omega_{\underline{dim}M}$ is generated by  $\omega_{\alpha},\alpha \in J$.
    Hence $n_{[M]}$ is generated by $n_{[S_{i_1}]},\cdots,n_{[S_{i_l}]}$ and some $h\in H$, which means $n_{[M]}\in BN_JB$.
    Also $n_{[M]}(t) = h_{[M]}(t)n_{[M]}\in BN_JB,  t\in \mathbb{K}^{\times}$.
    Thus
    \begin{align*}
        E_{[TM]}(t) = E_{[M]}(-t^{-1})n_{[M]}(t^{-1})E_{[M]}(-t^{-1}) \in BN_JB, \qquad t\in \mathbb{K}^{\times},
    \end{align*}
    i.e. $E_{[X]}(t) \in BN_JB$ for all $X\in \operatorname{ind}T\mathcal{F}_J$ and all $t\in \mathbb{K}$.
    Since $B\subseteq  BN_JB$, we have $P_J\subseteq  BN_JB$.

    Hence we have $P_J = BN_JB$.
\end{proof}

Now we use this definition to show some properties of the parabolic subgroups.

\begin{lemma}\label{nBnPJ}
    For any $n\in N$, let $\overline{n} = \overline{n_{[S_{i_1}]}} \cdots \overline{n_{[S_{i_l}]}}$, $(i_1,\cdots,i_l)$ being a reduced expression of $\overline{n}$.
    Let $J = \{ \alpha_{i_1}, \cdots, \alpha_{i_l} \}\subseteq \Pi$.
    Then 
    \begin{align*}
        <B,n> = <B,nBn^{-1}> = P_J.
    \end{align*}
\end{lemma}

\begin{proof}
    Obviously, $<B,nBn^{-1}> \subseteq  <B,n> \subseteq  P_J$.
    
    Consider 
    \begin{align*}
        nBn^{-1} = nHU_{\overline{n}}^+ U_{\overline{n}}^- n^{-1} = H(nU_{\overline{n}}^+ n^{-1})(nU_{\overline{n}}^- n^{-1}).
    \end{align*}
    Firstly we have $ H(nU_{\overline{n}}^+ n^{-1})\subseteq  B$.
    Recall that $ U_{\overline{n}}^- = \prod\limits_{M\in R(\overline{n})} E_{[M]} $, 
    where $R(\overline{n}) = \{ [S_{i_l}], [\omega_{S_{i_l}}(S_{i_{l-1}})], \cdots,[\omega_{S_{i_l}} \cdots \omega_{S_{i_2}}(S_{i_1})]   \}$.
    For any $2\leq t \leq l$, we have
    \begin{align*}
        nE_{[\omega_{S_{i_l}} \cdots \omega_{S_{i_t}}(S_{i_{t-1}})]} n^{-1} &=  E_{[\omega_{S_{i_1}} \cdots \omega_{S_{i_l}} \omega_{S_{i_l}} \cdots \omega_{S_{i_t}}(S_{i_{t-1}})   ]} \\
          &= E_{[\omega_{S_{i_1}} \cdots \omega_{S_{i_{t-1}}}(S_{i_{t-1}})]} \\
          &= E_{[T\omega_{S_{i_1}} \cdots \omega_{S_{i_{t-2}}}(S_{i_{t-1}})]} \in V.
    \end{align*}
    Hence $nU_{\overline{n}}^- n^{-1} = \prod\limits_{X} E_{[X]} $, where $X$ runs through the set 
    \begin{align*}
        T \{ [\omega_{S_{i_1}} \cdots \omega_{S_{i_{l-1}}}(S_{i_{l}})], [\omega_{S_{i_1}} \cdots \omega_{S_{i_{l-2}}}(S_{i_{l-1}})] ,\cdots, [S_{i_1}] \}.
    \end{align*}
    Thus we have $E_{[TS_{i_1}]} \subseteq  <B,nBn^{-1}>$.
    Since $E_{[S_{i_1}]} \subseteq  <B,nBn^{-1}>$, we have
    \begin{align*}
        n_{[S_{i_1}]}(t) = E_{[S_{i_1}]}(t) E_{[TS_{i_1}]}(t^{-1}) E_{[S_{i_1}]}(t) \in <B,nBn^{-1}>, \qquad  t\in \mathbb{K}^{\times}.
    \end{align*}
    
    Since $E_{[T\omega_{S_{i_1}}(S_{i_2})]} \subseteq  <B,nBn^{-1}>$, 
    \begin{align*}
        E_{[TS_{i_2}]} = n_{[S_{i_1}]} E_{[T\omega_{S_{i_1}}(S_{i_2})]} n_{[S_{i_1}]}^{-1} \subseteq  <B,nBn^{-1}>.
    \end{align*}
    Hence $n_{[S_{i_2}]}\in <B,nBn^{-1}>$.

    Then consider $E_{[T\omega_{S_{i_1}} \omega_{S_{i_2}} (S_{i_3})]} \subseteq  <B,nBn^{-1}>$.
    Conjugate by $n_{[S_{i_2}]}n_{[S_{i_1}]}$, we have $E_{[TS_{i_3}]} \subseteq  <B,nBn^{-1}>$.
    Repeat this process, we have
    \begin{align*}
    E_{[TS_{i_1}]},\cdots,E_{[TS_{i_l}]}\subseteq  <B,nBn^{-1}>
    \end{align*} 
    and
    \begin{align*}
        n_{[S_{i_1}]},\cdots,n_{[S_{i_l}]}\in <B,nBn^{-1}>.
    \end{align*}
    Hence 
    \begin{align*}
        P_J = BN_JB = <B,n_{[S_{i_1}]},\cdots,n_{[S_{i_l}]}> \subseteq  <B,nBn^{-1}>.
    \end{align*}
    Thus $<B,n> = <B,nBn^{-1}> = P_J$.
\end{proof}

\begin{proposition}\label{PJonly}
    $P_J,J\subseteq \Pi$ are the only subgroups of $G$ containing $B$.
\end{proposition}

\begin{proof}
    Let $P$ be a subgroup of $G$ containing $B$. 
    By Bruhat decomposition, we have 
    \begin{align*}
        P=\bigcup\limits_{n_{\alpha}\in N} Bn_{\alpha}B
    \end{align*}
    for a family of $n_{\alpha} \in N$.
    Hence $P$ is generated by $B$ and a subset of elements of $N$.
    By Lemma\ref{nBnPJ}, we have $<B,n_{\alpha}> = P_{J_{\alpha}}$ for some subset $J_{\alpha} \subseteq  \Pi$.
    Thus $P$ is generated by $P_{J_{\alpha}}$ for a family of  $J_{\alpha} \subseteq  \Pi$, 
    which means $P=P_J$ for $J=\bigcup\limits_{\alpha} J_{\alpha}$.
\end{proof}

\begin{proposition}\label{PJproperties}
    For $J,K\subseteq \Pi$, we have the following properties.

   (1) $P_J = N_G(P_J)$, i.e. $P_J$ is its own normalizer in $G$. 

   (2) Distinct $P_J,P_K$ can't be conjugate in $G$.

   (3) $P_J \cap P_K = P_{J\cap K}$.

   (4) $P_J$ for distinct $J\subseteq  \Pi $ are all distinct.
\end{proposition}

\begin{proof}
    (1) Since $N_G(P_J) \supset P_J \supset B$, $N_G(P_J)$ is generated by $B$ and some elements of $N$.
    Let $n\in N\cap N_G(P_J)$, then 
    \begin{align*}
        <B,n> = <B,nBn^{-1}> \subseteq  P_J.
    \end{align*}
    Hence $n\in P_J$ and $P_J = N_G(P_J)$.

    (2) Suppose $xP_Jx^{-1} = P_K$, where $x=bnb'$ for some $b,b'\in B$ and $n\in N$.
    Then $nP_Jn^{-1} = P_K$.
    Hence $<B,n> = <B,nBn^{-1}> \subseteq  P_K$, which means $n\in P_K$ and $P_J=P_K$.

    (3) $P_J \cap P_K $ is a subgroup of $G$ and it contains $B$. 
    Hence $P_J \cap P_K = P_L$ for some $L\subseteq  \Pi$ and obviously $ P_{J\cap K} \subseteq  P_L$.
    On the other hand, since $P_L\subseteq  P_J$, i.e. $BN_LB\subseteq  BN_JB$, we have $N_L\subseteq  N_J, W_L \subseteq  W_J$.
    Thus $L\subseteq  J$.
    Similarly, $L\subseteq  K$. 
    Hence $L\subseteq  J\cap K$ and $P_L \subseteq  P_{J\cap K}$.
    Thus 
    \begin{align*}
        P_J \cap P_K = P_L = P_{J\cap K}.
    \end{align*}

    (4) Suppose $P_J=P_K$, then $N_J=N_K,W_J=W_K$. Thus $J=K$.
\end{proof}

For any $J\subseteq \Pi$, we define $U_J$ be the subgroup of $G$ generated by $E_{[M]}(t)$ for all $t\in \mathbb{K}, M\in \operatorname{ind}\mathcal{B}\setminus \operatorname{ind}\mathcal{F}_J$, 
and $L_J$ be the subgroup of $G$ generated by $H$ and $E_{[M]}(t)$ for all $t\in \mathbb{K}, M\in \operatorname{ind}\mathcal{F}_J \cup \operatorname{ind}T\mathcal{F}_J$.

\begin{proposition}
    For any $J\subseteq \Pi$, we have the following properties.

    (1) $U_J$ is a normal subgroup of $P_J$.

    (2) $P_J = U_J L_J, U_J \cap L_J = \{1\}$.

    (3) $P_J = N_G(U_J)$.

\end{proposition}

\begin{proof}
    (1) Since $\mathcal{F}_J$ is extension-closed, the extension of any $ M \in \operatorname{ind}\mathcal{F}_J \cup \operatorname{ind}T\mathcal{F}_J$ and any $ X \in \operatorname{ind}\mathcal{B}\setminus \operatorname{ind}\mathcal{F}_J$ is still in $ \operatorname{ind}\mathcal{B}\setminus \operatorname{ind}\mathcal{F}_J$ (if exists).
    Thus by checking the conjugation of generators of $P_J$, we have $U_J$ is a normal subgroup of $P_J$.

    (2) By definition, $P_J$ is generated by $U_J$ and $L_J$. 
    By (1), $P_J = U_J L_J$.

    Let $\theta\in U_J \cap L_J$. 
    Checking its action on $\mathbb{S}$, we easily deduce that $\theta = 1$ and the statement follows.

    (3) Since $N_G(U_J) \supset P_J \supset B$, $N_G(U_J) = P_{J'}$ for some $J'\supset J$.
    If $J'\neq J$, choose an element $i\in J'\setminus J$, then $E_{[S_i]}\subseteq  U_J$.
    Since $n_{[S_i]}\in P_{J'} = N_G(U_J)$, we have $n_{[S_i]}E_{[S_i]}n_{[S_i]}^{-1} = E_{[TS_i]}\subseteq  U_J$, contradiction.
    Hence $J'=J$ and $N_G(U_J) = P_{J}$.
\end{proof}

$L_J$ and its conjugates by $U_J$ are the Levi subgroups of $P_J$.

\subsection{Simplicity of G}

In this subsection, we want to show that $G$ is simple except for some special cases.

\begin{proposition}
    Let $G'$ denote the commutator subgroup of $G$. 
    Then $G=G'$ except for cases $A_1(2),A_1(3),B_2(2),G_2(2)$ ( the notation $A_1(2)$ means $\mathcal{R}$ is of type $A_1$ and the field $\mathbb{K}=\mathbb{F}_2$. Others are similar).
\end{proposition}

\begin{proof}
    Firstly, consider $\mathbb{K}\neq \mathbb{F}_2,\mathbb{F}_3$. 
    Since $h_{[M]}(t) E_{[M]}(u) h_{[M]}(t)^{-1} = E_{[M]}(t^2u)$, we have
    \begin{align*}
        (h_{[M]}(t),E_{[M]}(u)) = E_{[M]}(t^2u-u) \in G'.
    \end{align*}
    Since $\mathbb{K}\neq \mathbb{F}_2,\mathbb{F}_3$, there exists $t\in \mathbb{K}^{\times}$ such that $t^2 - 1 \neq 0$.
    Take $u=\frac{w}{t^2-1}$, and thus $E_{[M]}(w)\in G'$ for any $w\in \mathbb{K}$.
   Since $G'$ contains all the generators of $G$, we have $G=G'$.

   Secondly, consider $\mathbb{K}=\mathbb{F}_3$ and $\mathcal{R}$ is not of type $A_1$.
   For any $M\in \operatorname{ind}\mathcal{R}$, there exists $X\in \operatorname{ind}\mathcal{R}$ such that $M,X$ have extension.
   For example, we may take $X=\tau^{-1}M $, where $\tau$ is the AR-translation functor.
   By Lemma \ref{ABG}, the subcategory $\mathcal{C}$ of $\mathcal{R}$ generated by $M,X$ is of type $A_2,B_2$ or $G_2$.
   We consider these cases respectively.

   \textbf{Case} type $A_2$
   \begin{align*}
    (E_{[TX]}(t),E_{[L_{M,X,1,1}]}(s)) = E_{[M]}(ts\gamma_{TX,L_{M,X,1,1}}^{M} ) \in G'.
   \end{align*} 
   Since $\gamma_{TX,L_{M,X,1,1}}^{M} = \pm 1$, we have $E_{[M]}\subseteq  G'$.

   \textbf{Case} type $B_2$

   In the subcategory $\mathcal{C}$, for all objects $Y$, $d(Y)$ only has two possible values.
   If $d(M)$ is the larger one,
   \begin{align*}
    (E_{[TX]}(t),E_{[L_{M,X,1,1}]}(s)) = E_{[M]}(ts\gamma_{TX,L_{M,X,1,1}}^{M} ) \in G'.
   \end{align*}
   Since $\gamma_{TX,L_{M,X,1,1}}^{M} = \pm 2$, we have $E_{[M]}\subseteq  G'$.

   If $d(M)$ is the smaller one, 
   \begin{align*}
    (E_{[M]}(t),E_{[L_{M,X,1,1}]}(s)) = E_{[L_{M,X,2,1}]}(ts\gamma_{M,L_{M,X,1,1}}^{L_{M,X,2,1}} ) \in G'.
   \end{align*}
   Since $\gamma_{M,L_{M,X,1,1}}^{L_{M,X,2,1}} = \pm 2$, we have $E_{[L_{M,X,2,1}]}\subseteq  G'$.
   Since
   \begin{align*}
    &(E_{[L_{M,X,1,1}]}(a),E_{[TX]}(b)) \\
    =& E_{[M]}(ab\gamma_{L_{M,X,1,1},TX}^{M}) E_{[L_{M,X,2,1}]}(\frac{a^2b}{2!}\gamma_{L_{M,X,1,1},TX}^{M}\gamma_{L_{M,X,1,1},M}^{L_{M,X,2,1}} ) \in G'
   \end{align*}
   and $\gamma_{L_{M,X,1,1},TX}^{M} = \pm 1$, we have $E_{[M]}\subseteq  G'$.

   \textbf{Case} type $G_2$
   \begin{align*}
    (h_{[X]}(t),E_{[M]}(s)) = E_{[M]}((t^{A_{XM}}-1)s) \in G'.
   \end{align*}
   There exists $t$ such that $t^{A_{XM}}-1 \neq 0$, hence $E_{[M]}\subseteq  G'$.

   In all cases, $G'$ contains all the generators of $G$. Hence $G=G'$.

   Finally, consider $\mathbb{K}=\mathbb{F}_2$ and $\mathcal{R}$ is not of type $A_1,B_2,G_2$.
   Since
   \begin{align*}
    (n_{[X]}(t),E_{[Y]}(s)) = E_{[\omega_X(Y)]}(t^{-A_{XY}}s) E_{[Y]}(s)^{-1} \in G',
   \end{align*}
   i.e.
   \begin{align*}
    (n_{[X]}(1),E_{[Y]}(1)) = E_{[\omega_X(Y)]}(1) E_{[Y]}(1)^{-1} \in G',
   \end{align*}
   it suffice to show for a representative of each $W$-orbit of $\operatorname{ind}\mathcal{R}$.
   Notice that $(-|-)$ is $W$-invariant, so does $d(-)$. 
   On the other hand, since for any $M\in \operatorname{ind}\mathcal{R}$, there exists $Y\in \{ S_{1},\cdots,S_{m} \}$ and $i_1,\cdots,i_l$ such that $M\cong \omega_{S_{i_1}}\cdots\omega_{S_{i_l}}(Y)$,
   we only need to show $E_{[S_i]}\subseteq  G'$ for $i=1,\cdots,m$.
   It's easy to see that $d(S_1),\cdots,d(S_m)$ have at most two values.
   Hence $d(\operatorname{ind}\mathcal{R})$ have at most two values.

   We claim that if $d(X) = d(Y), X,Y\in \operatorname{ind}\mathcal{R}$, then $X,Y$ are in the same $W$-orbit.
   We only need to prove for the case when $X,Y$ are simple objects with $(H_X|H_Y)\neq 0$.
   Suppose $S_i,S_j$ are two simple objects satisfying $d(S_i)=d(S_j)$ and $(H_{S_i}|H_{S_j}) \neq 0$.
   Then 
   \begin{align*}
    A_{S_i,S_j} = \frac{(H_{S_i}|H_{S_j})}{d(S_i)} =  \frac{(H_{S_i}|H_{S_j})}{d(S_j)} = A_{S_j,S_i}.
   \end{align*}
   Since $A_{S_i,S_j},A_{S_j,S_i}$ are elements in the Cartan matrix, we must have $A_{S_i,S_j} = A_{S_j,S_i}=-1$.
   (Otherwise they can't be equal.)
   Then $\omega_{S_j}\omega_{S_i}(S_j)=S_i$, and $S_i,S_j$ are in the same $W$-orbit.
   The claim is proved.
   Thus we only need to show one object $M$ for each value of $d(\operatorname{ind}\mathcal{R})$, $E_{[M]}$ is in $G'$.

   If all $d(M)$ are equal, and since $\mathcal{R}$ is not of type $A_1$, there exist $X,Y\in \operatorname{ind}\mathcal{R}$ have extension.
   \begin{align*}
    (E_{[X]}(1),E_{[Y]}(1)) = E_{[L_{X,Y,1,1}]}(1) \in G'.
   \end{align*}
   Hence $G=G'$.

   If there're $d(X),d(Y)$ of two different values, we may assume that $d(X)<d(Y)$ and $X,Y$ have extension.
   Since $\mathcal{R}$ is not of type $G_2$, the subcategory of $\mathcal{R}$ generated by $X,Y$ is of type $B_2$. Hence
   \begin{align}
    (E_{[X]}(1),E_{[Y]}(1)) = E_{[L_{X,Y,1,1}]}(1) E_{[L_{X,Y,2,1}]}(1) \in G',
   \end{align}
   where $d(L_{X,Y,1,1})=d(X)$ and $d(L_{X,Y,2,1})=d(Y)$.
   Since $\mathcal{R}$ is not of type $B_2$, in the quiver corresponding to the complete section, there exist two vertices joined by one edge.
   Suppose these two vertex correspond to $S_i,S_j$. Then
   \begin{align*}
    (E_{[S_i]}(1),E_{[S_j]}(1)) = E_{[L_{S_i,S_j,1,1}]}(1) \in G'.
   \end{align*}
   Hence $E_{[M]}\subseteq  G'$ for all $M\in \operatorname{ind}\mathcal{R}$ such that $d(M) = d(L_{S_i,S_j,1,1})$.
   By formula (5), it also holds for all $M\in \operatorname{ind}\mathcal{R}$ such that $d(M)$ have the different value.
   Thus $G'$ contains all the generators of $G$, and $G=G'$.
\end{proof}

\begin{proposition}\label{xGx-1}
    $\bigcap\limits_{x\in G} xBx^{-1} = \{1\}$.
\end{proposition}

\begin{proof}
    $\bigcap\limits_{x\in G} xBx^{-1}$ is a normal subgroup of $G$.
    We have shown that there exists $n_0 \in N$ such that $n_0 U n_0^{-1} = V$. 
    Hence $n_0 B n_0^{-1} \cap B = H$ and $\bigcap\limits_{x\in G} xBx^{-1} \subseteq  H$.
    
    For all $h\in \bigcap\limits_{x\in G} xBx^{-1} \subseteq  H$ and $u\in U$, we have $huh^{-1}u^{-1}\in U\cap H = \{1\}$. 
    Thus $huh^{-1} = u,  u\in U$.
    Similarly, for any $ v\in V,hvh^{-1} = v$.
    Thus $h$ commutes with all the generators of $G$, i.e. $h$ is contained in the center $Z(G)$ of $G$.
    Hence $\bigcap\limits_{x\in G} xBx^{-1} \subseteq  Z(G)$.
    
    Since $h\in H$, we can write $h=h_{[M_1]}(t_1)\cdots h_{[M_l]}(t_l)$ for some $t_i\in \mathbb{K}^{\times}, M_i \in \operatorname{ind}\mathcal{R}, i=1,\cdots , l$.
    For any $ E_{[M]}(t)$, we have $hE_{[M]}(t)h^{-1} = E_{[M]}(t)$.
    i.e.
    \begin{align*}
        E_{[M]}(t) = E_{[M]}(t_1^{A_{M_1M}}\cdots t_l^{A_{M_lM}}t), \qquad  t\in \mathbb{K}, M\in \operatorname{ind}\mathcal{R}.
    \end{align*}

    If $\mathbb{K}=\mathbb{F}_2$, then $H=1$ and thus $\bigcap\limits_{x\in G} xBx^{-1}=\{1\}$.

    If $\mathbb{K}\neq \mathbb{F}_2$, $t_1^{A_{M_1M}}\cdots t_l^{A_{M_lM}}=1$ holds for all $M \in \operatorname{ind}\mathcal{R}$.
    Notice that by definition $h$ acts on $\mathbb{S}$ by
    \begin{align*}
        &hu_{[M]}=t_1^{A_{M_1M}}\cdots t_l^{A_{M_lM}}u_{[M]} = u_{[M]}. \\
        &hH_S'=H_S'.
    \end{align*}
    Hence $h$ acts trivially on the Lie algebra $\mathfrak{g}$.
    Thus $h=1$ and $\bigcap\limits_{x\in G} xBx^{-1}=\{1\}$.
\end{proof}

\begin{theorem}
    The group $G$ is simple, except for cases $A_1(2),A_1(3),B_2(2),G_2(2)$.
\end{theorem}

\begin{proof}
    Let $G_1$ be a normal subgroup of $G$.
    Then $G_1B$ is a subgroup of $G$ containing $B$.
    Hence $G_1B = P_J$ for some $J\subseteq  \Pi$.
    Let $I=\Pi\setminus J$. 

    Suppose $I,J$ are both nonempty.
    Since the quiver corresponding to the complete section is connected, there exist two vertices $i,j$ of the quiver such that $\alpha_i \in I$, $\alpha_j \in J$, and there's one edge between $i$ and $j$(probably with valuation).
    Let $n_i=n_{[S_i]},n_j=n_{[S_j]}\in N$.
    Since $Bn_jB\subseteq  BN_JB=P_J=G_1B$ and $Bn_jB\cap B = \emptyset$, we have $Bn_jB\cap G_1 \neq \emptyset$.
    Hence $n_iBn_jBn_i\cap G_1 \neq \emptyset$.
    Since 
    \begin{align*}
        &Bn_jBBn_iB \subseteq  Bn_jn_iB \cup Bn_iB \\
        \text{and }\quad &Bn_jBBn_iB \subseteq  Bn_jn_iB \cup Bn_jB ,
    \end{align*}
    and $Bn_iB\cap Bn_jB = \emptyset$, we must have $Bn_jBBn_iB \subseteq  Bn_jn_iB $.
    Hence
    \begin{align*}
        n_iBn_jBn_i \subseteq  n_iBn_jn_iB \subseteq  Bn_in_jn_iB \cup Bn_jn_iB.
    \end{align*}
    Thus 
    \begin{align*}
        Bn_in_jn_iB \cap G_1 \neq \emptyset \quad \text{or} \quad Bn_jn_iB \cap G_1 \neq \emptyset,
    \end{align*}
    which means $n_in_jn_i\in N_J$ or $n_jn_i\in N_J$.

    The second case implies that $n_i\in N_J$ and thus $\omega_{S_i}\in W_J$, which is impossible.

    The first case implies that $n_{[\omega_{S_i}(S_j)]}(\eta_{S_i,S_j})\in N_J$, and thus $\omega_{[\omega_{S_i}(S_j)]} \in W_J$.
    By Lemma \ref{lemmaPJBNB}, $\omega_{\alpha_i}(\alpha_j) = \underline{dim}\omega_{S_i}(S_j)$ is supported on $J$.
    However, we have $\omega_{\alpha_i}(\alpha_j) = \alpha_j - A_{S_i,S_j}\alpha_i$ with $A_{S_i,S_j}\neq 0$, since vertices $i,j$ are connected.
    This means $\alpha_i \in J$, contradicts to $\alpha_i \in I$. 
    Hence $J = \emptyset$ or $\Pi$.

    If $J = \emptyset$, then $G_1B=B$, $G_1\subseteq  B$.
    For any $ x\in G$, we have $G_1=xG_1x^{-1}\subseteq  xBx^{-1}$.
    Hence $G_1\subseteq  \bigcap\limits_{x\in G} xBx^{-1} = \{1\}$ and thus $G_1=\{1\}$.

    If $J = \Pi$, then $G_1B=G$.
    We claim that $B$ is solvable.
    Since $(B,B)\subseteq  U$, we only need to show $U$ is solvable, which follows from the commutator formula.
    Specifically, suppose the largest length of $ M\in \operatorname{ind}\mathcal{R}$ is $a\in \mathbb{Z}_{>0}$, then commute $U$ with itself more than $a$ times we will get 1, which shows it's solvable.
    Hence 
    \begin{align*}
        G/G_1 = G_1B/G_1 \cong B/(G_1\cap B)
    \end{align*}
    is solvable.
    On the other hand, excluding cases $A_1(2),A_1(3),B_2(2),G_2(2)$, we have 
    \begin{align*}
        (G/G_1,G/G_1) = (G,G)/G_1 = G'/G_1 = G/G_1.
    \end{align*}
    Thus we have $G/G_1=\{1\}$ and $G_1=G$.

    Since the normal subgroups of $G$ can only be $\{1\}$ or $G$, $G$ is simple.
\end{proof}

\begin{proposition}
    Every $G$ (even non-simple ones) has trivial centre.
\end{proposition}

\begin{proof}
    Let $Z(G)$ be the center of $G$, which is a normal subgroup of $G$. 
    By the same argument as above, we have $Z(G) = \{1\}$ or $Z(G)B=G$.

    If $Z(G)B=G$, then $B$ is a normal subgroup of $G$. 
    We have proved that $N_G(P_J) = P_J$. 
    As a special case, take $J=\emptyset$ and we have $N_G(B) = B$. 
    Hence $B=G$ and $xBx^{-1} = B$, for any $ x\in G=B$.
    Thus $ \{1\} = \bigcap\limits_{x\in G} xBx^{-1} =B =G$.
    Of course in this case we also have $Z(G)=\{1\}$.
\end{proof}

\subsection{The Steinberg Theorem and Linear Algebraic Groups}
In this subsection, we introduce the Steinberg theorem, and as a corollary, when $\mathbb{K}$ is an algebraically closed field, we show $G$ is a semisimple linear algebraic group, and $\operatorname{Lie}(G)$ is isomorphic to the initial Lie algebra $\mathfrak{g}_{\mathbb{K}}=\mathfrak{g}_{\mathbb{Z}}\otimes \mathbb{K}$.

We write down the Steinberg theorem in our settings as follows.
This theorem shows we can define $G$ as an abstract group with generators and relations.

\begin{theorem}\label{steinberg}
    (Steinberg) With notations the same as the previous sections, 
    for $ M\in \operatorname{ind}\mathcal{R}$ and $ t\in \mathbb{K}$, we introduce a symbol $x_{[M]}(t)$.
    Let $\overline{G}$ be the abstract group generated by $x_{[M]}(t)$ and subject to the following relations:

    (1) $x_{[M]}(t_1) x_{[M]}(t_2) = x_{[M]}(t_1+t_2)$

    (2) For $M\ncong N,M\ncong TN$, 
    \begin{align*}
        (x_{[M]}(t),x_{[N]}(s)) = \prod_{\substack{L_{M,N,i,j}\in \operatorname{ind}\mathcal{R} \\ i,j>0}} x_{[L_{M,N,i,j}]}(C_{M,N,i,j}t^is^j).
    \end{align*}

    (3) $\overline{h_{[M]}}(t_1)\overline{h_{[M]}}(t_2) = \overline{h_{[M]}}(t_1t_2)$, $t_1t_2\neq 0$.

    where $\overline{h_{[M]}}(t) = \overline{n_{[M]}}(t) \overline{n_{[M]}}(-1)$ and $\overline{n_{[M]}}(t) = x_{[M]}(t)x_{[TM]}(t^{-1}) x_{[M]}(t)$.

    (4) $\overline{n_{[M]}}(t) x_{[M]}(s) \overline{n_{[M]}}(t)^{-1} = x_{[TM]}(t^{-2}s)$, $t\neq 0$.

    Let $\overline{Z} $ be the centre of $\overline{G}$. Then $\overline{G}/\overline{Z}\cong G$.
\end{theorem}

\begin{proof}
    We sketch the proof of this theorem, since it's similar to the previous process.
    
    We have already showed that $G$ satisfies the above relations, thus there is a surjective group homomorphism $\varphi:\overline{G}\rightarrow G$, mapping $x_{[M]}(t)$ to $E_{[M]}(t)$.
    
    Fix a complete section (and thus a hereditary subcategory $\mathcal{B}$) of $\mathcal{R}$, and let $\overline{U}$ be the subgroup of $\overline{G}$ generated by $x_{[M]}(t)$ with $M\in \operatorname{ind}\mathcal{B}, t\in \mathbb{K}$.
    By the commutator relations (2) and Prop \ref{commutator}, we have $\varphi|_{\overline{U}}:\overline{U}\rightarrow U$ is an isomorphism. 

    Let $\overline{H}$ be the subgroup generated by $\overline{h_{[M]}}(t)$ for all $M\in \operatorname{ind}\mathcal{R}$ and $t\in \mathbb{K}^{\times}$, and $\overline{N}$ be the subgroup generated by $\overline{H}$ and $\overline{n_{[M]}}(1)$ for all $M \in \operatorname{ind}\mathcal{R}$.
    We can calculate the conjugate relations similar to Lemma \ref{nEn} -- \ref{hnh}, and obtain that $\overline{N}/\overline{H}\cong W$.

    Similar to the proof of the Bruhat's decomposition of $G$, we can obtain that each element of $\overline{G}$ has a unique expression of the form $\overline{u}'\overline{h}\overline{n}\overline{u}$ for some $\overline{u}' \in \overline{U}, \overline{h} \in \overline{H}, \overline{n}\in \overline{N}$ and $\overline{u}\in \overline{U}_{w}^{-}$, where $w\in W$ is the image of $\overline{n}$ under the isomorphism $\overline{N}/\overline{H}\cong W$, and $\overline{U}_{w}^{-}$ is defined analogous to $U_{\overline{\varphi(\overline{n})}}^{-}$ (or equivalently, the inverse image of $U_{\overline{\varphi(\overline{n})}}^{-}$).

    Finally, we can show that the kernel of $\varphi$ is contained in $\overline{H}$ and then is exactly the centre $\overline{Z}$ of $\overline{G}$. 
    The theorem is proved.
\end{proof}

\begin{remark}\label{rmkH}
    (1) We can write down $\overline{Z} $ explicitly.
    Chosen a hereditary subcategory $\mathcal{B}$, let $\{S_1,\cdots,S_m\}$ be a set of representatives of isomorphism classes of simple objects in $\mathcal{B}$.
    Then $\overline{h} = \prod\limits_{i=1}^m \overline{h_{[S_i]}}(t_i) \in \overline{Z}$ if and only if 
    \begin{align}
        (t_1,\cdots,t_m) \in (\mathbb{K}^{\times})^m \qquad \text{satisfies } \quad \prod\limits_{i=1}^m t_i^{a_{ij}} = 1,\quad  j=1,\cdots,m,
    \end{align}
    where $(a_{ij})$ is the Cartan matrix.

    (2) the relation (4) in the theorem can be deduced from relations (1,2,3) when the root system is not of type $A_1$.
\end{remark}

\begin{corollary}
    When $\mathbb{K}$ is an algebraically closed field, $G$ is a semisimple linear algebraic group.
\end{corollary}

\begin{proof}
    Thm \ref{steinberg} shows that $G$ can be characterized by finitely many algebraic equations, which means $G$ is a closed subgroup of $\operatorname{GL}(\mathfrak{g}_{\mathbb{K}})$, and thus a linear algebraic group over $\mathbb{K}$.

    Since $\{E_{[X]}\}_{X\in \operatorname{ind}\mathcal{R}}$ is a family of closed, connected subgroups of $G$, we have $U,V$ are both closed and connected subgroups of $G$.
    The corresponding matrix of $E_{[X]}(t)$ with respect to the Chevalley basis is upper triangular (or lower triangular) with diagonal elements being 1, hence $E_{[X]}(t)$ are unipotent elements, and $U$ (resp. $V$) is a unipotent subgroup.
    Obviously $H$ is a subtorus of $G$.

    Since $U,H$ are closed and connected subgroups of $G$ and they generate $B$, we have $B$ is closed and connected.
    Moreover, $B$ is solvable since $U$ is solvable. 
    Thus $B$ is contained in some Borel subgroup $B'$ of $G$.
    If $B\neq B'$, by Prop \ref{PJonly} and Prop \ref{PJproperties}(4), we have $B'=P_J$ for some $J\neq \emptyset$.
    Then by definition $B'$ contains $E_{[TS_i]}$ for some $i$.
    Thus $E_{[TS_i]}\subseteq B'\cap V = V \cap B_u'$, where $B_u'$ is the unipotent radical of $B'$.
    On the other hand, we have $E_{[S_i]} \subseteq U \subseteq B_u'$.
    Thus $h_{[S_i]}(t)$ is contained in $B_u'$ for any $t\in \mathbb{K}^{\times}$, which means it's unipotent.
    Since $h_{[S_i]}(t)$ is semisimple, it follows that $h_{[S_i]}(t)=1$ for any $t\in \mathbb{K}^{\times}$, which is impossible.
    Hence $B=B'$ is a Borel subgroup of $G$.
    Similarly, $VH$ is also a Borel subgroup of $G$.
    Then $H=VH\cap UH$ is the intersection of two Borel subgroups, and thus is a maximal torus.

    The unipotent radical $R_u(G)$ of $G$ is contained in the intersection of all the Borel subgroups, thus we have $R_u(G) \subseteq H$. 
    Since elements in $R_u(G)$ are unipotent and elements in $H$ are semisimple, we have $R_u(G)=\{1\}$. 
    Hence $G$ is reductive.

    Since all the Borel subgroups of $G$ are conjugate to each other, Prop \ref{xGx-1} tells that the intersection of all the Borel subgroups is $\{1\}$.
    Then the radical $R(G)$ of $G$ is $\{1\}$, because it's the identity component of the intersection of all Borel subgroups.
    Hence $G$ is a semisimple linear algebraic group.
\end{proof}

\begin{remark}
    (1) As a corollary, we can easily show that the center $Z(G)$ of $G$ is trivial when $\mathbb{K}$ is an algebraically closed field.
    Since $G$ is reductive, $Z(G)$ is contained in the maximal torus $H$. 
    Using the same argument in the proof of Prop \ref{xGx-1}, we have $Z(G)=\{1\}$.

    (2) Since $G$ is semisimple, we have $G=(G,G)$.

    (3) A (linear algebraic) group is semisimple means the group only has trivial closed, connected, normal, solvable subgroup, which is not enough to show it's simple as an abstract group.
    However, in our case, $G$ is quasi-simple thanks to the irreducibility of the root system, and then is simple as an abstract group since $Z(G)=\{1\}$.
\end{remark}

\begin{corollary}
    When $\mathbb{K}$ is an algebraically closed field, the Lie algebra $\operatorname{Lie}(G)$ of $G$ is isomorphic to $\mathfrak{g}_{\mathbb{K}}$.
\end{corollary}

\begin{proof}
    We first consider the case $\mathbb{K}=\mathbb{C}$.    
    For $X\in \operatorname{ind}\mathcal{R}$, we can regard $E_{[X]}(t)$ as a matrix of $\operatorname{GL}(g_{\mathbb{C}})$ via its action on $\mathbb{S}$, and let 
    \begin{align*}
        A_{[X]}=\frac{\operatorname{d}E_{[X]}(t)}{\operatorname{d}t}\Big|_{t=0} = \operatorname{lim}\limits_{t\rightarrow 0} \frac{E_{[X]}(t)-1}{t}.
    \end{align*}
    By calculating its action on $\mathbb{S}$, it can be easily shown that for $c$ constant, 
    \begin{equation*}
        \frac{\operatorname{d}E_{[X]}(ct^i)}{\operatorname{d}t}\Big|_{t=0} = 
        \begin{cases}
            0,\qquad \qquad \text{if }i>1,\\
            cA_{[X]},\qquad \text{   if }i=1.
        \end{cases} 
    \end{equation*}

    Recall Prop \ref{commutator} showed that for $X,Y\in \operatorname{ind}\mathcal{R}, X\ncong Y, X\ncong TY$, 
     \begin{equation*}
        E_{[X]}(t) E_{[Y]}(s) E_{[X]}(t)^{-1}= \prod\limits_{i,j>0} E_{[L_{X,Y,i,j}]}(C_{X,Y,i,j}t^i s^j) E_{[Y]}(s).
    \end{equation*}
    We first regard $s$ as a constant and apply $\frac{\operatorname{d}}{\operatorname{d}t}|_{t=0}$ to both sides.
    Then we apply $\frac{\operatorname{d}}{\operatorname{d}s}|_{s=0}$ and obtain that 
    \begin{align*}
        [A_{[X]},A_{[Y]}] = \gamma_{XY}^L A_{[L]},
    \end{align*} 
    where the bracket is given by $[A,B]=AB-BA$.

    For each $X\in \operatorname{ind}\mathcal{R}$, define $C_{[X]}=[A_{[X]},A_{[TX]}]$.
    Then $C_{[X]}$ sends $u_{Y}$ to $-A_{XY}u_Y$, for each $Y\in \operatorname{ind}\mathcal{R}$, and $H_{Z}'$ to $0$, for each $Z\in \mathcal{R}$.
    Since $C_{[X]}$ are diagonal matrices, it's obvious that $[C_{[X]},C_{[Y]}]=0$ for any $X,Y\in \operatorname{ind}\mathcal{R}$.
    Moreover, using the Jacobi identity, we have $[C_{[X]},A_{[Y]}]=-A_{XY}A_{[Y]}$ for any $X,Y\in \operatorname{ind}\mathcal{R}$.

    The matrices $\{A_{[X]}\}_{X\in \operatorname{ind}\mathcal{R}}$ generate a subalgebra $\mathfrak{h}$ of $\operatorname{gl}(\mathfrak{g}_{\mathbb{C}})$. 
    Comparing to the modified version of Def \ref{PXg}, there exists a surjective algebra homomorphism 
    \begin{align*}
        \mathfrak{g}_{\mathbb{C}} \rightarrow \mathfrak{h}, \qquad u_X \mapsto A_{[X]}.
    \end{align*}
    
    We arbitrarily fix a complete section of $\mathcal{R}$ and obtain a hereditary subcategory $\mathcal{B}$. We fix a set $\Delta=\{S_1,\cdots,S_m\}$ of representatives of isomorphism classes of simple objects in $\mathcal{B}$, then $\{H_{S_i}'\}_{i=1,\cdots m}$ form a basis of $K'$.
    Assume 
    \begin{align*}
        \sum\limits_{X\in \operatorname{ind}\mathcal{R}} a_X A_{[X]}+\sum\limits_{Y\in \Delta}b_Y C_{[Y]} = 0
    \end{align*}
    for some $a_X,b_Y\in \mathbb{C}$.
    For any $Z\in \mathcal{R}$, apply the left hand side of the equation to $H_Z'$, we have 
    \begin{align*}
        0=( \sum\limits_{X\in \operatorname{ind}\mathcal{R}} a_X A_{[X]}+\sum\limits_{Y\in \Delta}b_Y C_{[Y]})(H_Z') = \sum\limits_{X\in \operatorname{ind}\mathcal{R}} a_X A_{[X]}H_Z' = \sum\limits_{X\in \operatorname{ind}\mathcal{R}} a_X A_{ZX} u_X.
    \end{align*}
    Recall that $\{u_X\}_{X\in \operatorname{ind}\mathcal{R}}$ are linearly independent, thus we have $a_XA_{ZX}=0$ for each $X\in \operatorname{ind}\mathcal{R}$.
    Since this holds for any $ Z\in \mathcal{R}$, we have $a_X=0$ for each $X\in \operatorname{ind}\mathcal{R}$, and $\sum\limits_{Y\in \Delta}b_Y C_{[Y]}=0$.
    Hence for any $Z\in \operatorname{ind}\mathcal{R}$, 
    \begin{align*}
        0=\sum\limits_{Y\in \Delta}b_Y C_{[Y]}u_Z = -(\sum\limits_{Y\in \Delta}b_Y A_{YZ})u_Z,
    \end{align*}
    which means $\sum\limits_{Y\in \Delta}b_Y A_{YZ}=0$ for any $Z\in \operatorname{ind}\mathcal{R}$.
    Since the Cartan matrix $(A_{YZ})_{Y,Z\in \Delta}$ is invertible for finite-type, the equations $\sum\limits_{Y\in \Delta}b_Y A_{YZ}=0,Z\in \Delta$ can only have zero solutions, i.e. $b_Y=0$ for each $Y\in \Delta$.
    Thus $\{A_{[X]}\}_{X\in \operatorname{ind}\mathcal{R}} \sqcup \{ C_{[Y]} \}_{Y\in \Delta}$ are linearly independent, and we have
 $       \operatorname{dim}\mathfrak{h} \geq \operatorname{dim}\mathfrak{g}_{\mathbb{C}} $.
 Hence $\mathfrak{g}_{\mathbb{C}}\cong \mathfrak{h}$.
 
 On the other hand, $\mathfrak{h}$ is contained in $\operatorname{Lie}(G)$, and $ \operatorname{dim} \operatorname{Lie}(G) =\operatorname{dim} G = \operatorname{dim} \mathfrak{g}_{\mathbb{C}}$.
 Thus we have $\operatorname{Lie}(G)=\mathfrak{h}\cong \mathfrak{g}_{\mathbb{C}}$.

 Then we consider any algebraically closed field $\mathbb{K}$.

Denote $\operatorname{dim}\mathfrak{g}_{\mathbb{K}} = n$. We recall the facts for $\operatorname{GL}(\mathfrak{g}_{\mathbb{K}})$ and refer Chapter 4 of \cite{Springer} for details.
 Note that $\mathbb{K}[\operatorname{GL}(\mathfrak{g}_{\mathbb{K}})]=\mathbb{K}[T_{ij},D^{-1}]_{1\leq i,j \leq n}$, where $D=\operatorname{det}(T_{ij})$ and $T_{ij}$ maps each matrix to its $(i,j)$-entry.
Denote by $\operatorname{gl}(\mathfrak{g}_{\mathbb{K}})$ the Lie algebra of all $n\times n$-matrices over $\mathbb{K}$, with Lie bracket given by $[A,B]=AB-BA$.
If char$\mathbb{K}=p>0$, we let the $p$-operation be taking the $p^{th}$ power.
For $A=(a_{ij})\in \operatorname{gl}(\mathfrak{g}_{\mathbb{K}})$, define a derivation $D_A\in \operatorname{Der}_{\mathbb{K}}(\mathbb{K}[\operatorname{GL}(\mathfrak{g}_{\mathbb{K}})],\mathbb{K}[\operatorname{GL}(\mathfrak{g}_{\mathbb{K}})])$ by 
\begin{align*}
    D_A T_{ij} = -\sum\limits_{k=1}^n a_{kj}T_{ik}.
\end{align*}
The map $A \mapsto D_A$ is injective, and $\operatorname{Lie}(\operatorname{GL}(\mathfrak{g}_{\mathbb{K}}))$ consists of these $D_A$. 
We can identity $\operatorname{Lie}(\operatorname{GL}(\mathfrak{g}_{\mathbb{K}}))$ and $\operatorname{gl}(\mathfrak{g}_{\mathbb{K}})$.
Since $G$ is a closed subgroup of $\operatorname{GL}(\mathfrak{g}_{\mathbb{K}})$, we can view $\operatorname{Lie}(G)$ as a subalgebra of $\operatorname{gl}(\mathfrak{g}_{\mathbb{K}})$.

For $X\in \operatorname{ind}\mathcal{R}$, since the entries of the matrix $A_{[X]}$ are all integers, we can define matrix $A_{[X]}\in \operatorname{gl}(\mathfrak{g}_{\mathbb{K}})$.    
Then $D_{A_{[X]}}, D_{C_{[X]}}$ are in $\operatorname{Lie}(G)$ and satisfy the similar relations. 
For example, for $X,Y\in \operatorname{ind}\mathcal{R}, X\ncong Y,X\ncong TY$, 
\begin{align*}
    [D_{A_{[X]}},D_{A_{[Y]}}]=D_{[A_{[X]},A_{[Y]}]}=D_{\gamma_{XY}^L A_{[L]}}= \gamma_{XY}^L D_{A_{[L]}}.
\end{align*} 

Since the map $A \mapsto D_A$ is injective, the linearly independence of $\{D_{A_{[X]}}\}_{X\in \operatorname{ind}\mathcal{R}} \sqcup \{ D_{C_{[Y]}} \}_{Y\in \Delta}$ follows from the linearly independence of $\{A_{[X]}\}_{X\in \operatorname{ind}\mathcal{R}} \sqcup \{ C_{[Y]} \}_{Y\in \Delta}$.

Similar to the previous case, we have $\operatorname{Lie}(G) \cong \mathfrak{g}_{\mathbb{K}}$. 
\end{proof}

\subsection{The Order of the Finite Chevalley Group}
In this subsection, we calculate the order of $G$ when $\mathbb{K}=\mathbb{F}_q$.
\begin{align*}
    |G| = \sum\limits_{\overline{n}\in N/H} |UHnU_{\overline{n}}^-| = |U|\cdot |H| \cdot \sum\limits_{\overline{n}\in N/H}|U_{\overline{n}}^-|.
\end{align*}

Denote $r=|\operatorname{ind}\mathcal{B}|$, then by Prop \ref{Uunique}, $|U| = q^r, |U_{\overline{n}}^-| = q^{l(\overline{n})}$.

As shown in Remark\ref{rmkH}(1), $|H| = \frac{1}{d}(q-1)^m$, where $m$ is the number of simple objects in $\operatorname{ind}\mathcal{B}$ and $d$ is given by formula(7) in Remark\ref{rmkH}, which can be easily calculated.

Hence
\begin{align*}
    |G| = \frac{1}{d} q^r (q-1)^m \sum\limits_{\overline{n}\in N/H} q^{l(\overline{n})} = \frac{1}{d} q^r (q-1)^m \sum\limits_{w\in W} q^{l(w)} .
\end{align*}

In practice, $\sum\limits_{w\in W} q^{l(w)} $ is hard to calculate.
Here we introduce a easier way to interpret this sum, following the calculations in chapter 10 of \cite{carter}.

We arbitrarily fix a complete section of $\mathcal{R}$, and obtain a hereditary subcategory $\mathcal{B}$.
Let $\{ S_1, \cdots ,S_m \}$ be a set of representatives of isomorphism classes of simple objects in $\mathcal{B}$.

For each $i = 1,\cdots,m$, define a linear map $\varpi_i : K' \rightarrow \mathbb{Z} $ such that $\varpi_i(H_{S_j}') = \delta_{ij}$, for all $j=1,\cdots,m$.
Define a multiplicative group $Q$: 
the elements in $Q$ are linear maps from $K'$ to $\mathbb{Q}$, of the form $\prod_{i=1}^m \varpi_i^{a_i}$ with $a_i\in \mathbb{Z}$ for each $i$.
The map $\varpi_i^{-1}$ is defined by $\varpi_i^{-1}(x) = -\varpi_i(x)$, for all $x\in K'$. 
And the multiplication is given by $(\varpi_i \cdot \varpi_j)(x) = \varpi_i(x)+\varpi_j(x)$, for all $x\in K'$.

For each $X\in \mathcal{R}$, Let $\beta(H_X)$ be the linear map from $K'$ to $\mathbb{Z}$ such that $(\beta(H_X))(H_Y') = (H_X|H_Y')$, for all $Y\in \mathcal{R}$.
Define a multiplicative group $P$ isomorphic to the Grothendieck group $K$, 
where the isomorphism is given by $\beta: \sum_{i=1}^m b_iH_{S_i} \mapsto \prod_{i=1}^m \beta(H_{S_i})^{b_i}$.
Notice that $\beta(H_{S_i})^{-1} = \beta(H_{TS_i})$, and the multiplication rule in $P$ is the same as that of $Q$. 

For any $i,j=1,\cdots,m$, denote $A_{S_i,S_j}$ by $A_{ij} $.
Since 
\begin{align*}
    \beta(H_{S_i})(H_{S_j}') = (H_{S_i}|H_{S_j}') = A_{ji} = \prod_{k=1}^m \varpi_k^{A_{ki}}(H_{S_j}'),
\end{align*}
we have $\beta(H_{S_i}) = \prod_{j=1}^m \varpi_j^{A_{ji}}$.
Hence $P$ is a subgroup of $Q$.

Then we define an action of $N$ on $Q$.
For any $n\in N$ and $i=1,\cdots ,m$, define $n\varpi_i$ be the linear map from $K'$ to $\mathbb{Q}$ such that 
\begin{align*}
    (n\varpi_i)(H_X') = \varpi_i (n^{-1}H_X'), 
\end{align*}
for any $X\in \mathcal{R}$.
For $\prod_{i=1}^m \varpi_i^{a_i}\in Q$, let $n(\prod_{i=1}^m \varpi_i^{a_i}) = \prod_{i=1}^m (n\varpi_i)^{a_i} $ .

To show this defines an action of $N$ on $Q$, we need to show that $n\varpi_i \in Q$.

Firstly, for any $X \in \mathcal{R}, M\in \operatorname{ind}\mathcal{R}$ and $i=1,\cdots,m$, we have 
\begin{align*}
    (n_{[M]}\varpi_i )(H_X') = \varpi_i(n_{[M]}^{-1} H_X').
\end{align*} 
Since 
\begin{align*}
    n_{[M]}^{-1}H_X = h_{[M]}(-1)n_{[M]}H_X =  H_{\omega_M(X)} = H_X - A_{MX}H_M,
\end{align*}
we have 
\begin{align*}
    d(X)(n_{[M]}\varpi_i )(H_X') &= \varpi_i(H_X) - A_{MX}\varpi_i(H_M) \\
    & = \varpi_i(H_X) - \frac{2(H_M|H_X)}{(H_M|H_M)}\varpi_i(H_M) \\
    &= \varpi_i(H_X) - \varpi_i(H_M')\beta(H_M)(H_X). 
\end{align*}
Since $\varpi_i(H_M')\in \mathbb{Z}$ and $\beta(H_M)\in Q$, we have 
\begin{align*}
    (n_{[M]}\varpi_i )(H_X') = (\varpi_i \cdot \beta(H_M)^{-\varpi_i(H_M')})(H_X').
\end{align*}
Hence $n_{[M]}\varpi_i = \varpi_i \cdot \beta(H_M)^{-\varpi_i(H_M')} \in Q$.

Secondly, consider the action of $h_{[M]}$, i.e. $(h_{[M]}\varpi_i)(H_X') = \varpi_i(h_{[M]}^{-1} H_X')$.
Notice that by definition we have $h_{[M]}^{-1}H_X' = H_X'$.
Hence $h_{[M]}\varpi_i = \varpi_i \in Q$ and the action of $N$ on $Q$ is well-defined.
Moreover, since $H$ acts trivially on $Q$, this action induces an action of $W\cong N/H$ on $Q$.

Notice that the action of $N$ on $P$ is compatible with the action of $N$ on $K$, via the isomorphism $\beta$.
Precisely, we have the following lemma.

\begin{lemma}
    For any $n\in N$ and $M\in \mathcal{R}$, we have $n\beta(H_M) = \beta(nH_M)$.
\end{lemma}

\begin{proof}
    Firstly, for any $i,j=1,\cdots,m$, we have 
    \begin{align*}
        n_{[S_i]}(\varpi_j) = \varpi_j \cdot \beta(H_{S_i})^{-\varpi_j(H_{S_i}')} = \varpi_j \cdot \beta(H_{S_i})^{-\delta_{ij}}.
    \end{align*}

    Hence  
    \begin{align*}
        n_{[S_i]}(\beta(H_{S_j})) &= n_{[S_i]}(\prod_{k=1}^m \varpi_k^{A_{kj}}) = \prod_{k=1}^m (\varpi_k\cdot \beta(H_{S_i})^{-\delta_{ik}})^{A_{kj}} \\
        &= \beta(H_{S_j}) \beta(H_{S_i})^{-A_{ij}} = \beta(H_{\omega_{S_i}(S_j)}) = \beta(n_{[S_i]}H_{S_j}).
    \end{align*}
    
    For any $M\in \mathcal{R}$, we can write $H_M = \sum_{j=1}^m b_{M,j}H_{S_j}$.
    Then $\beta(H_M) = \prod_{j=1}^m \beta(H_{S_j})^{b_{M,j}}$.
    Thus we have 
    \begin{align*}
        n_{[S_i]}\beta(H_M) = \prod_{j=1}^m \beta(n_{[S_i]}H_{S_j})^{b_{M,j}} = \beta(\sum_{j=1}^m b_{M,j} n_{[S_i]}H_{S_j}) = \beta(n_{[S_i]}H_M).
    \end{align*}

    Finally, since $H$ acts trivially on both $P$ and $K$, we have for any $n\in N$ and any $M\in \mathcal{R}$, $n\beta(H_M) = \beta(nH_M)$.
\end{proof}

Now we introduce an element $\rho$ in $Q$ which is useful in the calculations.
Let $\rho = \prod_{i=1}^m \varpi_i$. 

\begin{lemma}\label{rhoprop}
    We have the following properties.

    (1) For any $i=1,\cdots,m$, $\rho(H_{S_i}') = 1$. 

    (2) For any $i=1,\cdots,m$, $n_{[S_i]}(\rho) = \rho \cdot \beta(H_{S_i})^{-1}$.

    (3) For any $n\in N$, we have 
    \begin{align*}
        n(\rho) = \rho \cdot \prod_{M\in R(\overline{n^{-1}})} \beta(H_M)^{-1}.
    \end{align*}

    (4) We have 
    \begin{align*}
        \rho^2 = \prod_{M\in \operatorname{ind}\mathcal{B}} \beta(H_M).
    \end{align*}
   
\end{lemma}

\begin{proof}
    By definition, we have
    \begin{align*}
         \rho(H_{S_i}') = (\prod_{j=1}^m \varpi_j)(H_{S_i}') = \sum_{j=1}^m \varpi_j (H_{S_i}') = 1
    \end{align*}
   and
    \begin{align*}
        n_{[S_i]}(\rho) = \prod_{j=1}^m n_{[S_i]}(\varpi_j) = \prod_{j\neq i} \varpi_j \cdot \varpi_i \cdot \beta(H_{S_i})^{-1} = \rho \cdot \beta(H_{S_i})^{-1}.
    \end{align*}
    Hence (1),(2) are proved.

    (3) For any $n\in N$, let $\overline{n} = \overline{n_{[S_{i_1}]}} \cdots \overline{n_{[S_{i_l}]}} $ be a reduced expression of $\overline{n}$.
    Then $\overline{n^{-1}} = \overline{n_{[S_{i_l}]}} \cdots \overline{n_{[S_{i_1}]}} $ is a reduced expression of $\overline{n^{-1}}$.
    Recall that 
    \begin{align*}
         R(\overline{n^{-1}}) = \{ [S_{i_1}], [\omega_{S_{i_1}}(S_{i_2})],\cdots,[\omega_{S_{i_1}}\cdots \omega_{S_{i_{l-1}}}(S_{i_l})]   \}.
    \end{align*}
    Hence 
    \begin{align*}
        n(\rho) &= n_{[S_{i_1}]} \cdots n_{[S_{i_l}]}(\rho) \\
            &= n_{[S_{i_1}]} \cdots n_{[S_{i_{l-1}}]}(\rho\cdot \beta(H_{S_{i_l}})^{-1}) \\
            &=  n_{[S_{i_1}]} \cdots n_{[S_{i_{l-2}}]}(\rho\cdot \beta(H_{S_{i_{l-1}}})^{-1} \cdot \beta(H_{\omega_{S_{i_{l-1}}}(S_{i_l})})^{-1}   ) \\
            &= \cdots =\rho \cdot \beta(H_{S_{i_1}})^{-1} \cdots \beta(H_{\omega_{S_{i_1}} \cdots \omega_{S_{i_{l-1}}} (S_{i_l}) })^{-1} \\
            &= \rho \cdot \prod_{M\in R(\overline{n^{-1}})} \beta(H_M)^{-1}.
    \end{align*}

    (4) On the one hand, for any $i=1,\cdots ,m$, we have 
    \begin{align*}
        n_{[S_i]}(\prod_{M\in \operatorname{ind}\mathcal{B}} \beta(H_M)) &= \prod_{M\in \operatorname{ind}\mathcal{B}} \beta(H_{\omega_{S_i}(M)}) \\
        &= \prod_{M\in \operatorname{ind}\mathcal{B}} \beta(H_M) \cdot \beta(H_{S_i})^{-1} \cdot \beta(H_{TS_i}) \\
        &= \prod_{M\in \operatorname{ind}\mathcal{B}} \beta(H_M) \cdot \beta(H_{S_i})^{-2}. 
    \end{align*}

    On the other hand, we also have 
    \begin{align*}
        n_{[S_i]}(\prod_{M\in \operatorname{ind}\mathcal{B}} \beta(H_M)) &= \beta(\sum_{M\in \operatorname{ind}\mathcal{B}} H_{\omega_{S_i}(M)}) \\
        &= \beta(\sum_{M\in \operatorname{ind}\mathcal{B}} H_M -(\sum_{M\in \operatorname{ind}\mathcal{B}} A_{S_i,M})H_{S_i} ) \\
        &= \prod_{M\in \operatorname{ind}\mathcal{B}} \beta(H_M) \cdot \beta(H_{S_i})^{-\sum_{M\in \operatorname{ind}\mathcal{B}} A_{S_i,M}}.
    \end{align*}

    Hence for any $i=1,\cdots,m$, 
    \begin{align*}
        \sum_{M\in \operatorname{ind}\mathcal{B}} A_{S_i,M} = 2.
    \end{align*}

    Since $ \prod_{M\in \operatorname{ind}\mathcal{B}} \beta(H_M) \in Q$, we can write 
    \begin{align*}
         \prod_{M\in \operatorname{ind}\mathcal{B}} \beta(H_M) = \prod_{i=1}^m \varpi_i^{a_i},
    \end{align*}
    where 
    \begin{align*}
        a_j &= \sum_{i=1}^m a_i \delta_{ij} = (\prod_{i=1}^m \varpi_i^{a_i}) (H_{S_j}') = (\prod_{M\in \operatorname{ind}\mathcal{B}} \beta(H_M)) (H_{S_j}')\\
         &=  \sum_{M\in \operatorname{ind}\mathcal{B}} (H_M|H_{S_j}') =  \sum_{M\in \operatorname{ind}\mathcal{B}} A_{S_j,M} = 2.
    \end{align*}

    Thus 
    \begin{align*}
        \prod_{M\in \operatorname{ind}\mathcal{B}} \beta(H_M) = \prod_{i=1}^m \varpi_i^2 = \rho^2.
    \end{align*}
   
\end{proof}

For any subset $\Omega \subseteq \operatorname{ind}\mathcal{B}$, we consider  
\begin{align*}
   S_{\Omega} =  \rho \cdot \prod_{M\in \Omega} \beta(H_M)^{-1}.
\end{align*}

Since $\prod_{M\in \Omega} \beta(H_M)\in Q$, we can write 
\begin{align*}
    \prod_{M\in \Omega} \beta(H_M) = \prod_{i=1}^m \varpi_i^{b_{\Omega,i}},
\end{align*}
where 
\begin{align*}
    b_{\Omega,i} = \prod_{j=1}^m \varpi_j^{b_{\Omega,j}}(H_{S_i}') = \prod_{M\in \Omega}\beta(H_M)(H_{S_i}') = \sum_{M\in\Omega} (H_M|H_{S_i}') = \sum_{M\in \Omega} A_{S_i,M}.
\end{align*}

Hence
\begin{align*}
    S_{\Omega} = \rho \cdot \prod_{M\in \Omega} \beta(H_M)^{-1} = \prod_{i=1}^m \varpi_i^{1-\sum_{M\in \Omega} A_{S_i,M}}.
\end{align*}

Denote 
\begin{align*}
    a_{\Omega,i} = S_{\Omega} (H_{S_i}') = 1-\sum_{M\in\Omega} A_{S_i,M}.
\end{align*}

Then for any $L\in \operatorname{ind}\mathcal{R}$, we have 
\begin{align*}
    n_{[L]}(S_{\Omega}) = n_{[L]}(\prod_{i=1}^m \varpi_i^{a_{\Omega,i}}) = S_{\Omega}\cdot \beta(H_L)^{-S_{\Omega}(H_{L}')}.
\end{align*}
This means $S_{\Omega}$ is fixed by $n_{[L]}$ if and only if $S_{\Omega}(H_L') = 0$.

\begin{lemma}\label{omega0}
    If $a_{\Omega,i} >0$ for all $i=1,\cdots,m$, then $\Omega = \emptyset$.
\end{lemma}

\begin{proof}
    For all $i=1,\cdots,m$, since $a_{\Omega,i}>0$, we have 
    \begin{align*}
        (\sum_{M\in\Omega} H_M|H_{S_i}') = \sum_{M\in\Omega} (H_M|H_{S_i}') = \sum_{M\in\Omega} A_{S_i,M} \leq 0, 
    \end{align*}
    and hence 
    \begin{align*}
        (\sum_{M\in\Omega} H_M|H_{S_i}) \leq 0.
    \end{align*}

    Since $\Omega$ is a subset of $\operatorname{ind}\mathcal{B}$, we can write $\sum_{M\in\Omega} H_M = \sum_{i=1}^m c_i H_{S_i}$ for some $c_i\geq 0$, $i=1,\cdots,m$.
    Then 
    \begin{align*}
        (\sum_{M\in\Omega} H_M|\sum_{M\in\Omega} H_M) = \sum_{i=1}^m c_i(\sum_{M\in\Omega}H_M|H_{S_i}) \leq 0.
    \end{align*} 
    This means $\sum_{M\in\Omega} H_M=0 $ and $\Omega = \emptyset$.
\end{proof}

Let $\mathbb{Q}Q$ be the rational group algebra of $Q$.
The elements in $\mathbb{Q}Q$ are of the form $\sum_{x\in Q} \lambda_x x$, where $\lambda_x\in \mathbb{Q}$.
The action of $N$ on $Q$ can be linearly extended to an action on $\mathbb{Q}Q$.

Define a linear map $\theta:\mathbb{Q}Q \rightarrow \mathbb{Q}Q$ :
\begin{align*}
    \theta = \sum_{n\in N} (-1)^{l(\overline{n})} n.
\end{align*}

\begin{lemma}\label{SOmega}
    (1) If there exists some $L\in \operatorname{ind}\mathcal{R}$ such that $S_{\Omega}(H_L') = 0$, then $\theta(S_{\Omega}) = 0$.

    (2) If for any $L\in \operatorname{ind}\mathcal{R}$, we have $S_{\Omega}(H_{L}')\neq 0$, then $S_{\Omega} = n(\rho)$ for some $n\in N$, and $\Omega = R(\overline{n^{-1}})$.
\end{lemma}

\begin{proof}
    (1) Since $n_{[L]}(S_{\Omega}) = S_{\Omega}$ and $l(\overline{n_{[L]}})$ is odd, for any $n\in N$, we have 
    \begin{align*}
        (-1)^{l(\overline{nn_{[L]}})} nn_{[L]} (S_{\Omega}) + (-1)^{l(\overline{n})} n(S_{\Omega}) = 0.
    \end{align*}
    Hence $\theta(S_{\Omega}) = \sum_{n\in N} (-1)^{l(\overline{n})} n(S_{\Omega}) = 0$.

    (2) We define a total ordering on $P\otimes_{\mathbb{Z}} \mathbb{Q}$: 
    for $v_1,v_2 \in P\otimes_{\mathbb{Z}} \mathbb{Q}$, $v_1>v_2$ if write $v_1v_2^{-1} = \prod_{i=1}^m \beta(H_{S_i})^{\lambda_i}$ ($\lambda_i \in \mathbb{Q}$ for each $i$) and the first non-zero $\lambda_i>0$.
   
    Firstly, we show that there exists an $n\in N$ such that $n(S_{\Omega})(H_{S_i}')>0$ holds for all $i=1,\cdots,m$.
    Consider the set $\{ n(S_{\Omega}) |n\in N  \}$.
    Let $S_{\Omega}' = n'(S_{\Omega})$ for some $n'\in N$ be the largest among the elements in this set with respect to the ordering.
    Then since for each $i=1,\cdots,m$, 
    \begin{align*}
        n_{[S_i]}(S_{\Omega}') = S_{\Omega}' \cdot \beta(H_{S_i})^{-S_{\Omega}'(H_{S_i}')} \leq S_{\Omega}',
    \end{align*}
    we have $S_{\Omega}'(H_{S_i}')\geq 0$.
    If there exists an $i$ such that $S_{\Omega}'(H_{S_i}') = 0$, then $n_{[S_i]}(S_{\Omega}') = S_{\Omega}'$.
    Hence $n_{[S_i]}n'(S_{\Omega}) = n'(S_{\Omega})$, i.e. $(n')^{-1}n_{[S_i]}n'(S_{\Omega})=S_{\Omega}$.
    Suppose $\overline{n'} = \overline{n_{[S_{i_1}]}} \cdots \overline{n_{[S_{i_l}]}}$ is a reduced expression of $\overline{n'}$, 
    then 
    \begin{align*}
        \overline{(n')^{-1}n_{[S_i]}n'} = \overline{n_{[\omega_{S_{i_l}}\cdots \omega_{S_{i_1}} (S_i) ]}}.
    \end{align*}
    Thus $n_{[\omega_{S_{i_l}}\cdots \omega_{S_{i_1}} (S_i)]}(S_{\Omega}) = S_{\Omega}$.
    However, since for any $L\in \operatorname{ind}\mathcal{R}$, $S_{\Omega}(H_{L}')\neq 0$, i.e. $n_{[L]}(S_{\Omega}) \neq S_{\Omega}$, we obtain a contradiction.
    Hence for each $i=1,\cdots,m$, we have $n'(S_{\Omega})(H_{S_i}') =   S_{\Omega}'(H_{S_i}')> 0$.

    Secondly, for any $n\in N$, we show that $n(S_{\Omega}) = S_{\Omega'}$ for some subset $\Omega'$ of $\operatorname{ind}\mathcal{B}$.
    We have 
    \begin{align*}
        n(S_{\Omega}) = n(\rho\cdot \prod_{M\in \Omega} \beta(H_M)^{-1}) = \rho \cdot \prod_{M\in R(\overline{n^{-1}})} \beta(H_M)^{-1}  \cdot \prod_{M\in \Omega} \beta(nH_M)^{-1}.
    \end{align*}
    If $\overline{n} = \overline{n_{[S_{i_1}]}}\cdots\overline{n_{[S_{i_l}]}}$ is a reduced expression of $\overline{n}$, we denote $\omega_{\overline{n}} = \omega_{S_{i_1}}\cdots \omega_{S_{i_l}}$.
    Notice that $\omega_{\overline{n}}$ doesn't depend on which reduced expression is chosen.
    With this notation, $nH_M=H_{\omega_{\overline{n}}(M)}$.
    Then 
    \begin{align*}
        \prod_{M\in \Omega} \beta(nH_M)^{-1} &= \prod_{\omega_{\overline{n}}^{-1}(L)\in \Omega} \beta(H_L)^{-1} \\
         &=  \prod_{\substack{\omega_{\overline{n}}^{-1}(L)\in \Omega \\ L\in \operatorname{ind}\mathcal{B}}} \beta(H_L)^{-1} \cdot \prod_{\substack{\omega_{\overline{n}}^{-1}(L)\in \Omega \\ L\in \operatorname{ind}T\mathcal{B}}} \beta(H_L)^{-1} \\
         &=\prod_{\substack{\omega_{\overline{n}}^{-1}(L)\in \Omega \\ L\in \operatorname{ind}\mathcal{B}}} \beta(H_L)^{-1} \cdot \prod_{\substack{\omega_{\overline{n}}^{-1}(TL)\in \Omega \\ L\in \operatorname{ind}\mathcal{B}}} \beta(H_L).
    \end{align*}
    All $L$ appears in $\prod_{\omega_{\overline{n}}^{-1}(L)\in \Omega ; L\in \operatorname{ind}\mathcal{B}} \beta(H_L)^{-1} $ are in $\operatorname{ind}\mathcal{B}$ but not in $R(\overline{n^{-1}})$, 
    while all $L$ appears in $\prod_{\omega_{\overline{n}}^{-1}(TL)\in \Omega ; L\in \operatorname{ind}\mathcal{B}} \beta(H_L) $ are in $R(\overline{n^{-1}})$.
    Put them together, we have $n(S_{\Omega}) = S_{\Omega'}$ for some $\Omega'\subseteq \operatorname{ind}\mathcal{B}$.

    As a result, we have $S_{\Omega}' = n'(S_{\Omega}) = S_{\Omega''}$ for some $\Omega''\subseteq \operatorname{ind}\mathcal{B}$.
    By Lemma \ref{omega0}, we have $\Omega''=\emptyset$.
    This means $n'(S_{\Omega}) = \rho$ and $S_{\Omega}=(n')^{-1}(\rho)$.
\end{proof}

\begin{theorem}
    We have 
    \begin{align*}
        \theta(\rho) = |H|\cdot \rho^{-1} \cdot \prod_{M\in \operatorname{ind}\mathcal{B}}(\beta(H_M)-1).
    \end{align*}
\end{theorem}

\begin{proof}
    Let $\lambda = \rho^{-1} \cdot \prod_{M\in \operatorname{ind}\mathcal{B}}(\beta(H_M)-1)$.
    For any $i=1,\cdots,m$, 
    \begin{align*}
        n_{[S_i]}(\lambda) &= \rho^{-1} \cdot \beta(H_{S_i}) \cdot \prod_{M\in \operatorname{ind}\mathcal{B}}(\beta(n_{[S_i]}H_M)-1) \\
            &= \rho^{-1} \cdot \beta(H_{S_i}) \cdot \prod_{M\in \operatorname{ind}\mathcal{B}}(\beta(H_{\omega_{S_i}(M)})-1) \\
            &=\rho^{-1} \cdot \beta(H_{S_i}) \cdot \prod_{M\in \operatorname{ind}\mathcal{B}}(\beta(H_M)-1) \cdot \frac{\beta(H_{S_i})^{-1}-1}{\beta(H_{S_i})-1} \\
            &=-\lambda.
    \end{align*}
    Hence for any $n\in N$, we have $n(\lambda) = (-1)^{l(\overline{n})}\lambda$, and $\theta(\lambda) = |N|\lambda$.

    On the other hand, by Lemma\ref{rhoprop}(4), 
    \begin{align*}
        \lambda &= \rho \cdot \prod_{M\in \operatorname{ind}\mathcal{B}} \beta(H_M)^{-1} \cdot \prod_{M\in \operatorname{ind}\mathcal{B}}(\beta(H_M)-1) \\
         &= \rho \cdot \prod_{M\in \operatorname{ind}\mathcal{B}}(1-\beta(H_M)^{-1}) \\
         &= \rho \cdot \sum_{\Omega \subseteq \operatorname{ind}\mathcal{B}} (-1)^{|\Omega|} \prod_{M\in \Omega} \beta(H_M)^{-1} \\
         &= \sum_{\Omega\subseteq \operatorname{ind}\mathcal{B}} (-1)^{|\Omega|} S_{\Omega}.
    \end{align*}

    Thus by Lemma\ref{SOmega}, 
    \begin{align*}
        \lambda &= \frac{1}{|N|} \theta(\lambda) = \frac{1}{|N|} \sum_{\Omega \subseteq \operatorname{ind}\mathcal{B}} (-1)^{|\Omega|} \theta(S_{\Omega}) \\
        &=\frac{1}{|N|} \sum_{\overline{n}\in N/H} (-1)^{l(\overline{n})} \theta(n(\rho)) = \frac{1}{|N|} \sum_{\overline{n}\in N/H} \theta(\rho)  = \frac{1}{|H|} \theta(\rho),
    \end{align*}    
    and the desired result follows.
\end{proof}

Since $\mathbb{Q}Q$ is an integral domain, let $F=Frac(\mathbb{Q}Q)$ be the field of fractions of $\mathbb{Q}Q$, and let $F[t]$ be the polynomial ring in indeterminate $t$.

\begin{theorem}\label{Macdonald}
    We have 
    \begin{align*}
        \sum_{n\in N} ( \prod_{M\in \operatorname{ind}\mathcal{B}} \frac{1-t\cdot n\beta(H_M)^{-1}}{1-n\beta(H_M)^{-1}}    ) = \sum_{n\in N} t^{l(\overline{n})}.
    \end{align*}
\end{theorem}
    
\begin{proof}
    Again, let $\lambda = \rho^{-1} \cdot \prod_{M\in \operatorname{ind}\mathcal{B}}(\beta(H_M)-1)$.
    Then 
    \begin{align*}
        (-1)^{l(\overline{n})} \lambda = n(\lambda) = n(\rho)\cdot \prod_{M\in \operatorname{ind}\mathcal{B}}( 1-n\beta(H_M)^{-1}  ).
    \end{align*}
    Hence
    \begin{align*}
        \sum_{n\in N} ( \prod_{M\in \operatorname{ind}\mathcal{B}} &\frac{1-t\cdot n\beta(H_M)^{-1}}{1-n\beta(H_M)^{-1}}   ) = \frac{1}{\lambda} \sum_{n\in N} (-1)^{l(\overline{n})} n(\rho) \cdot \prod_{M\in \operatorname{ind}\mathcal{B}} (1-t\cdot n\beta(H_M)^{-1}) \\
          &=\frac{1}{\lambda} \sum_{n\in N} (-1)^{l(\overline{n})} n(\rho) \cdot \sum_{\Omega \subseteq \operatorname{ind}\mathcal{B}} (-t)^{|\Omega|} n(\prod_{M\in \Omega}\beta(H_M)^{-1})\\
          &=\frac{1}{\lambda} \sum_{\Omega \subseteq \operatorname{ind}\mathcal{B}} (-t)^{|\Omega|} \sum_{n\in N} (-1)^{l(\overline{n})} n(S_{\Omega}) \\
          &=\frac{1}{\lambda} \sum_{\Omega \subseteq \operatorname{ind}\mathcal{B}} (-t)^{|\Omega|} \theta(S_{\Omega}) =\frac{1}{\lambda} \sum_{\overline{n}\in N/H} (-t)^{l(\overline{n})} \theta(n(\rho)) \\
          &= \frac{1}{\lambda} \sum_{\overline{n}\in N/H} t^{l(\overline{n})} \theta(\rho) =|H|\sum_{\overline{n}\in N/H} t^{l(\overline{n})} =\sum_{n\in N} t^{l(\overline{n})}. 
    \end{align*}
\end{proof}

\begin{theorem}
    Let $l(M)$ be the length of $M \in \mathcal{R}$, then 
    \begin{align*}
        \sum_{\overline{n}\in N/H} t^{l(\overline{n})} = \prod_{M\in \operatorname{ind}\mathcal{B}} \frac{t^{l(M)+1}-1}{t^{l(M)}-1}.
    \end{align*}
\end{theorem}

\begin{proof}
    We have a homomorphism from $P$ to the infinite cyclic group generated by $t$, which maps $\beta(H_M)$ to $t^{-l(M)}$.
    Let $\mathbb{Q}P$ be the rational group algebra of $P$.
    This homomorphism can be extended to an algebra homomorphism $\mathbb{Q}P[t]\rightarrow \mathbb{Q}[t,\frac{1}{t}]$, which maps $t$ to $t$.
    The identity proved in Thm \ref{Macdonald} can be interpreted as an identity in $\mathbb{Q}P[t]$, by multiplying the denominator to both sides.
    Apply the homomorphism to this identity, we obtain
    \begin{align*}
        \sum_{n\in N}(\prod_{M\in \operatorname{ind}\mathcal{B}} \frac{1-t^{1+l(\omega_{\overline{n}}(M))}}{1-t^{l(\omega_{\overline{n}}(M))}}) = \sum_{n\in N} t^{l(\overline{n})}.
    \end{align*}

    Recall that our $n\in N$ is defined by the BGP reflection functors. 
    By Gabriel's theorem, for any $n\in N$ such that $\overline{n}\neq 1$, there exists some $M\in \operatorname{ind}\mathcal{B}$ such that $l(\omega_{\overline{n}}(M)) = -1$.
    Hence
    \begin{align*}
        \sum_{n\in N} t^{l(\overline{n})} =  \sum_{n\in H}(\prod_{M\in \operatorname{ind}\mathcal{B}} \frac{1-t^{1+l(M)}}{1-t^{l(M)}}) = |H| \prod_{M\in \operatorname{ind}\mathcal{B}} \frac{t^{1+l(M)}-1}{t^{l(M)}-1},
    \end{align*}
    and thus 
    \begin{align*}
        \sum_{\overline{n}\in N/H} t^{l(\overline{n})} = \prod_{M\in \operatorname{ind}\mathcal{B}} \frac{t^{1+l(M)}-1}{t^{l(M)}-1}.
    \end{align*}
\end{proof}

\begin{lemma}\label{k1k2}
    Let $k_i$ be the number of objects in $ \operatorname{ind}\mathcal{B}$ with length $i$.
    We have $k_1,k_2,\cdots$ is a descending sequence of positive integers, i.e.
\begin{align*}
   m= k_1 \geq k_2  \geq \cdots .
\end{align*}
\end{lemma}

\begin{proof}
    We prove this Lemma by considering the AR-quiver for each type.
    Since we arbitrarily fix a complete section of $\mathcal{R}$ and obtain a hereditary subcategory $\mathcal{B}$, we have a corresponding quiver with some orientation.
    Note that the result we want doesn't depend on the orientation of the quiver.
    Thus we only need to show for one special orientation and we may assume that $\mathcal{B}$ exactly corresponds to the orientation we choose. 

    Firstly, we consider the case of type A.
    Let the orientation of the quiver be 
    \begin{align*}
        1 \rightarrow 2 \rightarrow \cdots \rightarrow m
    \end{align*}
    with the vertices labelled by numbers $1,2,\cdots,m$.
    Then the AR-quiver for $\mathcal{B}$ is

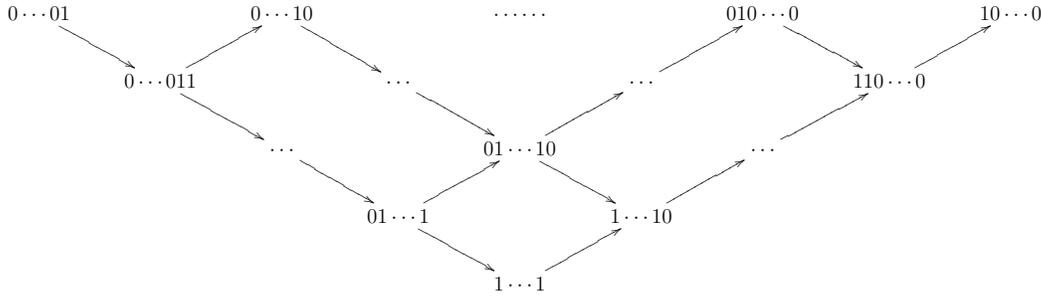
\resizebox{14cm}{!}{
    \xymatrix{
        0\cdots 01 \ar[rd]    &                                   & 0\cdots 10 \ar[rd]    &                                  & \cdots \cdots                           &                       & 010\cdots 0 \ar[rd]    &                        & 10\cdots 0 \\
                              & 0\cdots 011 \ar[rd] \ar[ru]       &                       & \cdots \ar[rd]                   &                                   & \cdots \ar[ru]            &                        & 110\cdots 0 \ar[ru]    &            \\
                              &                                   & \cdots \ar[rd]        &                                  & 01\cdots 10 \ar[rd] \ar[ru]       &                       & \cdots \ar[ru]             &                        &            \\
                              &                                   &                       & 01\cdots 1 \ar[rd] \ar[ru]       &                                   & 1\cdots 10 \ar[ru]    &                        &                        &            \\
                              &                                   &                       &                                  & 1\cdots 1 \ar[ru]                 &                       &                        &                        &           
    }
}

It's easy to see that the indecomposables in the first row have length 1, the indecomposables in the second row have length 2, etc.
Thus in this case $k_i=m+1-i$ for $i=1,\cdots,m$, and the Lemma follows.

Secondly, we consider the case of type DE.
It's well-known that any quiver of this type is of the form below (with probably different orientation).

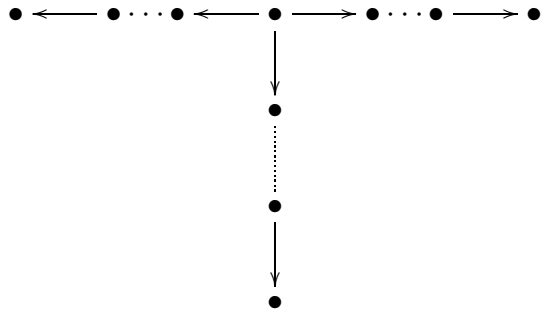
\[
\xymatrix{
    \bullet    &\ar[l]\bullet \cdots  \bullet  & \ar[l]\bullet \ar[r] \ar[d] & \bullet \cdots \bullet \ar[r] & \bullet \\
    & & \bullet \ar@{.}[d] & & \\
    & & \bullet \ar[d] & &\\
    & & \bullet & &
}
\]

We label the unique source of this quiver by 1, and denote the projective module corresponding to 1 by $P_1$.

Then we describe the AR-quiver corresponding to this orientation from left to right.
The AR-quiver starts from three branches of the form the same as that of type A, before they meet at the point $P_1$.
Ringel\cite{1099} called the vertex corresponding to $P_1$ in the AR-quiver a wing vertex, and $P_1$ a wing module.
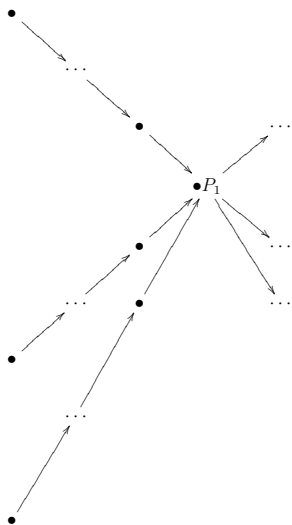
\[
\resizebox{4cm}{!}{
\xymatrix{
    \bullet \ar[rd]  &                    &                     &                                           &         \\
                        & \cdots \ar[rd]  &                     &                                           &         \\
                        &                    & \bullet \ar[rd]  &                                           & \cdots \\
                        &                    &                     & \bullet P_1 \ar[ru] \ar[rd] \ar[rdd] &         \\
                        &                    & \bullet \ar[ru]  &                                           & \cdots \\
                        & \cdots \ar[ru]  & \bullet \ar[ruu] &                                           &  \cdots \\
    \bullet \ar[ru]  &                    &                     &                                           &         \\
                        & \cdots \ar[ruu] &                     &                                           &         \\
                        &                    &                     &                                           &         \\
    \bullet \ar[ruu] &                    &                     &                                           &        
 }
}
\]
The lengths of the corresponding modules in the branches are calculated in the case of type A.
Assume that the three adjacent vertices of $P_1$ have lengths $x,y,z$, respectively.
Then $P_1$ has length $x+y+z+1$, and the three adjacent vertices on the right hand side of $P_1$ all have length $x+y+z$.
Following this process, we can calculate all the lengths of module in $\operatorname{ind}\mathcal{B}$, and get the desired result.

For example, we calculated the case of type $\rm E_6$.
The quiver for type $\rm E_6$ is as follows.
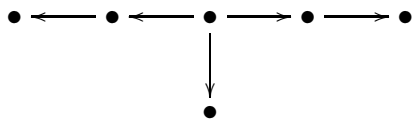
\[
\xymatrix{
    \bullet    &\ar[l] \bullet  & \ar[l]\bullet \ar[r] \ar[d] &  \bullet \ar[r] & \bullet \\
    & & \bullet  & & \\
}
\]
The following diagram is the AR-quiver, whose vertices are labelled by the lengths of the corresponding indecomposable modules.

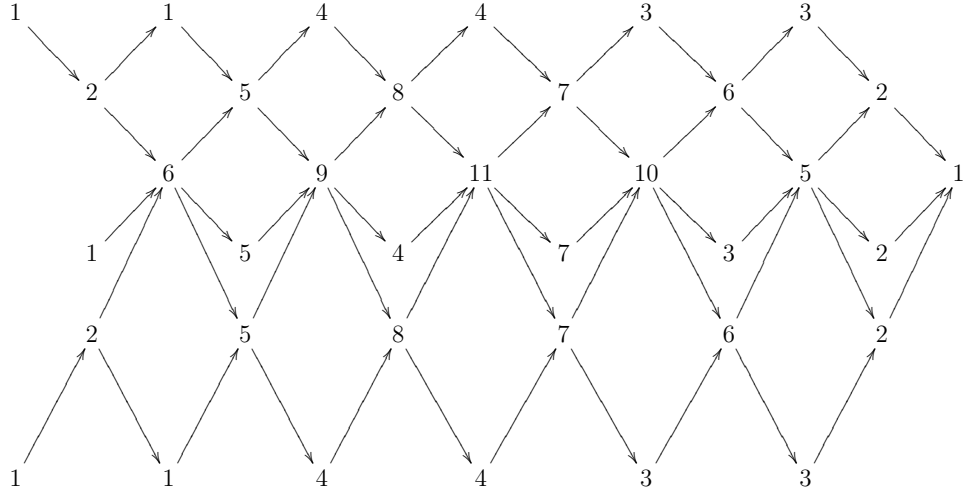
\[
    \resizebox{13cm}{!}{
\xymatrix{
    1 \ar[rd]  &                           & 1 \ar[rd]                        &                           & 4 \ar[rd]                        &                           & 4 \ar[rd]                         &                           & 3 \ar[rd]                         &                           & 3 \ar[rd]                        &               &   \\
                  & 2 \ar[rd] \ar[ru]   &                                     & 5 \ar[ru] \ar[rd]   &                                     & 8 \ar[ru] \ar[rd]   &                                      & 7 \ar[ru] \ar[rd]   &                                      & 6 \ar[ru] \ar[rd]   &                                     & 2 \ar[rd]  &   \\
                  &                           & 6 \ar[ru] \ar[rd] \ar[rdd] &                           & 9 \ar[ru] \ar[rd] \ar[rdd] &                           & 11 \ar[ru] \ar[rd] \ar[rdd] &                           & 10 \ar[ru] \ar[rd] \ar[rdd] &                           & 5 \ar[ru] \ar[rd] \ar[rdd] &               & 1 \\
                  & 1 \ar[ru]              &                                     & 5 \ar[ru]              &                                     & 4 \ar[ru]              &                                      & 7 \ar[ru]              &                                      & 3 \ar[ru]              &                                     & 2 \ar[ru]  &   \\
                  & 2 \ar[ruu] \ar[rdd] &                                     & 5 \ar[ruu] \ar[rdd] &                                     & 8 \ar[ruu] \ar[rdd] &                                      & 7 \ar[ruu] \ar[rdd] &                                      & 6 \ar[ruu] \ar[rdd] &                                     & 2 \ar[ruu] &   \\
                  &                           &                                     &                           &                                     &                           &                                      &                           &                                      &                           &                                     &               &   \\
    1 \ar[ruu] &                           & 1 \ar[ruu]                       &                           & 4 \ar[ruu]                       &                           & 4 \ar[ruu]                        &                           & 3 \ar[ruu]                        &                           & 3 \ar[ruu]                       &               &  
}
    }
\]
From the AR-quiver, we can see that for $\rm E_6$, 
\begin{align*}
    k_1=6, k_2=k_3=k_4=5,k_5=4,k_6=k_7=3,k_8=2,k_9=k_{10}=k_{11}=1.
\end{align*}

Finally, we consider the case of type BCFG.
By Thm \ref{deng}, the AR-quiver in this case can be obtained from some AR-quiver corresponding to type ADE.
Thus the lengths of modules in $\operatorname{ind}\mathcal{B}$ can also be calculated explicitly.

For example, we consider the case of type $\rm F_4$, whose quiver is of the following form.
Note that we use multiple edges (or arrows) to indicate the valuation.
\[
\xymatrix{
   \bullet \ar@{-}[r] & \bullet \ar@2{-}[r] & \bullet \ar@{-}[r] & \bullet
}
\]
This can be obtained from $\rm E_6$ with an automorphism.
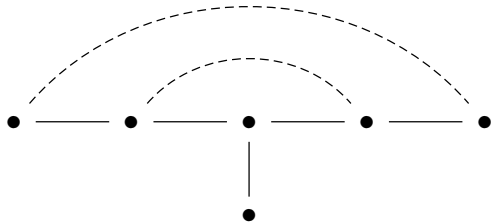
\[
\begin{tikzcd}
    \bullet \arrow[r, no head] \arrow[rrrr, no head, dashed, bend left=49] & \bullet \arrow[rr, no head, dashed, bend left=49] \arrow[r, no head] & \bullet \arrow[r, no head] \arrow[d, no head] & \bullet \arrow[r, no head] & \bullet \\
                                                                           &                                                                      & \bullet                                       &                            &        
    \end{tikzcd}
\]

Thus we can fold the AR-quiver of $\rm E_6$ and obtain that of $\rm F_4$.
Again we use the lengths of the corresponding indecomposable modules to label the vertices of the AR-quiver.

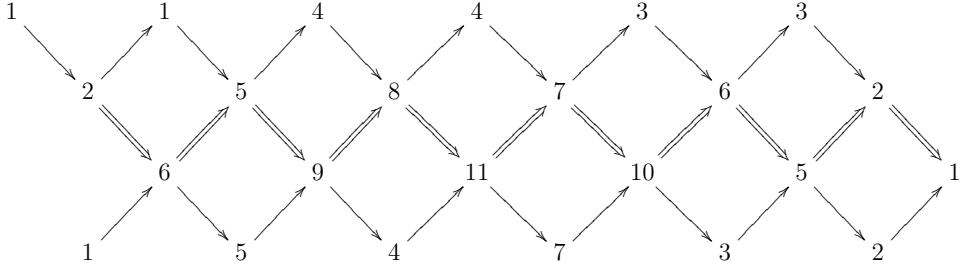
\[
    \resizebox{13cm}{!}{
\xymatrix{
    1 \ar[rd]  &                           & 1 \ar[rd]                        &                           & 4 \ar[rd]                        &                           & 4 \ar[rd]                         &                           & 3 \ar[rd]                         &                           & 3 \ar[rd]                        &               &   \\
                  & 2 \ar@2{->}[rd] \ar[ru]   &                                     & 5 \ar[ru] \ar@2{->}[rd]   &                                     & 8 \ar[ru] \ar@2{->}[rd]   &                                      & 7 \ar[ru] \ar@2{->}[rd]   &                                      & 6 \ar[ru] \ar@2{->}[rd]   &                                     & 2 \ar@2{->}[rd]  &   \\
                  &                           & 6 \ar@2{->}[ru] \ar[rd]  &                           & 9 \ar@2{->}[ru] \ar[rd]  &                           & 11 \ar@2{->}[ru] \ar[rd]  &                           & 10 \ar@2{->}[ru] \ar[rd]  &                           & 5 \ar@2{->}[ru] \ar[rd]  &               & 1 \\
                  & 1 \ar[ru]              &                                     & 5 \ar[ru]              &                                     & 4 \ar[ru]              &                                      & 7 \ar[ru]              &                                      & 3 \ar[ru]              &                                     & 2 \ar[ru]  &   
        }
    }
\]

Hence in this case, we have 
\begin{align*}
    k_1=4,k_2=k_3=k_4=k_5=3, k_6=k_7=2, k_8=k_9=k_{10}=k_{11}=1.
\end{align*}

Similar to this example, we can calculate all the lengths of modules in $\operatorname{ind}\mathcal{B}$ via the AR-quiver, and thus finish the proof of the Lemma.
\end{proof}

By Lemma\ref{k1k2}, we have $(k_1,k_2,\cdots)$ form a partition of $r=|\operatorname{ind}\mathcal{B}|$, and denote its dual partition by $(x_1,\cdots,x_m)$.

\begin{corollary}
    With the notations as above, 
     \begin{align*}
        |G| = \frac{1}{d} q^r (q^{x_1+1}-1) \cdots (q^{x_m+1}-1).
    \end{align*}
\end{corollary}

%



\begin{remark}
    $x_1,\cdots, x_m$ can be given explicitly, see \cite[Prop 10.2.5]{carter} for all the values with respect to the types of root systems.
\end{remark}

\nocite{BGP}
\nocite{Springer}
\nocite{RINGEL1990137}
\nocite{RingelLie}
\nocite{Happelbook}
\nocite{Riedtmann}


\end{document}